\documentclass{amsart}
\newtheorem{theorem}{Theorem}[section]
\newtheorem{lemma}[theorem]{Lemma}
\newtheorem{algorithm}[theorem]{Algorithm}
\newtheorem{proposition}[theorem]{Proposition}

\theoremstyle{definition}
\newtheorem{definition}[theorem]{Definition}

\theoremstyle{remark}
\newtheorem{remark}[theorem]{Remark}
\numberwithin{equation}{section}

\usepackage{amsmath}
\usepackage{amsfonts}
\usepackage{amssymb}
\usepackage{graphicx}
\usepackage{color}
\usepackage{enumerate}
\usepackage{bm}
\usepackage{xcolor}
\usepackage{stmaryrd}
\usepackage{comment}
\newcommand{\be}{\begin{equation}}
\newcommand{\ee}{\end{equation}}
\newcommand{\bea}{\begin{eqnarray}}
\newcommand{\eea}{\end{eqnarray}}
\newcommand{\beas}{\begin{eqnarray*}}
\newcommand{\eeas}{\end{eqnarray*}}

\DeclareMathOperator*{\argmin}{arg\,min}

\usepackage{soul} 
\usepackage[normalem]{ulem} 

\def\supp{\textrm{supp}}
\def\cE{{\mathcal E}}
\def\cA{{\mathcal A}}
\def\cF{{\mathcal F}}
\def\cJ{{\mathcal J}}
\def\cH{{\mathcal H}}
\def\cD{{\mathcal D}}
\def\cR{{\mathcal R}}

\def\cM{{\mathcal M}}
\def\cL{{\mathcal L}}
\def\cU{{\mathcal U}}
\def\E{{\mathbb E}}
\def\C{{\mathbb C}}
\def\N{{\mathbb N}}
\def\P{{\mathbb P}}
\def\R{{\mathbb R}}
\def\Z{{\mathbb Z}}
\def\bc{{\bm x}}
\def\bc{{\bm c}}

\def\bg{{\bm g}}
\def\bz{{\bm z}}
\def\by{{\bm y}}
\def\btg{\bm{\tilde{g}}}
\def\bv{{\bm v}}

\def\bA{{\bm A}}
\def\({\Bigl (}
\def\){\Bigr )}
\def\[{\Bigl [}
\def\]{\Bigr ]}
\def \<{\langle}
\def \>{\rangle}
\graphicspath{{./},{./figures/}}
\begin{document}

\title[High-dimensional compressed sensing on lower sets]{Polynomial approximation via compressed sensing of high-dimensional functions on lower sets}

\author{Abdellah~Chkifa}
\address{Department of Computational and Applied Mathematics, Oak Ridge National Laboratory, 1 Bethel Valley Road, P.O. Box 2008, Oak Ridge TN 37831-6164}
\email{chkifam@ornl.gov}

\author{Nick~Dexter}
\address{Department of Mathematics, University of Tennessee, Knoxville, TN 37996}
\email{ndexter@utk.edu}

\author{Hoang~Tran}
\address{Department of Computational and Applied Mathematics, Oak Ridge National Laboratory, 1 Bethel Valley Road, P.O. Box 2008, Oak Ridge TN 37831-6164}
\email{tranha@ornl.gov}

\author{Clayton G.~Webster}
\address{Department of Computational and Applied Mathematics, Oak Ridge National Laboratory, 1 Bethel Valley Road, P.O. Box 2008, Oak Ridge TN 37831-6164}
\email{webstercg@ornl.gov}

%
\thanks{This material is based upon work supported in part by: the U.S.~Defense Advanced Research Projects Agency, Defense Sciences Office under contract and award numbers HR0011619523 
and 1868-A017-15; the U.S.~Department of Energy, Office of Science, Office of Advanced Scientific Computing Research, Applied Mathematics program under contract number ERKJ259 and; the Laboratory Directed Research and Development program at the Oak Ridge National Laboratory, which is operated by UT-Battelle, LLC., for the U.S.~Department of Energy under Contract DE-AC05-00OR22725.}



\begin{abstract}
This work proposes and analyzes a compressed sensing approach to polynomial 
approximation of complex{-valued} functions in high dimensions. 
Of particular interest is the setting where the target function is smooth, characterized by a rapidly decaying orthonormal expansion, whose most important terms are captured by a lower (or downward closed) set. 
By exploiting this fact, we present an innovative weighted $\ell_1$ minimization procedure with a precise choice of weights, and a {new} iterative hard thresholding method, for imposing the {downward closed} preference.  
Theoretical results reveal that our computational approaches possess a provably reduced sample complexity compared to existing compressed sensing techniques presented in the literature.
In addition, the recovery of the corresponding best approximation using these methods is established through an improved bound for the restricted isometry property.  Our analysis represents an 
extension of the approach for Hadamard matrices in \cite{Bou14} to the general case of continuous bounded orthonormal systems, quantifies the dependence of sample complexity on the successful recovery probability, and provides an {estimate on the number of measurements} with explicit constants.
Numerical examples are provided to support the theoretical results and demonstrate the computational efficiency of the novel weighted $\ell_1$ minimization {strategy}.
\end{abstract}

\maketitle


\section{Introduction}
\label{sec:intro}
{Compressed sensing (CS) is an appealing approach for reconstructing signals from underdetermined systems, with far smaller number of measurements compared to the signal 
length \cite{CRT06,Donoho06}. Under the sparsity or compressibility assumption of 
the signals, this approach enjoys a significant improvement in sample complexity 
in contrast to traditional methods such as discrete least squares, projection, and interpolation. 
As the solutions of many parameterized partial differential equations (PDEs) 
are known to be compressible in the 
sense that they are well approximated by a sparse expansion in an orthonormal 
system {(see, e.g., \cite{CD15} and the references therein)}, it is 
no surprise that the interest in applying compressed sensing techniques to the 
approximation of high-dimensional functions and parameterized systems has 
been growing rapidly in recent years 
\cite{DO11, MG12, YGX12, YK13, RS14, 
PHD14, HD15, PengHampDoos15, NaraZhou15}. }

In these works, the target function is a quantity of interest 
(QoI) associated with the solution of a parameterized PDE of the form 
\be
\label{eq:PDE}
\cD(u,{\bm y})=0,
\ee
where ${\mathcal D}$ is a differential operator and ${\bm y}:=(y_1,\dots,y_d)$ is a parameter 
vector in a compact tensor product domain $\mathcal{U}=\prod_{k=1}^d \mathcal{U}_k\subset\R^d$, e.g.,  
$\mathcal{U}=[-1,1]^d$. The solution $u$ to such PDEs 
is therefore a map 
${\bm y}\in \mathcal{U} \mapsto u({\bm y})\in \mathcal{V}$ where $\mathcal{V}$ is the solution space, typically a Sobolev space, e.g., $\mathcal{V}=H^1_0$. The algorithms proposed in the previously cited works are designed to approximate a QoI consisting of a function $g:{\bm y}\in \mathcal{U}\mapsto {G(u({\bm y}))}$
which, e.g., is either the evaluation of $u$ at a fixed point of the space/time domain or a linear functional in 
$u$. 
{Introducing $\cF:=\mathbb{N}_0^d = \{{\bm \nu}=(\nu_k)_{k=1}^d\, :\, \nu_k\in\mathbb{N}_0 \}$, 
and a measure 
$\varrho:\mathcal{U}\rightarrow\mathbb{R}_+$ with $\varrho({\bm y}) = \prod_{k=1}^d \varrho_k(y_k)$, 
the resulting functions are smooth, complex-valued, and can be expanded in an ${L^2({\mathcal U},d\varrho)}$-orthonormal basis $\{\Psi_{\bm\nu}\}_{{\bm\nu}\in\mathcal{F}}$
according to
\be
\label{expansionPC}
g ({\bm y}) = \sum_{{\bm \nu} \in \mathcal{F}} c_{\bm\nu} \Psi_{\bm\nu} ({\bm y}),
\ee
where $\Psi_{\bm\nu} = \prod_{k=1}^d \Psi_{\nu_k}$ are tensor products of  ${L^2({\mathcal U_k},d\varrho_k)}$-orthonormal polynomials, and the coefficients $c_{\bm\nu}$ belong to $\C$. The series 
\eqref{expansionPC} is generally referred to as the polynomial chaos 
(PC) expansion of $g$ (see, e.g., \cite{GS02,XK02}), whose 
convergence rates are well understood {\cite{RS14}}. The orthonormal systems of particular interest in this work consist of Legendre and Chebyshev expansions. The polynomial approximation 
of the function $g$ in the CS setting is fairly straightforward. 
First, one truncates the expansion \eqref{expansionPC} in the multivariate polynomial space 
\begin{equation}
\label{eq:polyspace}
\mathbb{P}_{\mathcal{J}} := \text{span}\{{\bm y}\mapsto {\bm y}^{\bm \nu}\, :\, {\bm \nu}\in\mathcal{J}\}
\end{equation}
with $\cJ:=\{{\bm\nu}_1,\dots,{\bm\nu}_N\}$ a finite set of indices 
whose cardinality $N:=\#(\cJ)$ is large enough to yield  
$g \simeq \sum_{{\bm \nu}\in\cJ} c_{\bm \nu} \Psi_{\bm \nu} $. 
Then, for some $m\leq N$,  generate 
$m$ samples ${\bm y}_1,\dots,{\bm y}_m$ in the parametric domain $\mathcal{U}$ 
independently from the orthogonalization measure $\varrho$ associated 
with $\{\Psi_{\bm \nu}\}_{{\bm \nu}\in \cF}$, and find an approximation $g^{\#}$ of $g$ of the form
\begin{align}
\label{approx_f}
 g^{\#} = \sum_{{\bm \nu}\in \cJ} c^{\#}_{\bm \nu} \Psi_{\bm \nu} ,
\end{align}
where $\bc^{\#}:=(c^{\#}_{\bm \nu})_{{\bm \nu}\in\cJ}$ is the sparsest 
signal with an inherent interpolatory aspect, i.e., 
{among solutions $\bz$ of underdetermined system $\bm{\Psi} \bz = \bg$}. 
Here, the matrix $\bm{\Psi} \in \R^{m\times N}$ contains the samples of the PC basis and the vector
$\bg$ is the observation of the target function, i.e.,
\be
\label{defPsiFirst}
\bm{\Psi} := \(\Psi_{{\bm \nu}_j}({\bm y}_i)\)_{\substack{1\leq i\leq m\\1\leq j\leq N}},\quad \text{and}\quad
\bg := (g({\bm y}_1),\dots,g({\bm y}_m)).
\ee
respectively.  In practice, {noisy} formulations of this problem are also considered {by investigating the expansion tail $\sum_{\nu\notin \cJ} c_{\bm \nu} \Psi_{\bm \nu}$}.
}

To date, the sparse recovery of the polynomial expansion \eqref{expansionPC} via CS has shown to be very promising. However, this approach requires a {low uniform bound} of the underlying basis, given by
\begin{align*}
\Theta = \sup_{{\bm \nu}\in \cJ} \|{ \Psi}_{\bm \nu}\|_{L^{\infty}({\mathcal U})},
\end{align*}
as the sample complexity $m$ required to recover the best $s$-term approximation (up to a multiplicative constant) scales with the following bound (see, e.g., \cite{FouRau13})
\begin{align}
\label{intro_cond}
m\gtrsim \Theta^2 s \times \text{log factors}.
\end{align}
This poses a challenge for many multivariate polynomial approximation strategies  
as $\Theta$ is prohibitively large in high dimensions. In particular, for $d$-dimensional problems, 
$\Theta = 2^{d/2}$ for Chebyshev systems and $3^{d/2}$ for preconditioned 
Legendre systems \cite{RauWard12}.  Moreover, when using the standard Legendre expansion, the 
number of samples can exceed the cardinality of the polynomial subspace, unless 
the subspace a priori excludes all terms of high total order (see, e.g., \cite{YGX12,HD15}). 
Therefore, the advantages of sparse polynomial recovery methods,  
coming from reduced query complexity,  are eventually overcome by the curse of dimensionality, 
in that such techniques require at least as many samples as traditional sparse interpolation techniques
in high dimensions \cite{Nobile:2008wf,Nobile:2008uc,Gunzburger:2014hi}.

{
Nevertheless, in many engineering and science applications, the target functions, despite being high-dimensional, are smooth and often characterized by a rapidly decaying polynomial expansion, whose most important coefficients are {of} {low order} {\cite{CDS11,HoangSchwab12,CCS14,CD15}.} 
In such situations, the quest for finding the approximation containing the largest $s$ terms
can be restricted to polynomial spaces associated with lower (or downward closed) sets.
\begin{definition}[Lower set]
An index set $\Lambda\subset \cF$ is called a lower set 
(also called \textit{downward closed set}) if and only if
\begin{align}
{\bm \nu} \in \Lambda \quad \text{ and }\quad {\bm \mu} \le {\bm\nu} 
\Longrightarrow {\bm\mu} \in \Lambda,
\end{align}
where ${\bm \mu} \le {\bm \nu}$ if and only if  {$\mu_k \le \nu_k$ for all $1\le k\le d$}. 
\label{defLowerSet}
\end{definition}

{The practicality of downward closed sets is mainly computational, and has been 
demonstrated in different approaches such as quasi-optimal strategies {\cite{BNTT14}}, Taylor expansion \cite{CCDS13}, interpolation methods \cite{CCS13}, 
and discrete least squares \cite{CCH14}.
For instance, in the context of parametric PDEs such as \eqref{eq:PDE}, it was shown in \cite{CCS14} that for a large class of operators $\cD$ with a certain type of anisotropic dependence on ${\bm y}$, the solution map ${\bm y}\mapsto u({\bm y})$ can be approximated by best $s$-term PC expansions associated with {index sets of cardinality $s$}, resulting in algebraic rates $s^{-\alpha},\, \alpha>0$ in the uniform and/or mean average sense. The same rates are preserved with index sets that are lower.  
In addition, for $\mathcal{U} = [-1,1]^d$, such lower sets of cardinality $s$ also enable the equivalence property $\|\cdot\|_{L^2({\mathcal U},d\varrho)}\leq \|\cdot\|_{L^\infty} \leq s^{\gamma}\|\cdot\|_{{L^2({\mathcal U},d\varrho)}}$ in arbitrary dimensions $d$ with, e.g.,  
$\gamma=2$ for the uniform measure and $\gamma=\frac{\log3}{\log2}$ for Chebyshev measure.

This paper is focused on developing and analyzing CS approximations confined to downward closed sets, 
used to overcome the curse of dimensionality in the sampling complexity bound \eqref{intro_cond}. 
As such, our work also provides a fair comparison with existing numerical polynomial approaches 
in high dimensions \cite{CCDS13,CCS13,CCH14,BNTT14,TranWebsterZhang14}}.  To achieve our goal, 
we study two sparse recovery approaches for imposing the {downward closed} structure, namely a 
weighted $\ell_1$-minimization with the specific choice of weights 
{$\omega_{\bm \nu} = \|{ \Psi}_{{\bm \nu}}\|_{L^\infty(\Omega)}$}, 
and an iterative hard thresholding method constrained to lower sets}. 
In addition, we also develop a rigorous theoretical framework that provides the analytic 
evidence for the improved performance of our proposed methods in reconstructing smooth functions.

In the context of CS, it is a well-established fact that sparse recovery is strongly tied to the concept of 
the restricted isometry property (RIP) of the (normalized) sampling matrix {$\bm\Psi$}. 
However, motivated by the fact that the best $s$-term approximation is typically associated with a lower set, 
herein we adapt a weaker version of the RIP in which we call the \textit{lower RIP}.  
Unlike the standard RIP which requires all sub-matrices formed by $s$ columns of $\bm\Psi$ to be well 
conditioned, the lower RIP involves only $s$-tuples of columns whose indices form a lower set.
{Given the lower RIP assumption, we establish stable and robust 
reconstruction guarantees for the \textit{best lower $s$-term approximation} 
of $g$, which is the best among all approximations of $g$ supported on lower 
index sets of cardinality $s$. 
It is reasonable to expect that this approximation, while weaker, is close to best $s$-term approximation for smooth functions $g$ considered throughout this effort.

More importantly, the improved sample complexity for high-dimensional function recovery, using our methods, 
can be deduced directly from the sufficient condition for lower RIP. 
For clarification, a complete technical description of \eqref{intro_cond} is given by the 
condition 
\begin{align}
\label{intro:RIP_cond1}
m\ge C\Theta^2  s \log^3(s) \log(N),
\end{align}
which was developed in \cite{RudelsonVershynin08,Rauhut10,CheGurVel13}, and is often cited 
in the case of the standard RIP, used to guarantee uniform recovery with probability exceeding 
$1-N^{-\log^3(s)}$.   
In this work, we develop 
three critical components that enable us to systematically reduce the number of samples given by \eqref{intro:RIP_cond1}:
\begin{enumerate}
\item[1.]
The lower RIP is associated with downward closed sets which allows us to employ efficient {bounds} of basis functions defined on those sets, derived in \cite{CDL13,CCH14} for discrete least squares, and replace $\Theta^2  s$ by
\begin{align}
\label{intro:Ks}
K(s) = \sup_{\substack{\Lambda\subset \mathcal{J},\, \Lambda\, \text{lower }\\ {\#(\Lambda) =s }} } \Big \| \sum_{{\bm \nu}\in \Lambda }|\Psi_{\bm \nu}|^2\Big \|_{L^{\infty}({\mathcal U})},
\end{align} 
{which is significantly smaller.} 
\item[2.]
We can reasonably choose $\mathcal{J}$ as the Hyperbolic Cross index set
\begin{align}
\label{hyperbolic_cross}
\cH_s := \Big\{{\bm \nu} \in \cF: \prod_{k=1}^d (\nu_k + 1) \le s \Big\},
\end{align}   
which is the smallest set that surely contains the best lower $s$ indices
(i.e., the union of all lower sets of cardinality $s$).
The cardinality of $\mathcal{H}_s$ grows mildly in $s$ and $d$, compared to 
other common choices such as tensor product and total degree. Indeed, 
from \cite{DG15}, we have $N:=\#(\cH_s)\le 2 s^{3} 4^{d}$, which facilitates 
both linear growth of $m$ with respect to the dimension $d$, and accelerates 
matrix-vector multiplication. 
\item[3.]
We extend the chaining arguments, recently developed 
in \cite{Bou14,HR15} for unitary matrices, to general bounded orthonormal 
systems, so as to decrease the logarithm factor in \eqref{intro:RIP_cond1}
by one unit.  Following the approach in \cite{Bou14}, {we modify the covering argument for this task.} 
In addition, we provide the technical details necessary to quantify the universal 
constants, and the constraint of the number of samples $m$ on the success probability.
 It is worth noting that our analysis shows a success probability
slightly weaker than that associated with \eqref{intro:RIP_cond1}. 
\end{enumerate}

By combining all the above ingredients, the analysis herein 
{(see Theorem \ref{note:RIP_theorem} and \ref{theorem:lower-RIP})} 
details the improvements to \eqref{intro:RIP_cond1}, by showing that the sufficient condition required to reconstruct the 
best lower $s$-term approximation,  with probability exceeding $1-N^{-\log(s)}$ is given by
\begin{align}
\label{intro:RIP_cond2}
m\ge C K(s) \log^2(s) (\log(s) + d).
\end{align}
As shown in Lemma \ref{upperbound}, for 
$\mathcal{J} = \mathcal{H}_s$ in {high dimensions, i.e., $2^d > s$,}
\begin{align}
\label{proposal:eq1}
\Theta^2 s \ge\, 
\begin{cases}
  s^2/2, \! &\text{if }({ \Psi}_{{\bm \nu}})\text{ is Chebyshev basis,}  \\
  s^{\frac{\log 3}{\log 2}+1}/3, \!&\text{if }({ \Psi}_{{\bm \nu}})\text{ is Legendre basis,} 
\end{cases}
\end{align}
while as indicated in Lemma \ref{lemmaSharpBoundKs},
\begin{align}
\label{proposal:eq2}
K( s) \le   
\begin{cases}
   s^{\frac{\log 3}{\log 2}}, \! &\text{if }({ \Psi}_{{\bm \nu}})\text{ is Chebyshev basis,}  \\
   s^{2}, \!&\text{if }({ \Psi}_{{\bm \nu}})\text{ is Legendre basis.} 
\end{cases}
\end{align}
Therefore, the advantage of our sample complexity, given by \eqref{intro:RIP_cond2}, compared to 
the well-known condition \eqref{intro:RIP_cond1}, is that our {sufficient} requirements for recovery possess:
lower order of $s$; lower order in the logarithmic factor;  and an efficient and explicit definition of {$\log(N)$}  
given by $\log(s) + d$.
 
\subsection{Related works}
Our lower RIP is a specific case of the weighted RIP introduced in \cite{RW15}, 
several results on which carry over into our context. However, while the analysis 
therein applies for general weights, it only leads to the best weighted $s$-term error, which is incomparable to and, in case of large weights, much weaker than the best $s$-term error, in regards to the number of terms to be recovered. Therefore, the numerical benefit of weights in reducing the computational complexity is inconclusive. 
The idea of using the weights to boost the recovery performance of $\ell_1$ 
minimization has appeared elsewhere, e.g., in the context of regularization 
or removing aliasing \cite{Adcock15,RW15}, as well incorporating a priori information 
related to the support set or the decay of the polynomial coefficients  
\cite{FMSY12,YB13,YK13,PHD14}.  
On the contrary, our approach does not require any such a priori knowledge; 
for an improved recovery {performance}, the generic requirement on the 
target functions is that the multi-indices of best (largest) polynomial coefficients 
are captured in a lower set.

The RIP estimate herein extends the strategy in \cite{Bou14}, 
introduced to improve the standard RIP for Hadamard matrices, 
to the general case of continuous bounded orthonormal systems.  
%
{Upon completion of this work, we became aware of the work \cite{HR15}, in which a different strategy of net constructions was introduced}, leading to a reduction in sample complexity \eqref{intro:RIP_cond1} by one logarithmic factor, as well as an improved dependence on restricted isometry constant.
While \cite{HR15} is only concerned with asymptotic estimates for Fourier matrices, we believe that one might extend such arguments to the setting presented in this effort.
Finally, the compressed sensing approaches presented in this work as well as any RIP-based polynomial approximation framework require an a priori estimate of the expansion tail, whereas, the 
RIPless approach presented in \cite{Adcock15b} refrains from this requirement.

\subsection{Notation and preliminaries} 
\label{sec:prelim}
Throughout this paper, we use $C$ to denote a generic positive constant 
whose value may be different from place to place but which is independent of any 
parameters. 
For $\Lambda\subset \cF$, $\Lambda^c$ denotes 
the complement of $\Lambda$, $\bz_{\Lambda}$ is the restriction of $\bz=(z_{\bm \nu})_{{\bm \nu}\in\cF}$ to 
$\Lambda$. {For convenience, $\|\cdot\|_{L^{\infty}} := \|\cdot\|_{L^{\infty}({\mathcal U})}$.} Given the multi-index notation 
$ {\bm \nu} = ({\nu}_1,\ldots,\nu_d) \in {\cF}$, we define 
$$
\supp({\bm \nu}) := \{k:\, \nu_k \ne 0\},\quad\quad 
\| {\bm\nu}\|_0 := \#(\supp({\bm \nu})).
$$
{For $g = \sum_{{\bm \nu}\in \cF} c_{\bm \nu}\Psi_{\bm \nu}$ and $\Lambda$
a set of indices, we denote
$g_{\Lambda} := \sum_{{\bm \nu}\in\Lambda} c_{\bm \nu} \Psi_{\bm \nu}$.
{We normalize the sampling matrix and observation
in \eqref{defPsiFirst} as }
\begin{align}
\label{intro_system}
\bA := \frac{\bm \Psi}{\sqrt{m}}  = \(\frac{\Psi_{{\bm \nu}_j}({\bm y}_i)}{\sqrt m} \)_{\substack{1\leq i\leq m\\ 1\leq j\leq N}},\qquad 
\bm{\tilde{g}} : = \frac{\bg}{\sqrt{m}} = \(\frac{g({\bm y}_i)}{\sqrt m}\)_{1\leq i\leq m}.
\end{align}
Also, the normalized expansion tail is referred as 
\begin{align}
\label{intro_system2}
\bm{\xi}:= \(\frac{g_{\cJ^c}({\bm y}_i)}{\sqrt m}\)_{1\leq i\leq m} .
\end{align}
Under the newly introduced notation, the exact coefficients $\bc = (c_{\bm \nu})_{{\bm \nu}\in \cJ}$ satisfy $ \bm{A} \bc\, + \bm{\xi}= \btg$. Assuming that $\bm{\xi}$ is small (whose a priori upper bound is {assumed} if $\ell_1$ minimization is used), we approximate $g$ via $g^{\#} = \sum_{{\bm \nu}\in \cJ} c^{\#}_{\bm \nu} \Psi_{\bm \nu}$, where $\bc^{\#}=(c^{\#}_{\bm \nu})_{{\bm \nu}\in\cJ}$ is {among the solutions $\bz$ of}
$\bm{A} \bz \approx \btg .
$
}

\subsection{Organization}
Our paper is organized as follows. 
First, using the recently developed chaining technique, in Section \ref{sec:RIP_estim} a new RIP estimate for general bounded orthonormal systems is provided. To avoid unimportant technicalities, the discussion in this section will focus on standard RIP, however, the analysis is general and does not depend on whether standard RIP or lower RIP is considered.
%
In Section \ref{sec:lower_set}, we describe the new mathematical tools necessary to 
establish the concept of the lower RIP and the sparse recovery on lower sets. 
Section \ref{sec:recovery_alg} is devoted to presenting the innovative theoretical results and analysis 
for polynomial approximation using our versions of weighted $\ell_1$ minimization and iterative 
hard thresholding algorithms. 
Several high-dimensional computational experiments supporting the theory are given 
in Section \ref{sec:experiment}. 
Finally, several critical lemmas and the complete technical details of the rigorous proofs of our RIP estimates can be found in the Appendix.

\section{Improved RIP estimate for bounded orthonormal system}
\label{sec:RIP_estim}
The restricted isometry property (RIP) is an important ingredient for sparse recovery guarantees, which is given by the following definition. 
\begin{definition}[RIP]
For ${\bm A}\in \C^{m\times N}$, the restricted isometry constant $\delta_{s}$ associated to $\bm A$ is the smallest number for which
\begin{align}
\label{standard-RIP}
(1-\delta_{s})\|\bz\|_{2}^2 \le \|\bm{Az}\|_{2}^2 \le (1+ \delta_{s})\|\bz\|_{2}^2,
\end{align} 
for all $\bm{z} \in \C^N$ satisfying $\#(\supp(\bm z)) \le s$. We say that $\bm{A}$ satisfies the restricted isometry property if $\delta_s$ is small for reasonably large $s$. 
\end{definition}
{In this paper, we prove the following RIP estimate for bounded orthonormal system, inspired by the approach in \cite{Bou14}. 

\begin{theorem} 
\label{note:RIP_theorem}
Let $\delta, \gamma$ be fixed parameters with 
$0<\delta < 1/13$, $0 < \gamma < 1$ and $\{\Psi_{\bm \nu}\}_{{\bm \nu}\in \cJ}$ be an orthonormal 
system of finite size $N = \#(\cJ)$. Assume that 
\begin{align}
m \ge 2^6 e~\frac{\Theta^2 s}{\delta^2} \log\( \frac{\Theta^2 {s}}{\delta^2}\)
\max \biggl\{ \frac{2^5}{\delta^4}  \log\( 40 \frac{ \Theta^{2} s}{\delta^2} 
\log\( \frac{\Theta^2 {s}}{\delta^2}\) \) & \log(4N) , \label{num_sample0}
\\
\frac 1 \delta\log& \(\frac{1}{\gamma \delta }  
\log\( \frac{\Theta^2 {s}}{\delta^2}\)  \) \biggl\} , \notag
\end{align}
and $\by_1,\by_2,\ldots,\by_m$ are drawn independently from the orthogonalization 
measure $\varrho$ associated to $\{\Psi_{\bm \nu}\}$. Then with probability exceeding $1-\gamma$, 
the normalized sampling matrix $\bA \in \C^{m\times {N}}$ satisfies
\begin{align}
(1 - 13\delta) \|\bz\|^2_2 < \|\bm{Az}\|_2^2 < (1 + 13\delta)\|\bz\|_2^2, \label{note:problem1}
\end{align}
for all $\bm{z}\in \C^N$, $\#(\supp(\bz)) \le s$. 
\end{theorem}}
{The complete detailed proof of Theorem \ref{note:RIP_theorem} is given in the Appendix \ref{appen:mainproof}. 
However, to assist the reader in better understanding the logic of our proof we next provide a sketch that explains the essential features on how we achieved the improved RIP estimate.
%

\noindent
\textit{Sketch of proof}. To begin, let us denote 
\begin{align*}
\psi({\bm y},\bm{ z})  & := \sum_{{\bm \nu}\in {\mathcal{J}}} {z}_{\bm \nu} \Psi_{\bm \nu}(\bm{y})  ,\ \ \forall{\bm y}\in \mathcal{U},\, \bm{z}\in \mathbb{C}^N,
\\
\text{and  }\mathcal{E}_s   & := \{\bm{z} \in \mathbb{C}^N: \|\bm{z}\|_2 = 1,\, \#(\supp(\bm{z}))\le s\}. 
\end{align*}
Our goal is to derive conditions on $m$ such that for a set of $m$ random samples $\{{\bm y}_i\}_{i=1}^m \subset \cU$, drawn according to $\varrho$, then with high probability, there holds $\forall\bz\in \cE_s$: 
\begin{align}
\label{estim_1}
\frac{1}{m}\sum_{i=1}^m |\psi({\bm y}_i,\bm{ z})|^2 \approx \int_{\cU} |\psi({\bm y},\bm{ z})|^2 d\varrho. 
\end{align}
We construct a ``discrete'' approximation $\tilde{\psi}$ of $\psi$ such that \textit{(see Appendix \ref{appen:mainproof}, Step 1)}: 
\vspace{.1cm}
\begin{enumerate}[\ \ \ \ \ \ 1.]
\item for any $\bz \in  \cE_s$, $\tilde{\psi}(\cdot,\bz) \approx \psi(\cdot,\bz)$;
\item $\tilde{\psi}(\cdot,\bz)$ can be represented as a piecewise constant function on $\cU$: $\tilde{\psi}(\cdot,\bz) = \sum_{l\in \cL} \tilde{\psi}_l^{\bz}(\cdot) $, where each $\tilde{\psi}_l^{\bz}$ is a constant function, supported on a subset of $\cU$, representing a scale of $\tilde{\psi}(\cdot, \bz)$ and $\cL$ is a finite set of scale; and 
\item for each $l\in \cL$, $\{\tilde{\psi}_l^{\bz}: \bz\in \cE_s\}$ belongs to a finite class whose cardinality is optimized. 
\end{enumerate}
With the use of 1. and 2. one can establish the bound \textit{(see Appendix \ref{appen:mainproof}, Step 2):}
\begin{align}
\notag 
\left| \frac{1}{m}\sum_{i=1}^m |\psi({\bm y}_i,\bm{ z})|^2  - \int_{\cU} |\psi({\bm y},\bm{ z})|^2 d\varrho\right| 
&\lesssim \left| \frac{1}{m}\sum_{i=1}^m |\tilde \psi(\by_i,\bz)|^2 -\int_{\cU} |\tilde \psi(\by,\bz)|^2 d\varrho
\right|
\\
& \le \sum_{l \in\cL} \left| \frac{1}{m}\sum_{i=1}^m |\tilde{\psi}^{\bz}_l ({\bm y}_i)|^2 - \int_{\cU} |\tilde{\psi}_{l}^{\bz} ({\bm y}) |^2 d\varrho \right|.
\label{estim_2}
\end{align}
Using the basic tail estimate given by \textit{Lemma \ref{lemma:tailbound}} yields
for any $l$ and $\bz$, with high probability
\begin{align}
\label{estim_3}
\frac{1}{m}\sum_{i=1}^m \tilde{\psi}^{\bz}_l ({\bm y}_i) \approx \int_{\cU}\tilde{\psi}_{l}^{\bz} ({\bm y}) d\varrho. 
\end{align}
We can then obtain \eqref{estim_1} by employing \eqref{estim_2} and applying the union bound for \eqref{estim_3} over all $l,\bz$ \textit{(see Appendix \ref{appen:mainproof}, Step 3)}. For this argument to yield small $m$, it is critical to construct $\tilde{\psi}(\cdot,\bz)$ in such a way that the total number of functions $\tilde{\psi}_l^{\bz}$ (over $l\in \cL,\bz\in \cE_s)$ is finite 
and optimized, justifying the requirement given by 3. above. 

As an example, for each $l\in \cL$, we can define $\tilde{\psi}_l^{\bz}$ according to a covering of $ \cE_s$ under the pseudo-metric 
\begin{align*}
d(\bz,\bz') =  \sup_{\by\in \cU} |\psi(\by,\bz - \bz')|, 
\end{align*}
so that $\#\{\tilde{\psi}_l^{\bz}: \bz\in \cE_s\}$ is roughly a covering number of $\cE_s$.  However, one can check 
that in our high-dimensional setting, this covering number grows exponentially in the dimension of $\cU$.
Fortunately, an inspection of $|\psi(\cdot,\bz-\bz')|$ reveals that these functions often have tall spikes in a small subregion of $\cU$, while are relatively small for the rest of the domain. This motivates us to consider a new ``distance'' between $\bz$ and $\bz'$, which is significantly smaller than $d(\bz,\bz')$, given by an upper bound of $ |\psi(\by,\bz - \bz')|$ for \textit{most} $\by\in \cU$. More rigorously, we define
\begin{align*}
d_{\varsigma}(\bm{z},\bm{z}') := \inf_{\substack { {\widetilde{\cU} \subset \cU}\\ {\varrho(\widetilde{\cU}) = 1- \varsigma} } } \sup_{\by\in \widetilde{\cU}}   |\psi(\by,\bz - \bz')|.
\end{align*}
Although $d_{\varsigma}$ is not a proper pseudo-metric, 
an adaptation of the covering number result can still be derived in this case \textit{(see Lemma \ref{note:lemma2})}. This argument is similar in spirit to \cite[Lemma 3.5]{HR15}. The approximation $\tilde{\psi}$, constructed with $d_{\varsigma}$, may not agree with $\psi$ in a small subdomain of $\cU$, but one can tune $\varsigma$ so as to not affect the estimate \eqref{estim_2}. 

This completes the sketch of the main proof. $\blacksquare$
}
\begin{remark}
In brief, the RIP (and subsequently, best 
$s$-term reconstruction) occurs with probability exceeding 
$1-\gamma$ under the condition
\begin{align}
\label{RIP_cond2}
m \ge  C \Theta^2 s   
\max \{ \log^2 (\Theta^2 s) & \log(N),  
\log(\Theta^2 s) \log (\log(\Theta^2 s)/\gamma ) \} .  
\end{align}
The first constraint in \eqref{RIP_cond2} therefore 
reduces the order of $\log (s)$ in \eqref{intro:RIP_cond1} by 
one unit. The second constraint, on the other hand, has an additional log factor 
compared to the well-known one, i.e., $m \ge  C \Theta^2 s    \log (1/\gamma ) $, 
see \cite{Rauhut10}, after balancing leading to a weaker success probability, 
as discussed in Section \ref{sec:intro}. 
\end{remark}

\section{Sparse recovery on lower sets}
\label{sec:lower_set}
In this section we focus on a smooth $g$, given by \eqref{expansionPC}, 
and exploit the fast decay of its polynomial expansion to further improve \eqref{RIP_cond2}. 
Central to this task is the concept of lower or downward closed sets, given 
Definition \ref{defLowerSet}.  With this in mind, instead of best $s$-term approximations, 
we are interested in \textit{best lower $s$-term approximation} 
of $g$, which is the best among all approximations of $g$ supported on lower sets of 
cardinality $s$. 
More rigorously, let $\Lambda^*$ be a lower subset of $\cF$ which realizes 
the infimum
\begin{align}
\Lambda^* :=\argmin\limits_{\substack{\Lambda \text{ lower} \\  {\#(\Lambda) \leq s }}} \|g -g_\Lambda\|,
\end{align}
where the norm to be specified later.  Here $\|g - g_{\Lambda^*}\|$ is the \textit{best lower $s$-term error}, and 
our goal is to find  approximations of $g$ with error scaling linearly in $\|g - g_{\Lambda^*}\|$.} 
We expect the best lower $s$-term error, while 
generally larger, is close to best $s$-term error in our setting. These quantities are particularly 
identical provided that $g$ is $s$-sparse, $\supp(g)$ lower, represented by finite 
Legendre and Chebyshev expansions. 

To achieve our goal, it is reasonable to consider a relaxed version of RIP that 
specifically involves $s$-tuples of columns associated with lower sets. Given 
a multi-index set $\Lambda\subset \cF$, we introduce the quantity 
\begin{align}
\label{define_K1}
K(\Lambda) := \Big \| \sum_{{\bm \nu}\in \Lambda }|\Psi_{\bm \nu}|^2\Big \|_{L^{\infty}}, 
\end{align}
and, with an abuse of notation, denote  
\begin{align}
 K(s) := \sup_{\substack{\Lambda\subset \cJ,\, \Lambda\, \text{lower }\\ {\#(\Lambda) =s }} } K(\Lambda),
 \label{define_Ks}
\end{align}
which has already been mentioned in \eqref{intro:Ks}. 
We define next the \textit{lower restricted isometry property (lower RIP)}.
This property is exclusive to the present setting and {defined here for submatrices} whose
columns are associated with indices ${\bm \nu}\in\cF$. 

\begin{definition}[Lower RIP]
For ${\bm A}\in \C^{m\times N}$ as in \eqref{intro_system}, the lower 
restricted isometry constant $\delta_{\ell,s}$ associated to $\bm A$ is the 
smallest number for which
\begin{align}
\label{lower-RIP}
(1-\delta_{\ell,s})\|\bz\|_{2}^2 \le \|\bm{Az}\|_{2}^2 \le (1+ \delta_{\ell,s})\|\bz\|_{2}^2 
\end{align} 
for all $\bm{z} \in \C^N$ satisfying $K(\supp(\bm z)) \le K(s)$. We say that $\bm{A}$ satisfies the lower restricted isometry property if $\delta_{\ell,s}$ is small for reasonably large $s$. 
\end{definition}

\begin{remark}
\label{remarkLowerRIPweightedRIP} 
The lower RIP is a specific case of the weighted RIP, introduced 
in \cite{RW15} for general weights, here with the weights 
$\omega_{\bm \nu} = \|\Psi_{\bm \nu}\|_{L^{\infty}}$. 
By introducing the notation 
$\|\bz\|_{0,\omega} =\sum_{{\bm \nu}\in \supp(\bz)} \omega_{\bm \nu}^2$ 
for $\bz\in \C^N$, the weighted RIP constant $\delta_{\omega,s}$ was 
defined as the smallest number $\delta_{\omega,s}$ for which 
\begin{align}
\label{weighted-RIP}
(1-\delta_{\omega,s})\|\bz\|_{2}^2 \le \|\bm{Az}\|_{2}^2 \le (1+ \delta_{\omega,s})\|\bz\|_{2}^2,
\quad \quad \|\bm{z}\|_{0,\omega}   \le s.
\end{align} 
By \eqref{define_K1}, observe that 
$K(\supp(\bm z)) \leq \|\bz\|_{0,\omega}$,
hence given $\bz$ such that $\|\bz\|_{0,\omega}\leq K(s)$, then
$K(\supp(\bm z))\leq K(s)$ so that \eqref{lower-RIP} is satisfied 
showing that $\delta_{\omega,K(s)} \leq \delta_{\ell,s}$. For the 
Chebyshev and Legendre systems, the polynomials $\Psi_{\bm \nu}$ 
all attain their supremums at $(1,\dots,1)$, hence for any 
$\bz\in \C^N$,  then $K(\supp(\bz)) =\|\bz\|_{0,\omega}$, showing that 
\begin{align}
\label{low_vs_weigh}
\delta_{\ell,s} = \delta_{\omega,K(s)}.
\end{align}
Note the change of order in this relation: loosely speaking, the lower RIP of order 
$s$ corresponds to the weighted RIP of order $K(s)$.
\end{remark}

An important subclass of $\bm z$ satisfying \eqref{lower-RIP} is $\bm z\in \C^N$ 
with $\#(\supp(\bz)) \le s$ and $\supp(\bz)$ lower. One may want to consider a 
more natural isometry property which requires \eqref{lower-RIP} for only vectors
$\bz$ in the above class. 
We can 
see from the following analysis that this property is weaker but requires the same 
sampling cost as \eqref{lower-RIP}.
The sample complexity for lower RIP is established in the following theorem. 

{
\begin{theorem} 
\label{theorem:lower-RIP}
Let $\delta, \gamma$ be fixed parameters with $0<\delta < 1/13$, $0 < \gamma < 1$ and $\{\Psi_{\bm \nu}\}_{{\bm \nu}\in \cJ}$ be an orthonormal system of finite size $N = \#(\cJ)$. Assume that 
\begin{align}
m \ge   {2^6 e}~\frac{K(s)}{\delta^2} 
\log\( \frac{K(s)}{\delta^2}\)   \max \biggl\{  \frac{2^5}{\delta^4}  
\log\( 40 \frac{ K(s)}{\delta^2}  \log\( \frac{K(s)}{\delta^2}\) \) & \log(4N) , \label{num_sample}
\\
\frac 1 \delta\log& \(\frac{1}{\gamma \delta }  
  \log\( \frac{K(s)}{\delta^2}\)  \) \biggl\} ,  \notag
\end{align}
and $\by_1,\by_2,\ldots,\by_m$ are drawn independently from the orthogonalization measure 
$\varrho$ associated to $\{\Psi_{\bm \nu}\}$. Then with probability exceeding $1-\gamma$, the 
normalized sampling matrix ${\bm A}\in \C^{m\times {N}}$ satisfies
\begin{align}
(1 - 13\delta)\|\bz\|_2^2 < \|\bm{Az}\|_2^2 < (1 + 13\delta)\|\bz\|_2^2, 
\label{note:lower-RIP}
\end{align}
for all $\bm{z}\in \C^N$, $K(\supp(\bz)) \le K(s)$. 
\end{theorem}}
The proof of Theorem \ref{theorem:lower-RIP} is discussed in the Appendix \ref{app:lowerRIP}.
This proof essentially follows the same path as the proof of Theorem \ref{note:RIP_theorem} with few minor changes.
\begin{remark}
In brief, the random matrix $\bm A$ satisfies the lower RIP of order $s$ and, 
subsequently, guarantees lower reconstruction with probability exceeding 
$1-\gamma$ if the sample size $m$ satisfies
\begin{align}
\label{RIP_cond3}
m \ge  C K(s) \max 
\{  \log^2 ( K(s) ) & \log(N),  
\log(K(s)) \log (\log( K(s) )/\gamma ) \}.
\end{align}
\end{remark}

{Next, we present a theoretical comparison between the complexity bounds required by standard RIP \eqref{RIP_cond2}
 and lower RIP \eqref{RIP_cond3}, showing the computational cost saving with our best lower approximations.} Assuming that no information about the support set or the decay of the polynomial coefficients is a priori known, we reasonably make the choice $\cJ = \cH_s$, which is the smallest set that surely contains the best lower $s$ indices. {For the sake 
of notational clearness, we denote $(L_{\bm \nu})_{{\bm \nu}\in \cF}$ and 
$(T_{\bm \nu})_{{\bm \nu}\in \cF}$ the Legendre and Chebyshev 
basis of $L^2(\mathcal{U},d\varrho)$ {with $\varrho$ being} the uniform or
Chebyshev measure respectively. For such polynomials, we 
have for any ${\bm \nu}\in\cF$
\begin{align}
\label{BoundLinftyTL}
\|T_{\bm \nu}\|_{L^\infty} = 2^{\|{\bm \nu}\|_0},\quad \quad \text{and}\quad
\|L_{\bm \nu}\|_{L^\infty} = \prod_{k=1}^d \sqrt{2\nu_k+1}.
\end{align}
}
First, we have the following sharp estimates.
\begin{lemma}
\label{upperbound}
Let $\cH_s$ be defined as in \eqref{hyperbolic_cross} with $s\leq2^{d+1}$. 
There holds
\begin{align}
s/2   \le \sup_{{\bm \nu}\in \cH_s} \|T_{\bm \nu}\|_{L^\infty}^2    \le s,
\quad\mbox{and}\quad 
s^{\frac {\log3}{\log2}}/3 \le  \sup_{{\bm \nu}\in \cH_s} 
\|L_{\bm \nu}\|_{L^\infty}^2  \le  s^{\frac {\log3}{\log2}}.
\label{comp2}
\end{align}
\end{lemma}
\begin{proof}
For ${\bm \nu}\in\cH_s$, it is easy to see that
$\|T_{\bm \nu}\|_{L^\infty}^2 =2^{\|{\bm \nu}\|_0} \leq \prod_{k=1}^d (\nu_k + 1) \leq s$.
Also, since $b\mapsto \frac {\log(2b+1)}{\log(b+1)}$ is decreasing over 
$[1,+\infty)$, $ (2b+1) \leq (b+1)^{\log3/\log2}$ for any $b\geq1$ which 
implies 
$\|L_{\bm \nu}\|_{L^\infty}^2 = \prod_{k=1}^d (2\nu_k + 1) 
\leq \prod_{k=1}^d (\nu_k + 1)^{\frac {\log3}{\log2}} \leq s^{\frac {\log3}{\log2}}$.

{On the other hand,} since $s\leq2^{d+1}$, then $2^{d'}<s\leq 2^{d'+1}$ for some 
$d'\leq d$, so that the index $\bm{\nu}=e_1+\dots+e_{d'}$ belongs to 
$\cH_s$ and yields $\|T_{\bm \nu}\|_{L^\infty}^2=2^{d'}\geq s/2$ and 
$\|L_{\bm \nu}\|_{L^\infty}^2 = 3^{d'} =(2^{d'})^{\frac {\log3}{\log2}}
\geq (s/2)^{\frac {\log3}{\log2}}=s^{\frac {\log3}{\log2}}/3$.
\end{proof}


For the polynomial systems such as Chebyshev or Legendre
(or more generally Jacobi systems), 
{$\log(\Theta) \lesssim \log(s)$ 
over the hyperbolic cross $\cH_s$ regardless of the dimension $d$.} 
An immediate consequence of the previous lemma and condition \eqref{RIP_cond2} is that the RIP can be obtained 
for $\cJ =\cH_s$ with the bound 
\begin{align}
\label{sec2:bound}
m\ge C~s^{1+\beta} \max\left\{ \log^2(s) \log(N) ,\, \log( s) \log (\log( s)/\gamma )\right\},
\end{align}
where $\beta =1$ for Chebyshev systems and 
$\beta = \frac{\log3}{\log2}\simeq 1.58$  for Legendre systems
respectively. 
Following from the estimate (see \cite{DG15})
\begin{align}
\#(\cH_s) \leq 
\varepsilon^{-1} s^{1 + 1/\varepsilon} (1-\varepsilon)^{-d/\varepsilon},\qquad  0<\varepsilon<1,
\label{HC_size}
\end{align} 
it is easy to see that $N=\#(\cH_s) \leq 2 s^{3} 4^{d}$, and if we set $\varepsilon=1/2$ we obtain 
\begin{align}
m\ge C~s^{1+\beta} \max\left\{ \log^2(s)(\log(s)+d) ,\, \log( s) \log (\log( s)/\gamma )\right\}.
\label{RIPHyperCross}
\end{align}
{Although we have eliminated the exponential growth on $d$, the condition \eqref{intro_cond} has not been broken up to this step. Rather, the bound \eqref{RIPHyperCross} is merely acquired from \eqref{intro_cond} with an estimate of $\Theta$ on the 
Hyperbolic Cross subspace}.

{
We proceed to detail the complexity bound of lower RIP. As the Chebyshev and Legendre polynomials attain their supremum at the
point ${\bf 1} = (1,\ldots,1)$ with the supremums given in \eqref{BoundLinftyTL}, the value of $K$ defined in \eqref{define_K1}-\eqref{define_Ks} is then known and one can 
derive estimates for it. For these two systems, we use 
the notations $K_T(\Lambda)$, $K_T(s)$, $K_L(\Lambda)$ and $K_L(s)$ 
respectively, where  
\be
K_T(\Lambda) = \sum_{{\bm \nu}\in\Lambda} 2^{\|{\bm \nu}\|_0},\quad\text{and }\quad
K_L(\Lambda) = \sum_{{\bm \nu}\in\Lambda} \prod_{k=1}^d (2\nu_k+1).
\ee
{The following estimates of $K_T(\Lambda) $ and $K_T(\Lambda)$ can be found in \cite{CCH14}.}
\begin{lemma}
\label{lem:old}
Let $\Lambda\subset \cF$ be a lower set with $\#(\Lambda)\geq2$. There holds 
\begin{align}
\label{boundKT}
2\#(\Lambda)-1 &\leq K_T(\Lambda) \leq 
(\#(\Lambda))^{\frac{\log 3}{\log 2}},\\
3\#(\Lambda)-1 &\leq K_L(\Lambda) \leq 
(\#(\Lambda))^{2}.
\label{boundKL}
\end{align}
\end{lemma}
We note that the left sides in \eqref{boundKT} and 
\eqref{boundKL} follow from $\prod_{k=1}^d (2\nu_k+1)\geq 3$ 
and  $2^{\|{\bm \nu}\|_0} \geq 2$ for any index ${\bm \nu}\neq0$.
We also note that the right side inequalities are sharp, {equalities 
hold for lower sets of the form $\{{\bm \mu}\leq e_1+\dots+e_{d'}\}$}. An immediate 
implication of the Lemma \ref{lem:old} are the bounds 
$2s-1\leq K_T(s) \leq s^{\frac {\log3}{\log2}}$ and $3s-1\leq K_L(s) \leq s^2$. 
The upper bounds are actually sharp in high dimension. We indeed have
\begin{lemma}
Let $s\leq2^{d+1}$. There holds 
\begin{align}
\frac{s^{\frac{\log 3}{\log 2}}}{3} 
\leq K_T(s) \leq 
s^{\frac{\log 3}{\log 2}},\quad\quad\quad
\frac{s^2}4 \leq K_L(s) \leq s^2.
\label{K-estim}
\end{align}
\label{lemmaSharpBoundKs}
\end{lemma}
\begin{proof}
For ${\bm \nu}\in\cH_s$, the rectangular block 
$\cR_{\bm \nu}:=\{{\bm \mu}\leq{\bm \nu}\} \subset\cH_s$ is lower and has a 
tensor format, so that {from identities $1+\sum_{b'=1}^b 2 = 1+2b$ and
$\sum_{b'=0}^b (2b'+1) = (1+b)^2$}, one infers 
\begin{align*}
K_T(\cR_{\bm \nu})=\Pi_{k=1}^d(1+2\nu_k),
\quad\mbox{and}\quad
K_L(\cR_{\bm \nu})=\Pi_{k=1}^d(1+\nu_k)^2.
\end{align*}
{Since $s\leq2^{d+1}$, then $2^{d'}<s\leq 2^{d'+1}$ for some 
$d'\leq d$. For ${\bm \nu}=e_1+\dots+e_{d'}\in \cH_{s}$}, one obtains 
$K_T(\cR_{\bm \nu}) = 3^{d'}=(2^{d'})^{\frac{\log3}{\log2}}$ and $K_L(\cR_{\bm \nu}) = (2^{d'})^2$, {which implies}
$$
K_T(\cR_{\bm \nu}) \geq (s/2)^{\frac{\log3}{\log2}} 
= s^{\frac{\log 3}{\log 2}}/3,
\quad\quad
K_L(\cR_{\bm \nu}) \geq (s/2)^2=s^2/4,
$$
which completes the proof.
\end{proof}
}
Combining the complexity bound \eqref{RIP_cond3} with the 
estimate \eqref{K-estim}, we arrive at the following condition 
for uniform recovery of best lower $s$-term approximations, with 
probability exceeding $1-\gamma$,
\begin{align}
\label{RIP_cond4}
m \ge  C s^{1+\beta'}   \max \{  \log^2 (s) &
(\log(s) + d), \, \log(s) \log (\log(s)/\gamma) \},  
\end{align}
where $\beta' = {\frac{\log 3}{\log 2}}-1\simeq0.58$ for Chebyshev 
system and $\beta' = 1$ for Legendre system. These conditions 
eliminate the dependence on $\Theta^2 s$ at the cost of a super-linear 
growth on $s$, {yet} are clearly weaker than those required by standard RIP, 
see \eqref{RIPHyperCross}. 


{We close this section by pointing out  that \eqref{RIP_cond4} still depends linearly on $d$. 
This dependence can be fully eliminated if instead of $\cJ=\cH_s$, we work 
with $\cJ= \cA_s$, the union of all anchored sets $\Lambda$ of cardinality smaller 
than $s$. Such sets are described in {\cite{CMN15}} and are 
{characterized} by $\Lambda$ lower and $e_k\in\Lambda$ if and
only if $e_{k'}\in\Lambda$ for any $k'=1,\dots,k-1$.
Indeed, it is easy to see that $\cA_s$ is included in {the projection of $ \cH_s$ into an $s$-dimensional space}, hence $d$ in \eqref{RIP_cond4} can be replaced by $s$.
Such subclass of lower sets is also relevant in polynomial approximation of 
parametric PDEs (see, e.g., \cite{CD15}). It should also be emphasized that other types of polynomial spaces, e.g., Total Degree, have been attempted to overcome the fast growth of 
query complexity in high-dimensional problems (see, e.g., \cite{YGX12}). 
However, these approaches impose an {\em a priori} choice of the polynomial subspace and, additionally, employ 
the standard RIP, given by \eqref{intro:RIP_cond1}.  On the contrary, in our work $\cJ$ is determined optimally,
based on the number of terms to be reconstructed, and our lower RIP requires less samples than standard RIP, as discussed throughout.}

\section{Basis pursuit and thresholding algorithms for polynomial approximation on lower sets}
\label{sec:recovery_alg}
In this section, we study two different approaches that enable us to realize sparse 
reconstruction under the lower RIP. The algorithms considered herein include a weighted 
$\ell_1$-minimization, with a precise choice of weights, and a new iterative 
hard thresholding method. To begin, let $\bm {\omega} = (\omega_{\bm \nu})_{{\bm \nu}\in\cF} \in {(0,\infty)^{\cF}}$ 
be a sequence of weights.
Given a vector $\bm{z} = (z_{\bm \nu})_{{\bm \nu} \in \cF}$ of complex {components}
or a function $g = \sum_{{\bm \nu} \in \cF} z_{\bm \nu} \Psi_{\bm \nu}$, we define the 
weighted $\ell_1$-norm of $\bz$ and $g$ by 
$$
\|g\|_{\omega,1} = \|\bz\|_{\omega,1} := \sum_{{\bm \nu} \in \cF} \omega_{\bm \nu} |z_{\bm \nu}|,
$$
and the \textit{best lower $s$-term} error in weighted $\ell_1$ norm by
$$ 
\sigma^{(\ell)}_s(g)_{\omega,1} := 
\inf_{\substack{\Lambda \text{ lower} \\  {\#(\Lambda) \le s}}} 
\|g-g_\Lambda\|_{\omega,1}.
$$
{Recall that we are working with $\cJ = \cH_{s}$, unless otherwise stated.}

\subsection{Weighted $\ell_1$ minimization}
{Assuming that an estimate $\eta$ of the tail $g_{\cH_{s}^c}$ is available (specified later), our weighted $\ell_1$-minimization procedure for 
recovering an approximation $g^{\#}$ of $g$, defined by 
\begin{equation}
\label{eq:gsharp}
g^{\#} =\sum_{{\bm\nu} \in \cH_{s}} c^{\#}_{\bm \nu} \Psi_{\bm \nu},
\end{equation}
is given in the following: 
\textit{
Given $\omega_{\bm \nu} = \|\Psi_{\bm \nu}\|_{L^\infty}$. Find $g^{\#}$ from \eqref{eq:gsharp}, where 
$\bc^\# = (c_{\bm \nu}^\#)_{{\bm\nu} \in \cH_{s}}$ solves the following constrained optimization problem
\begin{align}
\bc^\# = \argmin_{\bz\in\C^N} \|\bz\|_{\omega,1}
\quad  \mbox{subject to }\quad 
\|\tilde\bg - \bA \bz\|_2  \leq \frac{\eta}{\sqrt{m}}.  
\label{form:weighted_ell1}
\end{align}
 }}
The recovery guarantees using weighted $\ell_1$ minimization have been analyzed 
in \cite{RW15} for general choice of $(\omega_{\bm \nu})_{{\bm \nu}\in \cJ}$. As discussed in 
Section \ref{sec:intro}, the benefit of weighting in terms of query complexity is 
inconclusive therein. In this work, we specifically define the weights $\omega_{\bm \nu} = \|\Psi_{\bm \nu}\|_{L^\infty}$ 
for use with $\ell_1$-minimization, and how that smooth functions can be reconstructed with a significantly reduced 
number of samples compared to the unweighted method {(thanks to lower RIP)}. 
Such choice of weight has not been previously discussed in the literature 
although the condition $\omega_{\bm \nu} \ge \|\Psi_{\bm \nu}\|_{{L^\infty}}$ 
has been imposed elsewhere \cite{RW15,Adcock15}. 


In what follows we consider and analyze two scenarios using our weighted $\ell_1$ minimization procedure 
\eqref{eq:gsharp}-\eqref{form:weighted_ell1}.  {First,} we assume only an upper estimate of the tail 
$g_{\cH_{s}^c}$ is available, and second, exact knowledge of the tail is assumed.

\subsubsection{Given an upper estimate of the tail}
In this case, the proof of the recovery guarantee of $g^{\#}$ in \eqref{eq:gsharp},
using the optimization procedure \eqref{form:weighted_ell1}, 
follows the arguments of general weighted $\ell_1$-minimization analysis (see \cite[Theorem 6.1]{RW15}), 
and will not be repeated here. 
However, we remove the condition of large $s$ required in the work \cite{RW15}, 
{based on an improvement of the null space property specific to our setting 
(see Proposition \ref{prop:nullspace}).  We first need some intermediate estimates.
\begin{lemma}
For any $d\geq1$ and any $s\geq2$ 
\be
K_T(2s) \geq 2 K_T(s),\quad\quad\text{and }\quad
K_L(2s) \geq 4 K_L(s). 
\label{K2sKs}
\ee
In addition, 
\begin{align}
K_T(s) \geq \frac {~3~}2~ \max_{{\bm \nu}\in\cH_s}  \|T_{\bm \nu}\|^2_{L^\infty},\quad\quad\text{and }\quad
K_L(s) \geq \frac {~4~}3~ \max_{{\bm \nu}\in\cH_s}  \|L_{\bm \nu}\|^2_{L^\infty}.
\label{lemma3KsK2s}
\end{align}
\label{LemmaGrowthKs}
\end{lemma}
\begin{proof}
For ${\bm \nu}=(\nu_k)_{1\leq k\leq d}$, we use the notation
$\hat {\bm \nu}=(\nu_k)_{2\leq k\leq d}$. Let $\Lambda$ be a lower set 
of cardinality $s$. We introduce 
\begin{align*}
\Lambda' := \{(2\nu_1,\hat{\bm \nu}),(2\nu_1+1,\hat{\bm \nu}): {\bm \nu} =(\nu_1,\hat {\bm \nu})\in\Lambda\}.
\end{align*}
It is easily checked that $\Lambda' \subset \cF$ is lower and $\#(\Lambda') =2\#(\Lambda)=2s$.
Therefore $\Lambda' \subset \cH_{2s}$.
Moreover, for the tensorized 
Chebyshev and Legendre systems, 
$\omega_{\bm \nu} =\prod_{k=1}^d \omega_{\nu_k}$ where {$\omega_{\nu_{k}}$} 
denote the sup norm in one dimension. Hence
\begin{align}
\label{newEst1}
K(2s)\geq K(\Lambda') = \sum_{{\bm \mu}\in\Lambda'} \omega_{\bm \mu}^2=
\sum_{{\bm \nu}\in\Lambda} (\omega_{2\nu_1}^2+\omega_{2\nu_1+1}^2) 
\omega_{\hat {\bm \nu}}^2
\geq \sum_{{\bm \nu}\in\Lambda} 2 \omega_{\nu_1}^2 \omega_{\hat {\bm \nu}}^2
=2K(\Lambda),
\end{align}
where we have used the increase of the weights which yields 
$\omega_{2\nu_1},\omega_{2\nu_1+1}\geq \omega_{\nu_1}$. For Legendre system,
we have 
{
\begin{align}
\label{newEst2}
(\omega_{2\nu_1}^2+\omega_{2\nu_1+1}^2) = (4\nu_1+1+4\nu_1+3)=4(2\nu_1+1) = 4 \omega_{\nu_1}^2 .
\end{align}}
Since $\Lambda$ is an arbitrary lower set included in $\cH_s$, {\eqref{newEst1} and \eqref{newEst2} imply} \eqref{K2sKs}. 

Now let ${\bm \nu}\in\cH_s$ {be} the index that 
maximizes $\|T_{\bm \nu}\|_{L^\infty}$ over $\cH_s$. We have $\cR_{\bm \nu} \subset \cH_s$, so that 
$K_T(s)\geq K_T(\cR_{\bm \nu})=\prod_{k=1}^d (1+2\nu_k)\geq 3^{\|{\bm \nu}\|_0}
= (3/2)^{\|{\bm \nu}\|_0} \|T_{\bm \nu}\|_{L^\infty}^2$. {Similarly, for} ${\bm \nu}\in\cH_s$ maximizing 
$\|L_{\bm \nu}\|_{L^\infty}$, $K_L(s)\geq K_L(\cR_{\bm \nu})=\prod_{k=1}^d (1+\nu_k)^2\geq 
(4/3)^{\|{\bm \nu}\|_0}\prod_{k=1}^d(1+2\nu_k)$,
where we have used $(1+t)^2 \geq \frac 43(1+2t)$ for any 
$t\geq1$. Since for $s\geq2$,  we get that ${\bm \nu}\neq0$ 
and $\|{\bm \nu}\|_0 \geq1$. The proof is complete.
\end{proof}
\begin{remark}
In the previous proof, if we define 
$\Lambda'$ by copying $\Lambda$ using $3\nu_1$, $3\nu_1+1$, $3\nu_1+2$,
we get $K_T(3s)\geq 3 K_T(s)$. {For convenience, in the next proposition, we will employ the estimates:} for any $d\geq1$ and 
any $s\geq2$, 
\be
K_T(3s) \geq 3 K_T(s),\quad\quad\text{and }\quad 
K_L(2s) \geq 3 K_L(s), 
\label{K2sKsRemark}
\ee
{where the second one is slightly weaker than \eqref{K2sKs}.}

\end{remark}

We now are able to provide the null space property associated 
with Chebyshev and Legendre systems, defined on the Hyperbolic Cross index set.
\begin{proposition}
\label{prop:nullspace}
Let $s\ge 2$, $\cJ=\cH_s$ and 
$\bm{A} \in \C^{m\times N}$ be a normalized sampling matrix satisfying the lower RIP 
\eqref{lower-RIP} with 
\begin{align}
\delta_{\ell,\alpha s } < 1/3, 
\label{cond_small_delta}
\end{align}
where $\alpha = 2$ for Legendre system 
and $\alpha = 3$ for Chebyshev system. 
Then, for any $\Lambda \subset \cH_s$ with 
$K(\Lambda) \leq K(s)$ and any $\bz\in\C^N$,
\begin{align}
\label{nullspace}
\|\bz_{\Lambda}\|_2 \le \frac{\rho}{\sqrt{K(s)}} \|\bz_{\Lambda^c}\|_{\omega,1} + \tau\|{\bm{ A z}}\|_2
\end{align}
with $\rho = \frac{4\delta}{1-\delta}$, $\tau = \frac {\sqrt{1+\delta}}{1-\delta}$ 
and $\delta = \delta_{\ell,\alpha s}$. 
\end{proposition}
\begin{proof}
We have $K(s) \geq K(\Lambda)= \sum_{{\bm \nu}\in\Lambda} \omega_{\bm \nu}^2 $, 
therefore proving \eqref{nullspace} is equivalent to showing that $\bA$ satisfies 
the weighted robust null space property of order $K(s)$ with constants $\rho$ 
and $\tau$, see \cite[Definition 4.1]{RW15}. In view of \cite[Theorem 4.5]{RW15}
and by an inspection of its proof, this can follow with $\rho$ and $\tau$ as in {our} 
proposition if $\bA$ {satisfies} weighted RIP with $\delta_{\omega,3K(s)}<1/3$ for 
$K(s)>(4/3)\sup_{{\bm \nu}\in\cH_s} \omega_{\bm \nu}^2$. 
In view of \eqref{K2sKsRemark}, $3K(s)\leq K(\alpha s)$ 
for both Legendre and Chebyshev systems. Since $\delta_{\omega,t}$ is increasing 
in $t$, $\delta_{\omega,3K(s)} \leq \delta_{\omega,K(\alpha s)} = \delta_{\ell,\alpha s}<1/3$, see 
Remark \ref{remarkLowerRIPweightedRIP} for the equality. We also have from 
\eqref{lemma3KsK2s} that $K(s)>(4/3) \sup_{{\bm \nu}\in\cH_s} \omega_{\bm \nu}^2$ for any 
$s>2$ for both systems, which completes the proof.
\end{proof}

Combining \eqref{RIP_cond4} and Proposition \ref{prop:nullspace} 
yields the uniform recovery of $g$ up to the best lower $s$-term error and the tail bound.
\begin{theorem}
\label{theorem:lower_ell1}
Let $s\ge 2$, $\cJ=\cH_s$ and $N=\#(\cH_s)$. Consider a number of samples 
\begin{align}
\label{conditionMLower}
m \ge  C s^{1+\beta'}   \max \{  \log^2 ( s ) 
(\log(s) + d), \, \log(s) \log (\log(s)/\gamma) \} ,  
\end{align}
where $\beta' = \frac{\log 3}{\log 2}-1$ if $\{\Psi_{\bm \nu}\}$ is a Chebyshev system and 
$\beta' = 1$ if $\{\Psi_{\bm \nu}\}$ is a Legendre system. Let ${\bm y}_1,\by_2, \dots ,{\bm y}_m$ 
be drawn independently from the orthogonalization measure $\varrho$ associated 
to $\{\Psi_{\bm \nu}\}$ and $\bA \in \C^{m \times N}$ be the associated normalized 
sampling matrix as in \eqref{intro_system}. 
Then, with probability exceeding $1-\gamma$, the following holds for all functions 
$g = \sum_{{\bm \nu} \in \cF} c_{\bm \nu} \Psi_{\bm \nu}$: {Given $\tilde \bg$, $\bm{\xi}$ as in \eqref{intro_system}-\eqref{intro_system2} and $\eta$ satisfying $\|\bm \xi\|_2 \le \frac{\eta}{\sqrt{m}}$}, 
the function $g^\# = \sum_{{\bm \nu} \in \cH_{s}} c_{\bm \nu}^\# \Psi_{\bm \nu}$, with $\bc^\# = (c_{\bm \nu}^\#)_{{\bm \nu} \in \cH_{s}}$ 
solving \eqref{form:weighted_ell1}, satisfies
\begin{align*}
\| g - g^\#\|_{\infty}   \le \|g - g^\#\|_{\omega,1} & \le c_1 \sigma_s^{(\ell)}(g)_{\omega,1} + d_1  \eta \sqrt{\frac{K(s)}{{m}}},
\\
\| g - g^\#\|_{2}  &  \le c_2 \frac{\sigma_s^{(\ell)}(g)_{\omega,1}}{\sqrt{K(s)}} + d_2  \eta \sqrt{\frac{1}{{m}}}.
\end{align*}
Above, $c_1,c_2,d_1$ and $d_2$ are universal constants. 
\end{theorem}

\begin{remark}
In Theorem \ref{theorem:lower_ell1}, $\bm{\xi}$ is actually a random variable varying with the sampling points $({\bm y}_i)_{1\le i\le m}$. It can be shown however that $\|\bm \xi\|_2 \le {\|g_{\cH_{s}^c}\|_{\omega,1}}$ for every set of samples, thus $\eta$ can be set deterministically as $ \sqrt{m} \|g_{\cH_{s}^c}\|_{\omega,1} \le \eta$.
\end{remark}

\subsubsection{Given an exact estimate of the tail}

{In this section, we seek to prove 
a stronger error rate which is independent of tail bound.} It should be mentioned that a 
result of this type has been derived in \cite{RW15}, where the index set is conditioned 
by the weight as 
\begin{align}
\label{defineJ}
\cJ = \{{\bm\nu} \in\mathbb{N}_0^d: {\omega}_{\bm \nu}^2 \le C K(s)\}, 
\end{align}
and is used to the best weighted $K(s)$-term reconstruction. 
{Note that this is comparable to best lower 
$s$-term in our setting, as a lower set of cardinality $s$ has a weighted cardinality 
approximately equal to $K(s)$ (see Remark \ref{remarkLowerRIPweightedRIP}}). 
In our work, however, $\cJ$ is specified 
instead to be the smallest set that contains the supports of all downward closed sets of cardinality $s$, {i.e., $\cJ=\cH_{s}$}. As shown in Lemma \ref{LemmaGrowthKs}, $\omega_{\bm \nu}^2 \le \frac{~3~}4 K(s)$ for all ${\bm \nu}\in \cH_{s}$. The converse, i.e., 
$\omega_{\bm \nu}^2 \ge CK(s)$ for all $\bm{\nu}\notin \cH_{s}$, however, does not hold. {Indeed}, for our Chebyshev weights, definition \eqref{defineJ} would lead to a $\cJ$ with 
infinite cardinality. 
Therefore, {$\cH_{s}$
represents a significantly smaller index set than those considered in \cite{RW15}. Nonetheless, the recovery of $g$ up to the best lower $s$-term error is still available without condition \eqref{defineJ}, provided that 
$c_{\bm \nu}/\omega_{\bm \nu}$ is small in $\cH_s^c$.} 
{To clarify this assumption, for $g = \sum_{{\bm \nu} \in \cF} c_{\bm \nu} \Psi_{\bm \nu}$, we introduce a parameter $\lambda \ge 0$ such that}

\begin{align}
\max_{{\bm \nu}\in \cH_s^c} \frac{|c_{\bm \nu}|}{\omega_{\bm \nu}} 
\le (1 + \lambda) \min_{{\bm \nu}\in \tilde{\cJ}} \frac{|c_{\bm \nu}|}{\omega_{\bm \nu}}
\label{assump:smooth}
\end{align}
for some set  ${\tilde\cJ} \subset \cH_s$ with $K({\tilde\cJ}) \ge 2K(s)$, or equivalently,
\begin{align}
\label{assump:smooth2}
\lambda := \min_{\substack{ {\tilde\cJ} \subset \cH_s \\ K({\tilde\cJ}) \ge 2K(s) }} \left(\frac{\max\limits_{{\bm \nu}\in \cH_s^c} {|c_{\bm \nu}|}/{\omega_{\bm \nu}}}{  \min\limits_{{\bm \nu}\in \tilde{\cJ}} {|c_{\bm \nu}|}/{\omega_{\bm \nu}}}\right) -1. 
\end{align} 
Many subsets $\tilde{\cJ}$ of $\cH_s$ with cardinality $2s$ satisfy $K(\tilde{\cJ}) \ge 2K(s)$, see \eqref{K2sKs}. Thus, the minimum in the right hand side of \eqref{assump:smooth} can be taken over only $2s$ multi-indices in $\cH_s$. 
On the other hand, for function $g$ whose expansion coefficients decay fast, it 
is reasonable to assume small $\max\limits_{{\bm \nu}\in \cH_s^c} \frac{|c_{\bm \nu}|}{\omega_{\bm \nu}}$, due to small $c_{\bm \nu}$ and possibly big $\omega_{\bm \nu}$ for ${\bm \nu} \in \cH_s^c$. As a result, $\lambda$ is expected to be small in this effort. 

{We first need an estimate of 
$\|\bc_{\cH_s^c}\|_2$. The choice $\cJ = \cH_s$ is not essential in this development, for which reason we state the result for general $\|\bc_{\cJ^c}\|_2$. }

\begin{lemma}
\label{lemma:est_suppl}
Let $\bc = (c_{\bm \nu})_{{\bm \nu}\in \cF}$ and $\cJ$ be a subset of 
$\cF$. For all $s\ge 1$, there 
holds
\begin{align}
\|\bc_{\cJ^c}\|_2 \le \frac{\|\bc_{\cJ^c}\|_{\omega,1}}{\sqrt{K(s)}} + 
\sqrt{K(s)} \max_{{\bm \nu}\in \cJ^c} \frac{|c_{\bm \nu}|}{\omega_{\bm \nu}}.
\end{align}
\end{lemma}
\begin{proof}
{
We introduce $\cJ' = \{{\bm \nu}\in \cJ^c: \omega^2_{\bm \nu} \ge K(s) \}$ 
and $\cJ'' = \cJ^c \setminus \cJ'$. By definition of $\cJ'$, we have 
$$
\|\bc_{\cJ'}\|_2 
\leq \sqrt{\sum_{{\bm \nu}\in\cJ'} \frac{\omega_{\bm \nu}^2 c_{\bm \nu}^2}{K(s)}} 
\leq \frac{\|\bc_{\cJ'}\|_{\omega,1}}{\sqrt{K(s)}}.
$$ 
Since $\|\bc_{\cJ^c}\|_2\leq \|\bc_{\cJ'}\|_2+\|\bc_{\cJ''}\|_2$
and $\|\bc_{\cJ^c}\|_{\omega,1}= \|\bc_{\cJ'}\|_{\omega,1}+\|\bc_{\cJ''}\|_{\omega,1}$,
one only needs to show that 
$$
\|\bc_{\cJ''}\|_2 \le \frac{\|\bc_{\cJ''}\|_{\omega,1}}{\sqrt{K(s)}}+\sqrt{K(s)} 
\max_{{\bm \nu}\in \cJ''} \frac{|c_{\bm \nu}|}{\omega_{\bm \nu}}.
$$
We order $\cJ''$ according to a 
non-increasing order of $(|c_{\bm \nu}|/\omega_{\bm \nu})_{{\bm \nu}\in \cJ''}$ 
and then partition $\cJ''$ as $\cJ''=\cup_{l=1}^L \cJ_l$ 
where we inductively choose the sets $\cJ_l$ according to: 
$\cJ_0=\emptyset$; for $l\geq1$ and $\cJ_0,\dots,\cJ_{l-1}$ 
having been built, we set $\cM_l= \cJ''\setminus \{\cJ_0\cup\dots\cup\cJ_{l-1}\} $, 
and let $\cJ_l= \cM_l$ if $K(\cM_l) < K(s)$, or be {a subset containing largest elements of} $\cM_l$ such that $K(s) \leq K(\cJ_l) \leq 4K(s)$.
If the induction does not terminate, $L=\infty$ in which 
case
$$
K(s) \le  K(\cJ_l) \le 4 K(s), \quad\quad \quad \forall \,l\geq1.
$$
If the induction terminates, $L<\infty$ and we have only $K(\cJ_L)<K(s)$.

Now, we denote by $r_l^+$ and $r_l^-$ the largest and smallest 
 entries in $(|c_{\bm \nu}|/\omega_{\bm \nu})_{{\bm \nu}\in \cJ_l}$. An easy 
 extension of \cite[Lemma 6.14]{FouRau13} yields for all $1\leq l<L$:
\begin{align*}
\|\bc_{\cJ_l}\|_2 \le \frac{\|\bc_{\cJ_l}\|_{\omega,1}}{\sqrt{K(\cJ_l)}} + 
\frac{\sqrt{K(\cJ_l)}}4 (r_l^+ - r_l^-) \le \frac{\|\bc_{\cJ_l}\|_{\omega,1}}{\sqrt{K(s)}} 
+ \sqrt{K(s)}(r_l^+ - r_l^-).
\end{align*}
In the case $L<\infty$, we have 
$ \|\bc_{\cJ_L}\|_2 \leq r_{L}^{+} \sqrt{K(\cJ_L)} \leq r_{L}^{+} \sqrt{K(s)}$. {There follows in both cases $L<\infty$ and $L=\infty$:}
$$
\|\bc_{\cJ''}\|_2 
\leq \sum_{l=1}^L \|\bc_{\cJ_l}\|_2
\leq \frac{\|\bc_{\cJ''}\|_{\omega,1}}{\sqrt{K(s)}} +\sqrt{K(s)} r_1^+.
$$
Since $r_1^+ \leq \max\limits_{{\bm \nu}\in \cJ''} \frac{|c_{\bm \nu}|}{\omega_{\bm \nu}}$ the 
proof is complete.
}
\end{proof}

{We are now ready to state and prove the recovery guarantee, 
assuming an exact estimate of $g_{\cH_{s}^c}$ exists.}
\begin{theorem}
\label{theorem:interp}
Let $s\ge 2$, $\cJ=\cH_s$, $N=\#(\cH_s)$, and $m$ as in \eqref{conditionMLower}. 
Consider a function $g = \sum_{{\bm \nu} \in \cF} c_{\bm \nu} \Psi_{\bm \nu}$ 
with $\|g\|_{\omega,1} < \infty$, and $\lambda$ defined as in \eqref{assump:smooth2}. 
We introduce 
$E_{g}:=\max\{\sqrt2 \|g_{\cH_{s}^c}\|_2,\frac{ \sqrt2\|g_{\cH_{s}^c}\|_{\omega,1}}{\sqrt{K(s)}}\}$
and let $\eta$ be such 
that
\begin{align}
\label{cond_noise}
E_{g}
\le \frac{\eta}{\sqrt{m}} \le  
(1+ \varepsilon)
E_{g}.
\end{align}
for some $\varepsilon >0$. Then, with probability exceeding $1-\gamma$, the function 
$g^\#= \sum_{ {\bm \nu} \in \cH_{s}} c_{\bm \nu}^\# \Psi_{\bm \nu}$, 
with the vector $\bc^\# = (c_{\bm \nu}^\#)_{{\bm \nu} \in \cH_{s}}$ solving \eqref{form:weighted_ell1}, satisfies 
\begin{gather}
\label{interp1}
\begin{aligned}
\| g - g^\#\|_{\infty}   \le 
\|g - g^\#\|_{\omega,1} & 
\le (1 + \lambda + \varepsilon) c_3 \sigma_s^{(\ell)}(g)_{\omega,1},
\\
\| g - g^\#\|_{2}  & 
\le (1 + \lambda + \varepsilon)  c_4 \frac{\sigma_s^{(\ell)}(g)_{\omega,1}}{\sqrt{K(s)}}.
\end{aligned}
\end{gather}
Here, $c_3$ and $c_4$ are universal constants. 
\end{theorem}
\begin{proof}
We introduce $\widehat{\bm\xi} = (\widehat{\xi}_i)_{1\le i\le m} = (g_{\cH_{s}^c}({\bm y}_i))_{1\le i\le m}$.
Since ${\bm y}_i$ are i.i.d. random variables with respect 
to $\varrho$, by Lemma \ref{lemma:est_suppl},
\begin{align}
\E[\widehat{\xi}_i^2] = \|{\bc}_{\cH_{s}^c}\|_2^2 \le 
\frac{2 \|{\bc}_{\cH_{s}^c}\|^2_{\omega,1}}{{K(s)}} + 2 {{K(s)}} M_{\cH_{s}^c}^2,
\label{inequalityInProof}
\end{align}
where we have defined $M_{\cH_{s}^c} := \max_{{\bm \nu} \in \cH_{s}^c} \frac{|c_{\bm \nu}|}{\omega_{\bm \nu}}$. 
We consider two cases: 

$\bullet$ Case 1: $\|\bc_{\cH_s^c}\|_2 > \frac{ \|\bc_{\cH_s^c}\|_{\omega,1}}{\sqrt{K(s)}}$. 
Since $| \widehat{\xi}_i | \le  \|\bc_{\cH_s^c}\|_{\omega,1} < \sqrt{K(s)}\|{\bc}_{\cH_s^c}\|_2$, it holds 
$$
\E\[(\widehat{\xi}_i^2 -\E[\widehat{\xi}_i^2])^2\] 
\leq \E[\widehat{\xi}_i^4] 
\leq  K(s)\|{\bc}_{\cH_s^c}\|_2^2~\E [\widehat{\xi}_i^2] 
\leq K(s)\|{\bc}_{\cH_s^c}\|_2^4.  
$$
Applying Bernstein's inequality with the mean-zero random variable 
$\widehat{\xi}_i^2 -\E[\widehat{\xi}_i^2]$ 
$$
\P\( \sum_{i=1}^m \frac{\widehat{\xi}_i^2}{m}    - \|{\bc}_{\cH_s^c}\|_2^2  \ge \kappa  \) \le \exp\(- \frac{m\kappa^2/2}{K(s)\|{\bc}_{\cH_s^c}\|_2^4 + \frac{\kappa}{3} K(s) \|{\bc}_{\cH_s^c}\|_{2}^2  }\)\!.
$$
Choose $\kappa = \|{\bc}_{\cH_s^c}\|_2^2  $, there follows
$$
\P\(  \sum_{i=1}^m \frac{\widehat{\xi}_i^2}{m}   \ge 2 \|{\bc}_{\cH_s^c}\|_2^2 \) \le \exp\(-\frac{3m}{8 K(s)}\).
$$

$\bullet$ Case 2: $\|{\bc}_{\cH_s^c}\|_2 \le \frac{ \|{\bc}_{\cH_s^c}\|_{\omega,1}}{\sqrt{K(s)}}$. 
Similarly to \cite{RW15}, by Bernstein's inequality,
$$
 \P\(  \sum_{i=1}^m \frac{\widehat{\xi}_i^2}{m} \ge \frac{ 2 \|\bc_{\cH_s^c}\|^2_{\omega,1}}{K(s)}\) 
 \le 
 \P\(  \sum_{i=1}^m \frac{\widehat{\xi}_i^2}{m} - \|{\bc}_{\cH_s^c}\|_2^2 
 \ge \frac{ \|\bc_{\cH_s^c}\|^2_{\omega,1}}{K(s)}  \)  
 \le \exp\(-\frac{3m}{8K(s)}\).
$$

In both cases, given that 
$m\ge (8/3) s^{1+\beta'} \log(1/\gamma) \geq
(8/3) K(s) \log(1/\gamma)$,
then with probability 
exceeding $1-\gamma$, 
$$
\frac1m\|\widehat{\bm \xi}\|_2^2 \le 
\max\Big\{2 \|\bc_{\cH_s^c}\|_2^2,\frac{ 2 \|\bc_{\cH_s^c}\|^2_{\omega,1}}{K(s)} \Big\} 
\le\frac{ 4 \|\bc_{\cH_s^c}\|^2_{\omega,1}}{K(s)} +  4 K(s) M_{\cH_s^c}^2,
$$
where the second inequality follows from \eqref{inequalityInProof}.
Under the condition \eqref{cond_noise}, an application of Theorem 
\ref{theorem:lower_ell1} yields 
\begin{align}
\label{est2}
\|g - g^\#\|_{\omega,1} & \le c_1 \sigma_s^{(\ell)}({\bc})_{\omega,1} + (1+ \varepsilon) d_1 \|{\bc}_{\cH_s^c}\|_{\omega,1} + (1+ \varepsilon) d_1 K(s) {M_{\cH_s^c}}. 
\end{align}
Let {${\Lambda^*}$ denote the support of best lower $s$-term approximation of $g$ in $\ell_{{\omega,1}}$-norm, $\tilde{\cJ}$ be determined by \eqref{assump:smooth}-\eqref{assump:smooth2}.} For every $ {\bm \nu} \in \hat{\cJ} := \tilde{\cJ} \setminus \Lambda^{*}$, there holds 
$\omega^2_{\bm \nu} M_{\cH_s^c} \le  (1 + \lambda) \omega_{\bm \nu}^2 
\frac{|c_{\bm \nu}|}{\omega_{\bm \nu}} = (1+ \lambda) \omega_{\bm \nu} |c_{\bm \nu}|$. Summing these over 
$ \hat \cJ $ gives 
\begin{align}
\label{est3}
K(s)M_{\cH_s^c} \le K( \hat{\cJ} ) M_{\cH_s^c} \le (1 + \lambda) \|{\bc}_{\hat{\cJ}}\|_{\omega,1},
\end{align}
the first inequality coming from $K(\tilde{\cJ}) \ge 2K(s)$. We finally combine \eqref{est2} and \eqref{est3} to obtain \eqref{interp1}, which completes the proof.
\end{proof}

\subsection{Iterative hard thresholding}
{Thresholding algorithms for finding best $s$-term approximations consist of 
solving problems of the form}:
\begin{align}
\min \|\bm{Az} - \bm {\tilde{g}} \|_2\quad\quad \mbox{subject to}\quad\quad \supp(\bz) \le s. 
\end{align}
In this section, we study a specific thresholding approach for lower sparse recovery, which is 
guaranteed with the reduced query complexity \eqref{RIP_cond4}. The method solves the following 
constrained minimization problem:
\begin{align}
\label{eq:iht}
\min \|\bm{Az} - \bm {\tilde{g}} \|_2\quad\quad \mbox{subject to}\quad\quad \supp(\bz) \le s,\;\; \supp(\bz)\mbox{ lower}.
\end{align}
To achieve the minimum in \eqref{eq:iht} we first define the hard lower thresholding operator 
\begin{align}
H^{\ell}_s(\bz) = 
\argmin_{\substack{\supp(\tilde\bz)\;\; \text{lower} \\  |\supp(\tilde\bz)| =s }} 
\|\bz-\tilde\bz\|_2.
\end{align}
{Our goal is to approximate a smooth function of the form 
$g = \sum_{{\bm \nu} \in \cF} c_{\bm \nu} \Psi_{\bm \nu}$ by a sparse 
expansion supported in a predefined {polynomial subspace} $\mathbb{P}_{\cJ}$. We let 
$s$ be the sparsity level, $\cJ=\cH_s$ defined as in 
\eqref{hyperbolic_cross}, with ${\bm A} \in \C^{m\times N}$ and 
$\bm {\tilde{g}} \in \C^m$ be the normalized sampling matrix and vector 
of observations respectively. 
In what follows, we consider the following lower version 
of iterative hard thresholding algorithm \cite{BD09}.}
\begin{algorithm}
\label{alg:lower_ITH}
{\em (Iterative hard thresholding on lower sets)}
\vskip1pt
\begin{enumerate}
\item {\bf Initialization}: set the initial approximation $\bc^0$ as an $s$-sparse lower vector, e.g., 
$\bc^0 = {\bm 0}$.
\item {\bf Iteration}: repeat until a stopping criterion is met at $n = \overline{n}$: 
\begin{equation*}
\label{threshold}
{\bc}^{n+1} = H^{\ell}_s({\bc}^n + {\bm A}^*(\bm {\tilde{g}}  - \bm{A c}^n)).
\end{equation*} 
\item {\bf Output}: ${\bc}^{\#} = {\bc} ^{\overline{n}}$, and ${{g}}^{\#} \!=\! \sum_{{\bm \nu} \in \cH_s} {c}^{\#}_{\bm \nu} \Psi_{\bm \nu}$. 
\end{enumerate}
\end{algorithm}

We remark that a related thresholding approach entitled {\em iterative hard weighted thresholding} 
has been proposed in \cite{Jo13}. Similarly to that work, our method is 
designed for function reconstruction with preference to low-indexed terms. However, 
we focus on high-dimensional polynomial approximations and use the 
lower instead of weighted sparsity constraint for the thresholding operator. 
A surrogate of $H^{\ell}_s$ can exploit the lower set structure and is 
independent of weights, whose optimal choice is an important problem 
for iterative hard weighted thresholding. Below we provide the convergence result for 
Algorithm \ref{alg:lower_ITH}. 

\begin{theorem}
\label{theorem:lower_ITH}
Let $\cJ=\cH_s$. For $s\ge 2$, 
consider a number of samples as in \eqref{conditionMLower}. Let ${\bm y}_1, \ldots , {\bm y}_m$ be drawn independently from 
the orthogonalization measure $\varrho$ associated to $\{\Psi_{\bm \nu}\}$ and 
$\bA \in \C^{m\times N}$ the normalized sampling matrix. Then:
\begin{itemize}
\item[(i)]
with probability exceeding $1-\gamma$, the following holds for all $g= \sum_{{\bm \nu} \in \cF} c_{\bm \nu} \Psi_{\bm \nu}$: the function $g^n \!=\! \sum_{{\bm \nu} \in \cH_s} c^n_{\bm \nu} \Psi_{\bm \nu}$, with 
$\bc^n = (c^n_{\bm \nu})_{{\bm \nu} \in \cH_s}$ 
solving \eqref{threshold}, satisfies
\begin{align}
\label{ITH:error1}
 \|g^n -g_{\Lambda^*}\|_2 
 \le \rho^n \|g^0 - g_{\Lambda^*}\|_2 
 + \frac{\tau}{\sqrt{m}} \|{\bm {\tilde\xi}}\|_2,
\end{align} 
where ${\bm{\tilde \xi}} = ({\tilde \xi}_i)_{1\le i\le m} =$ $(g_{(\Lambda^*)^c}({\bm y}_i))_{1\le i \le m}$.

\item[(ii)]
for a fixed function $g = \sum_{{\bm \nu} \in \cF} c_{\bm \nu} \Psi_{\bm \nu}$ with 
$\|g\|_{\omega,1} < \infty$, 
with probability exceeding $1-\gamma$, for all $n\in \N$, the 
function $g^n = \sum_{{\bm \nu} \in \cH_s} c^n_{\bm \nu} \Psi_{\bm \nu} $, with 
$\bc^n = (c^n_{\bm \nu})_{{\bm \nu} \in \cH_s}$ solving \eqref{threshold}, satisfies
\begin{align}
\label{ITH:error2}
 \|g^n - g \|_2 \le \rho^n \|g^0 - g_{\Lambda^*}\|_2 +  (\tau\sqrt{2} + 1)  
 \sigma_s^{(\ell)}(g)_2 +  \frac{\tau \sqrt2}{\sqrt{K(s)}}{\|g_{(\Lambda^*)^c}\|_{\omega,1}}.
\end{align} 
\end{itemize}
Here, $\rho$ and $\tau$ are universal constants with $\rho<1$, {${\Lambda^*}$ is the support of best lower $s$-term approximation of $g$ in $\ell_{2}$-norm, and $ 
\sigma^{(\ell)}_s(g)_{2} = \|g_{(\Lambda^*)^c}\|_{2}$ is the best lower $s$-term error in $\ell_2$-norm.}
\end{theorem}

\begin{proof}
We will prove that the assertion (i) holds provided that {$\delta_{\ell,\alpha s} < 1/2$ (where $\alpha = 2$ for Legendre system and $\alpha =3$ for Chebyshev system)}, following the technique in \cite{Fou10} 
for standard iterative hard thresholding. First, 
let $\bv^n = \bc^n + \bA^*(\bm {\tilde{g}}  - \bA \bc^n)$. On one hand, by 
definition \eqref{threshold} of $\bc^{n+1}$ in Algorithm \ref{alg:lower_ITH}, and since $\Lambda^*$ is lower of cardinality smaller than $s$\begin{align*}
 \|\bv^n - \bc_{\Lambda^*} \|_2^2
 &\geq
\|\bm{v}^n - \bc^{n+1} \|_2^2 =
\|(\bv^n - \bc_{\Lambda^*}) 
- (\bc^{n+1} - \bc_{\Lambda^*} ) \|_2^2\\ 
&
= \|\bv^n - \bc_{\Lambda^*} \|_2^2 
+\|\bc_{\Lambda^*} - \bc^{n+1} \|_2^2 
-2\Re 
\< \bv^n - \bc_{\Lambda^*}, \bc^{n+1} - \bc_{\Lambda^*} \>.
\end{align*}
This yields that 
$\|\bc_{\Lambda^*} - \bc^{n+1} \|_2^2 \le 2\Re 
\< \bv^n - \bc_{\Lambda^*}, \bc^{n+1} - \bc_{\Lambda^*} \>$. 
On the other hand, 
\begin{align*}
& \< \bv^n - \bc_{\Lambda^*}, \bc^{n+1} - \bc_{\Lambda^*} \>  
= \< (\bm{Id} - \bA^*\bA)\bc^n + \bA^* \bm {\tilde{g}} - \bc_{\Lambda^*}, 
\bc^{n+1}  - \bc_{\Lambda^*} \>  \\
=& \, \< (\bm{Id} - \bA^*\bA)(\bc^n - \bc_{\Lambda^*}), 
\bc^{n+1} - \bc_{\Lambda^*} \> + 
\frac1{\sqrt m} 
\< \tilde{\bm{\xi}} , \bA (\bc^{n+1}  - \bc_{\Lambda^*})\>.
\end{align*}
It is easy to see that 
$\Lambda^{n,*}:= \supp(\bc^n) \cup \supp(\bc^{n+1}) \cup \supp(\bc_{\Lambda^*})$ 
is a lower set which satisfies {$K(\Lambda^{n,*}) \le 3K(s)\leq K(\alpha s)$ (see \eqref{K2sKsRemark} for the second inequality)}, thus from Theorem \ref{theorem:lower-RIP},
\begin{align*}
|\< (\bm{Id} - \bA^*\bA)(\bc^n - \bc_{\Lambda^*}), \bc^{n+1}  - \bc_{\Lambda^*} \>| 
& \le \delta_{\ell,\alpha s} \|\bc^n - \bc_{\Lambda^*}\|_2 \| \bc^{n+1}  - \bc_{\Lambda^*} \|_2, 
\\ 
|\< \tilde{\bm{\xi}} , \bA (\bc^{n+1}  - \bc_{\Lambda^*}) \>| 
& \le \sqrt{1 + \delta_{\ell,\alpha s}} \|\tilde{\bm{\xi}}\|_2 \|\bc^{n+1}  - \bc_{\Lambda^*}\|_2. 
\end{align*}
We have then 
$$
\| \bc_{\Lambda^*} - \bc^{n+1} \|_2 \le 2 \delta_{\ell,\alpha s} 
\|\bc^n - \bc_{\Lambda^*}\|_2 +  \frac{2 \sqrt{1 + \delta_{\ell,\alpha s}}}{\sqrt m}
\|\tilde{\bm{\xi}}\|_2.
$$
The inequality \eqref{ITH:error1} follows given 
$\delta_{\ell,\alpha s}<1/2$. 

For (ii), note that we have
$\E(\tilde{\xi}^2_i)  = \|\bc_{(\Lambda^*)^c}\|_2^2$ and
$|\tilde{\xi}_i|  \le \|\bc_{(\Lambda^*)^c}\|_{\omega,1}$
implying that
\begin{align*}
\E\(\tilde{\xi}_i^2 -\E[\tilde{\xi}_i^2]\)^2  
\le \E[\xi_i^4] \le 
\|\bc_{(\Lambda^*)^c}\|_{\omega,1}^2 
\|\bc_{(\Lambda^*)^c}\|_2^2. 
\end{align*}
We can consider two cases $ \|\bc_{(\Lambda^*)^c}\|_2^2 > \frac{ \|{\bc}_{(\Lambda^*)^c}\|^2_{\omega,1}}{K(s)}$ and $ \|{\bc}_{(\Lambda^*)^c}\|_2^2 \le \frac{ \|{\bc}_{(\Lambda^*)^c}\|^2_{\omega,1}}{K(s)}$ and prove 
using Bernstein's inequality, similarly to Theorem \ref{theorem:interp},
that with probability exceeding $1-\gamma$, we always have
\begin{align*}
\frac{\|\tilde{\bm{\xi}}\|_2^2}{{m}} \le 2  
\| \bc_{(\Lambda^*)^c}\|_2^2 +  \frac{2\|\bc_{(\Lambda^*)^c}\|^2_{\omega,1}}{K(s)}.
\end{align*}
Substituting the previous inequality to \eqref{ITH:error1}, we obtain 
$$
 \|g^n - g_{\Lambda^*}\|_{2} \le 
 \rho^n \|g^0 - g_{\Lambda^*}\|_2 
 + \tau \sqrt 2 \|{\bc}_{(\Lambda^*)^c}\|_2 + 
 \frac{ \tau \sqrt 2 }{\sqrt{K(s)}} { \|{\bc}_{(\Lambda^*)^c}\|_{\omega,1}}. 
$$
The inequality \eqref{ITH:error2} then can be concluded by virtue of the triangle inequality. 
\end{proof}

In general, it may not be feasible to find the optimal vector 
$H^{\ell}_s(\bz)$ exactly. Greedy procedures can be used to 
explore a near optimal lower, $s$-sparse truncation of $\bm z$ 
and provide a surrogate $\widetilde{H}^{\ell}_s(\bz)$ of $H^{\ell}_s(\bz)$ 
(see \cite{CCDS13}). A significant advantage in our context is 
that we do not have to compute the components of $\bm z$ inductively but have them all at hand. The exploration cost is therefore 
a fraction of that of matrix multiplication. {We can also consider 
new additional algorithms that take advantage of full knowledge of $\bm z$.}
%
%
%
%
The numerical realization of Algorithm \ref{alg:lower_ITH} and comparison 
with standard iterative hard thresholding and iterative hard weighted 
thresholding will be conducted in future work.

\section{Numerical illustrations}
\label{sec:experiment}
\begin{figure}[!htb]
\begin{center}
\includegraphics[width=0.305\textwidth,clip=true,trim=0mm 1mm 13mm 13mm]{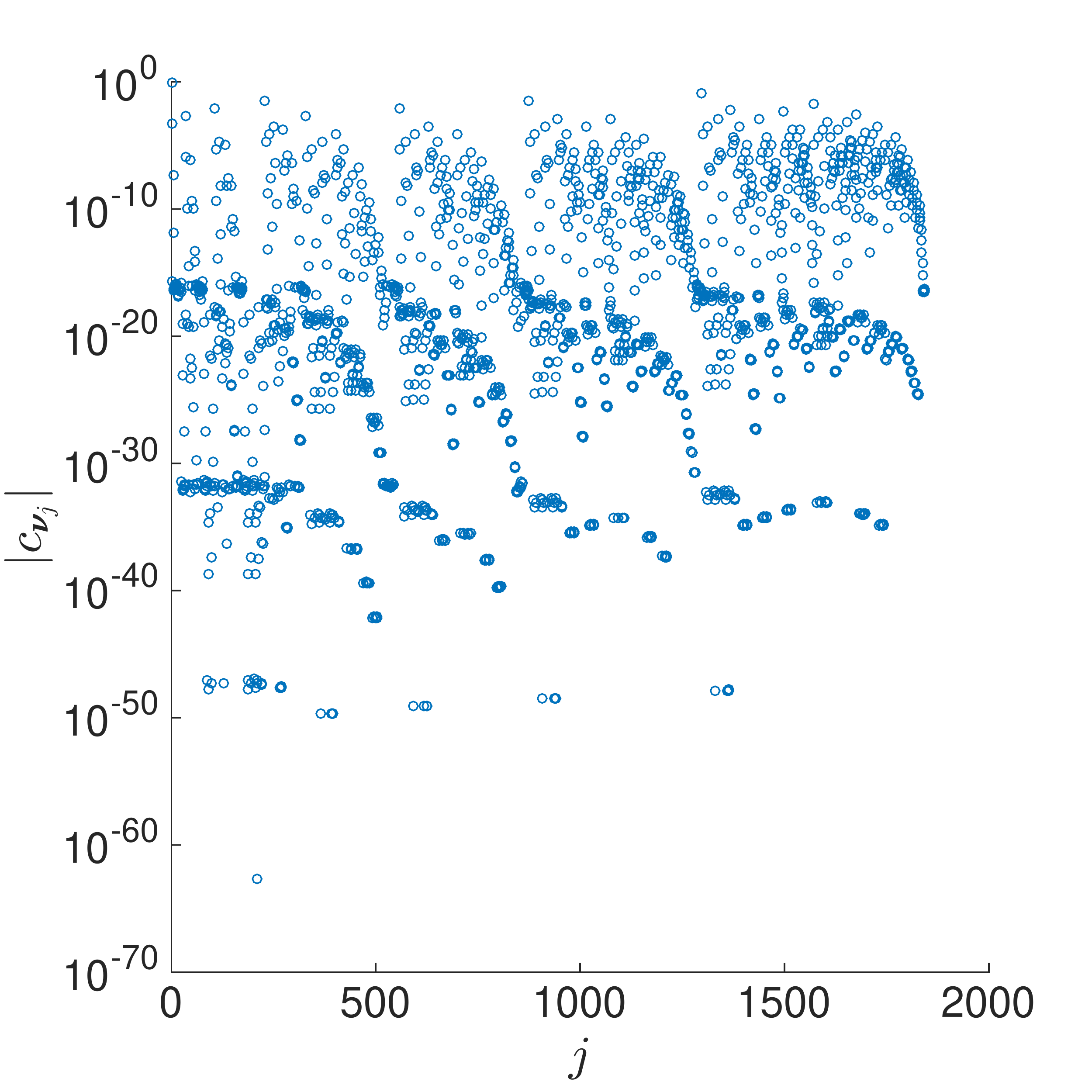}
\includegraphics[width=0.305\textwidth,clip=true,trim=0mm 1mm 13mm 13mm]{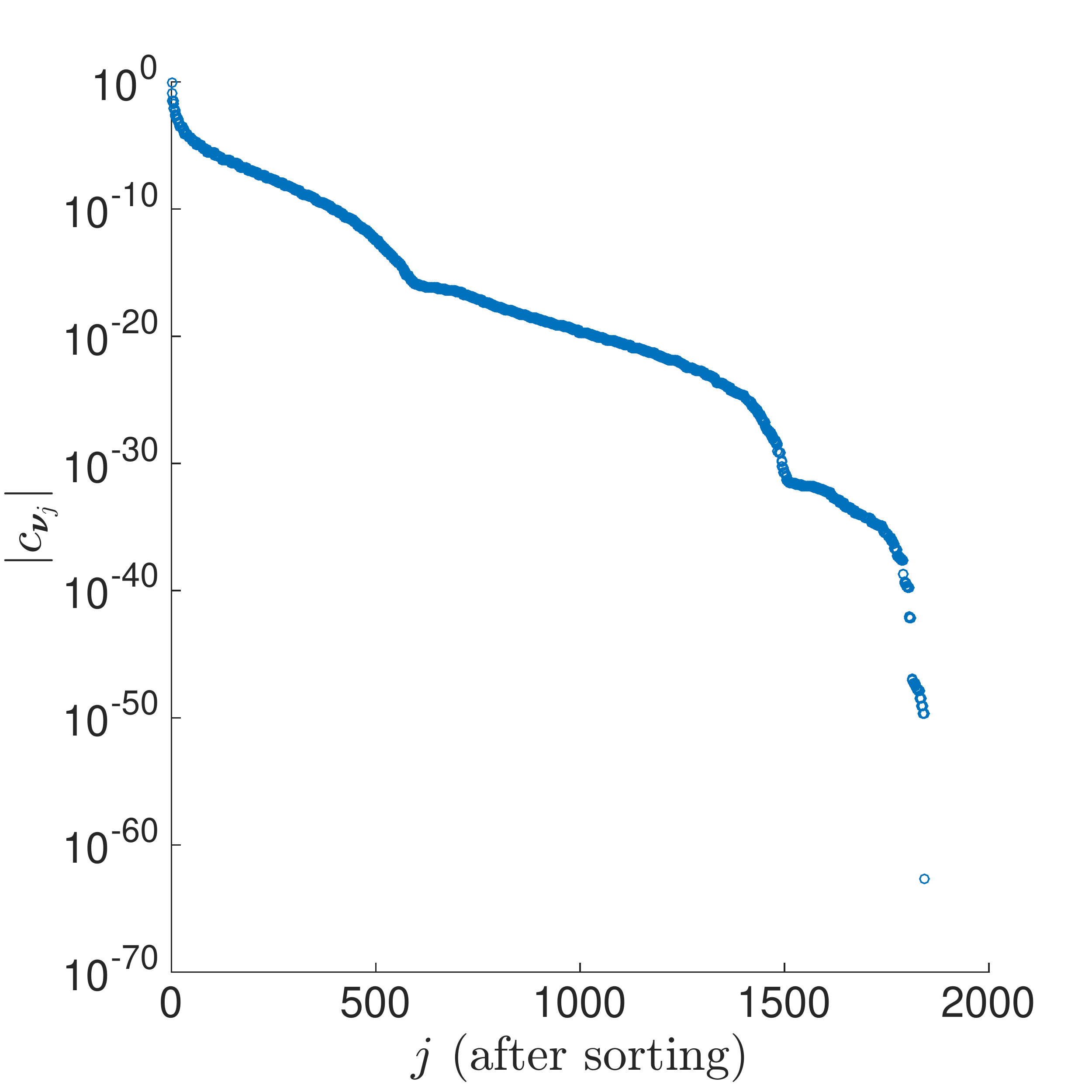}
\includegraphics[width=0.305\textwidth,clip=true,trim=0mm 1mm 13mm 13mm]{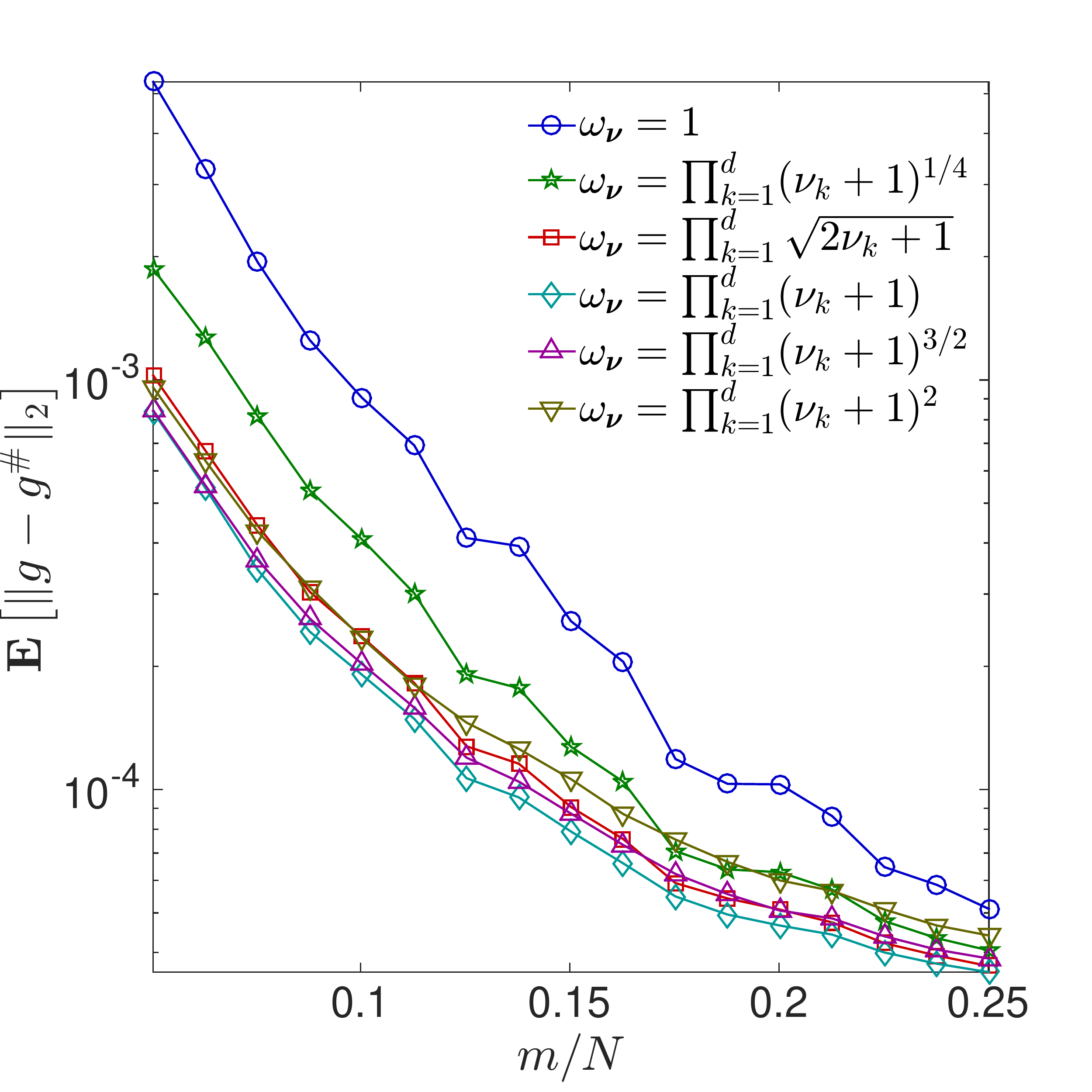}
\vskip5pt
\includegraphics[width=0.305\textwidth,clip=true,trim=0mm 1mm 13mm 13mm]{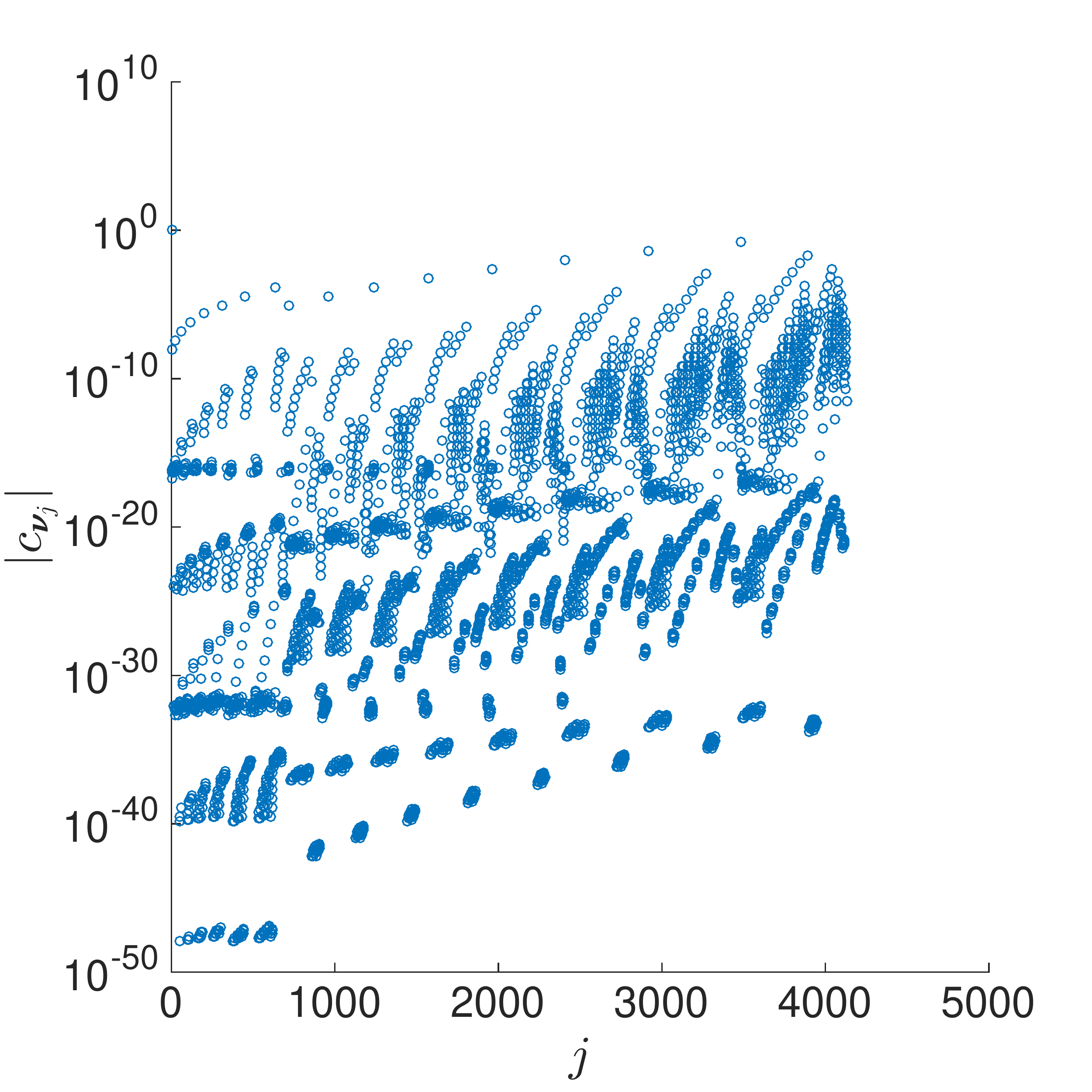}
\includegraphics[width=0.305\textwidth,clip=true,trim=0mm 1mm 13mm 13mm]{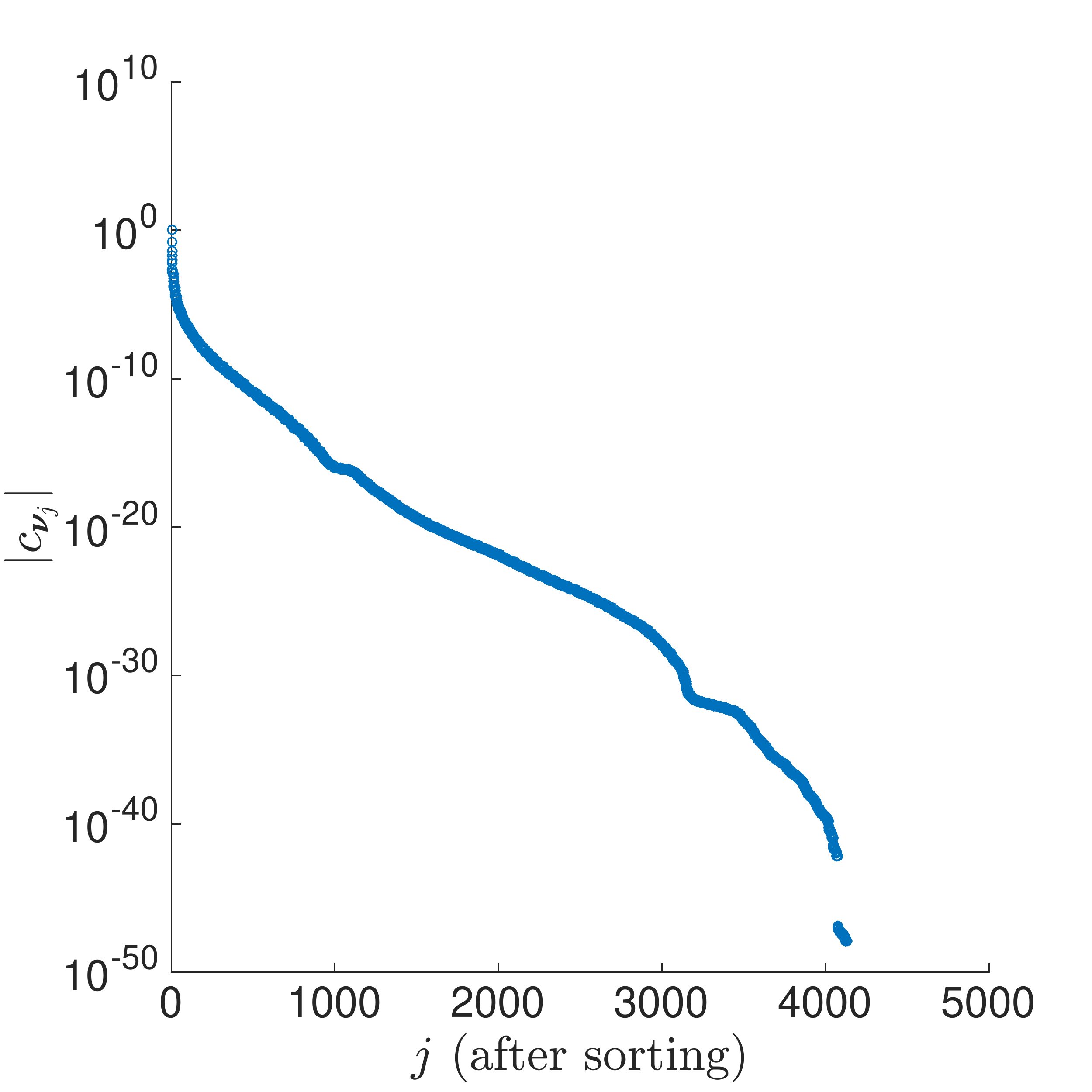}
\includegraphics[width=0.305\textwidth,clip=true,trim=0mm 1mm 13mm 13mm]{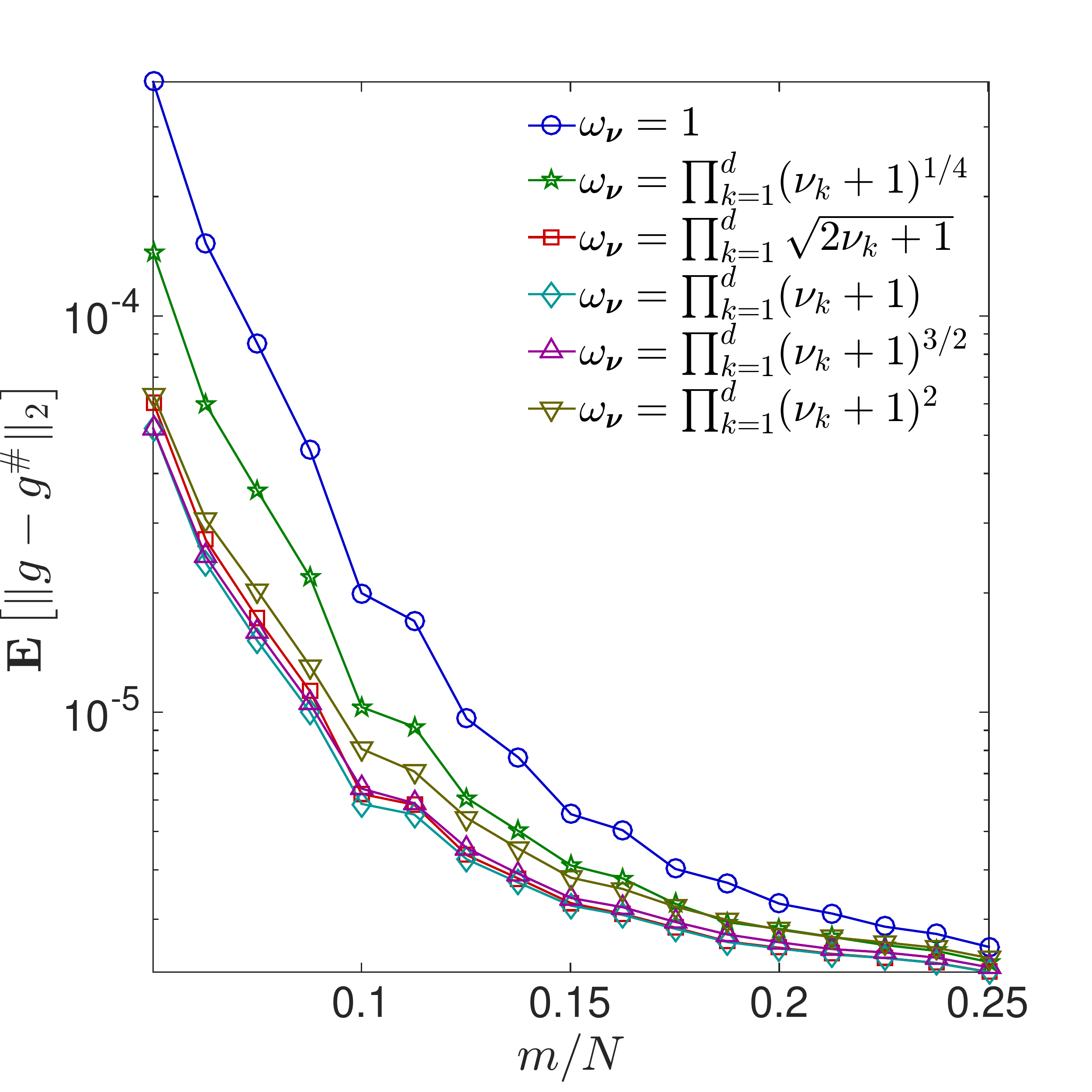}
\vskip5pt
\vspace{-.1in}
\caption{ 
Comparison of the averaged $L^2_\varrho$ error in approximating $g({\bm y}) = \frac{\prod_{k=1}^{\lceil d/2 \rceil} \cos(16 y_k/2^k)}{\prod_{k=\lceil d/2 \rceil + 1}^{d} (1-y_k/4^k) } $ using weighted $\ell^1$ minimization with various choices of weights. {\bf(Top)} $d=8$, $N=1843$, and $\| g_{\mathcal{H}_s^c} \|_{2} = 1.7487e-06$. {\bf(Bottom)} $d=16$, $N=4129$, and $\| g_{\mathcal{H}_s^c} \|_{2} = 1.3482e-06$.
}
\label{fig:mixed_trig_rational_results}
\end{center}
\end{figure}

\begin{figure}[!htb]
\begin{center}
\includegraphics[width=0.305\textwidth,clip=true,trim=0mm 1mm 13mm 13mm]{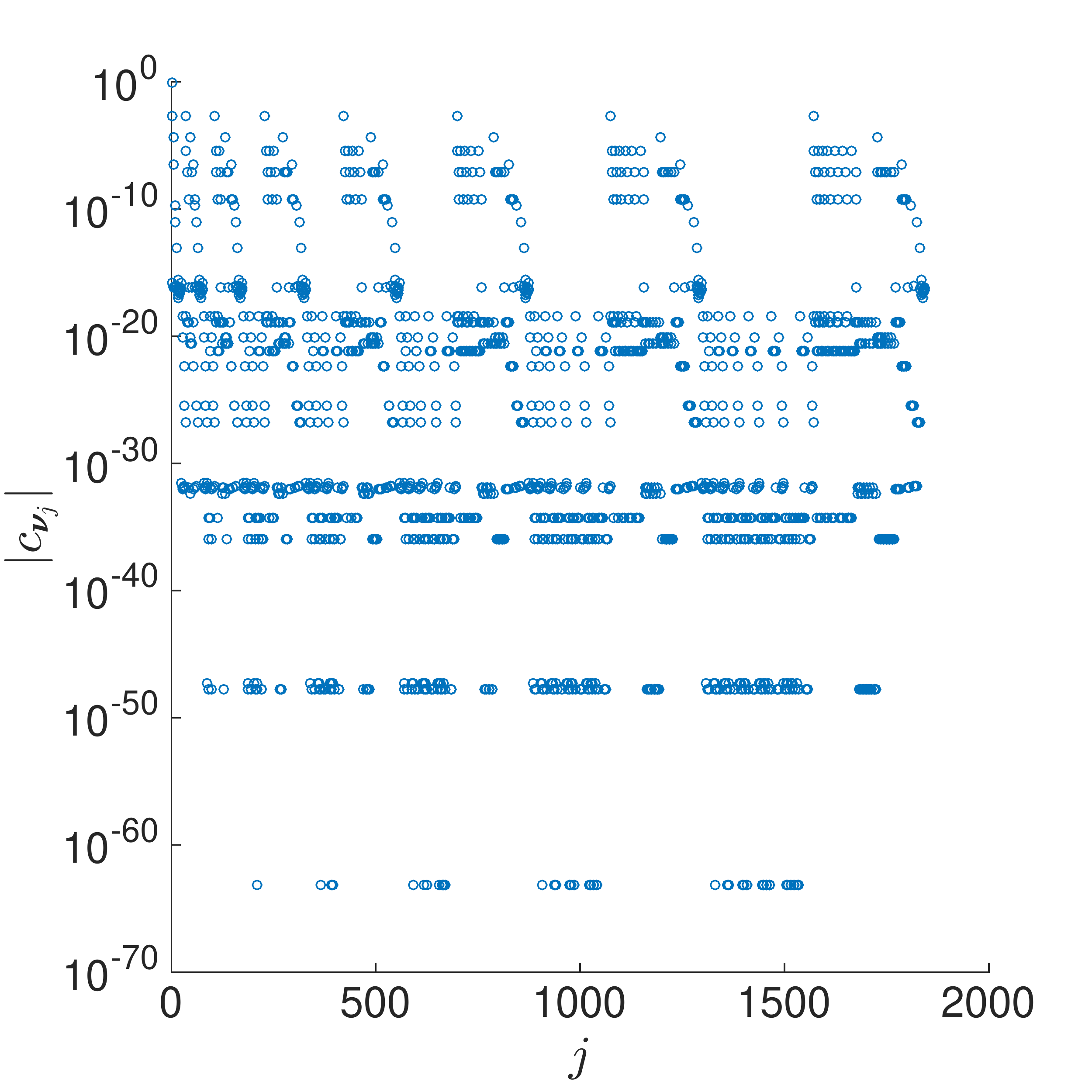}
\includegraphics[width=0.305\textwidth,clip=true,trim=0mm 1mm 13mm 13mm]{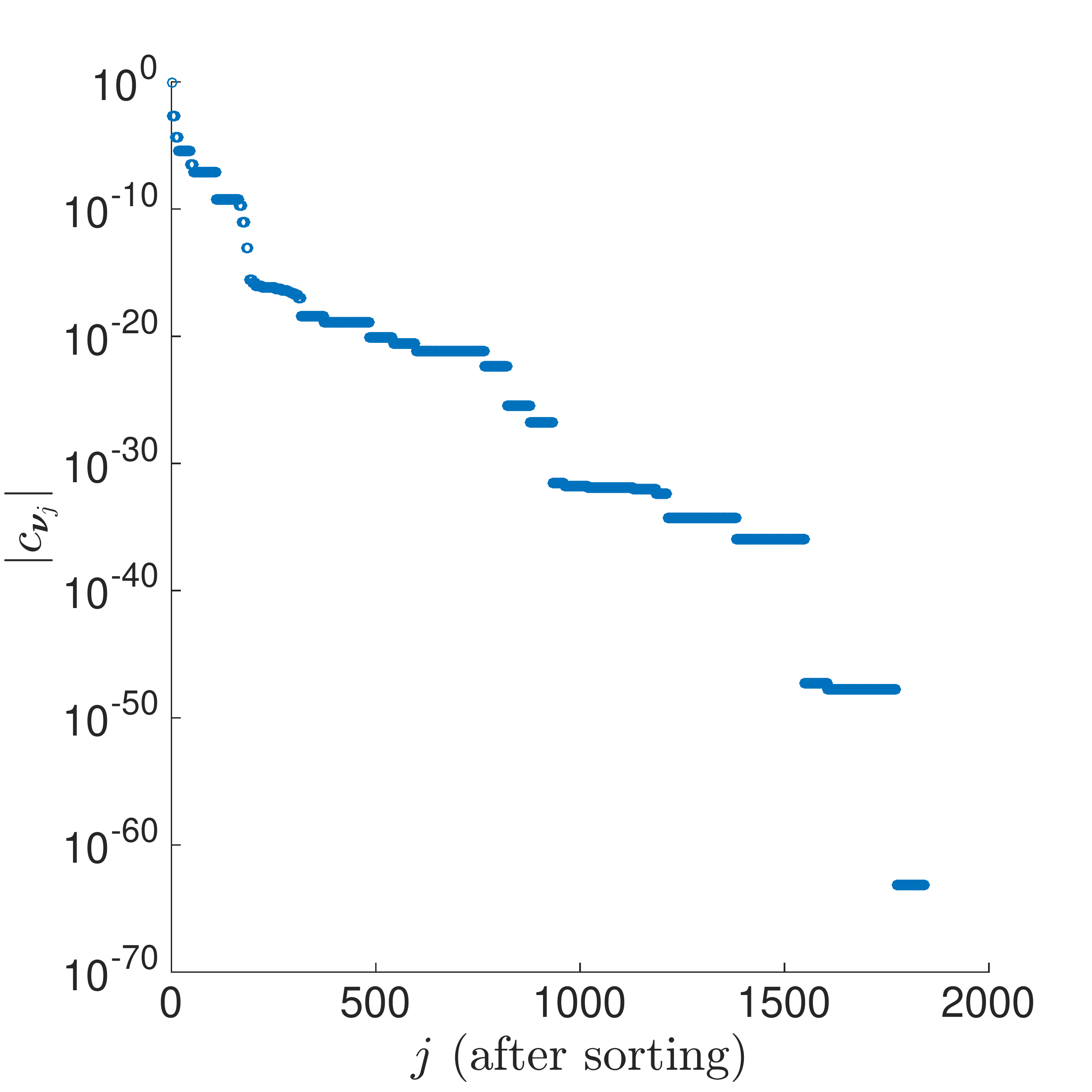}
\includegraphics[width=0.305\textwidth,clip=true,trim=0mm 1mm 13mm 13mm]{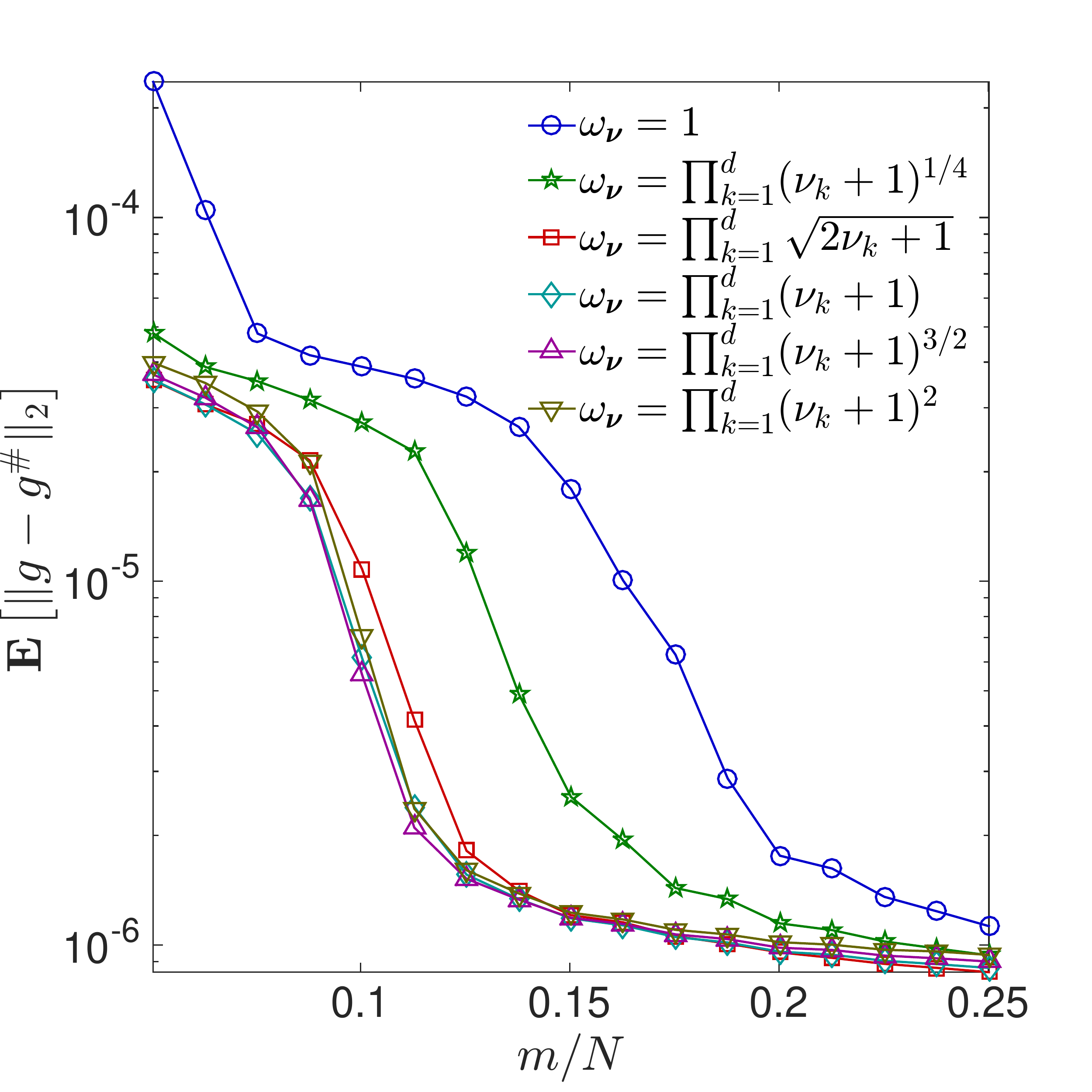}
\vskip5pt
\includegraphics[width=0.305\textwidth,clip=true,trim=0mm 1mm 13mm 13mm]{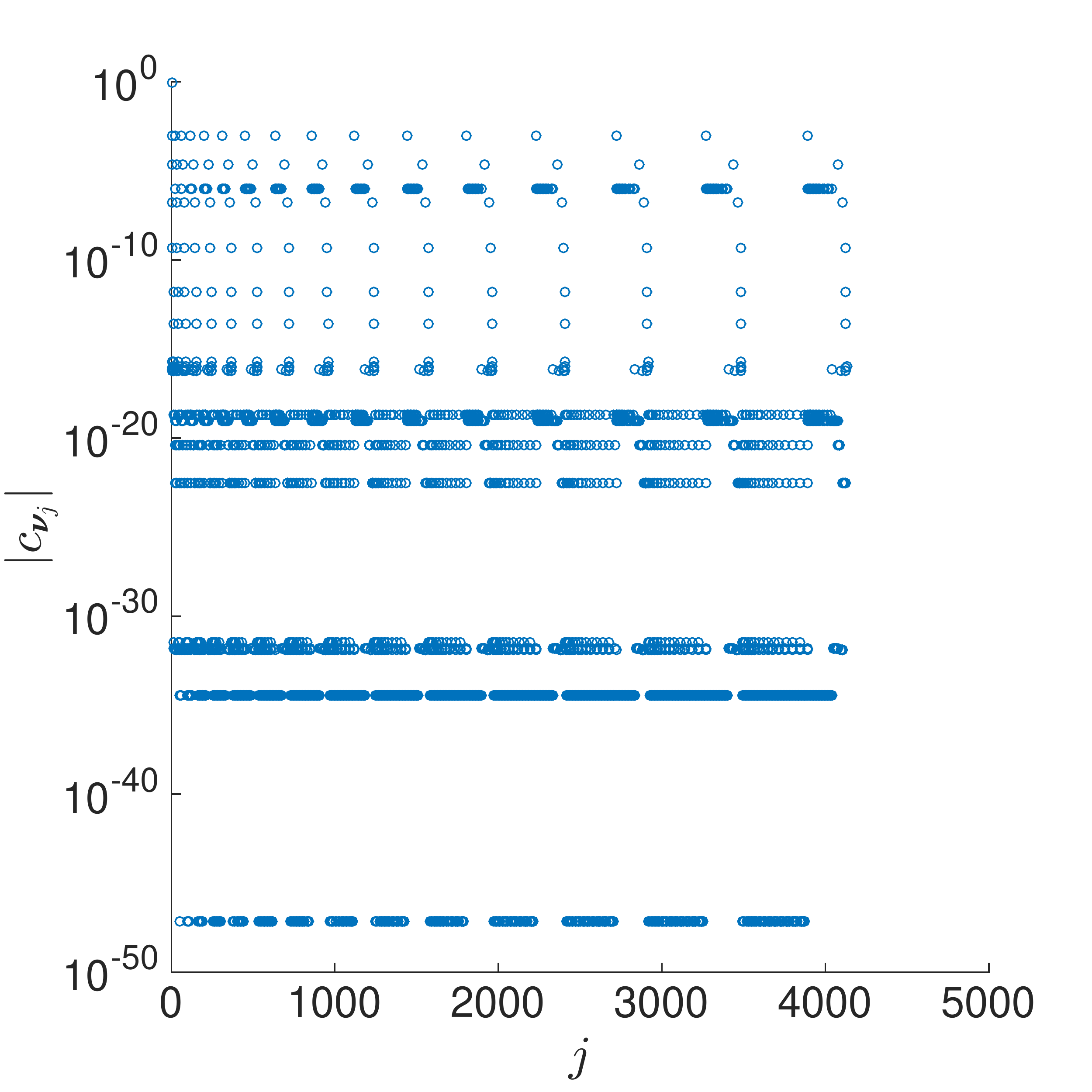}
\includegraphics[width=0.305\textwidth,clip=true,trim=0mm 1mm 13mm 13mm]{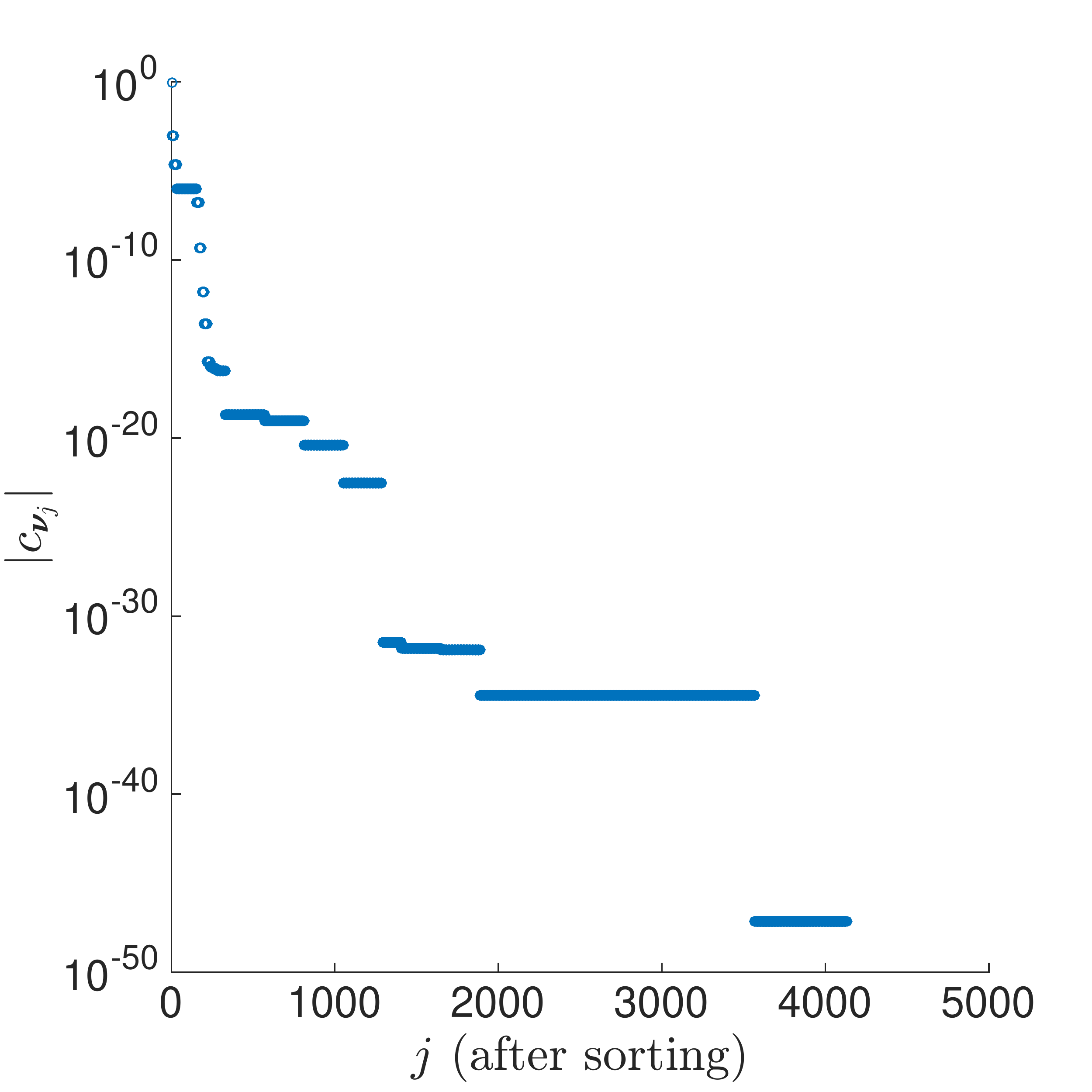}
\includegraphics[width=0.305\textwidth,clip=true,trim=0mm 1mm 13mm 13mm]{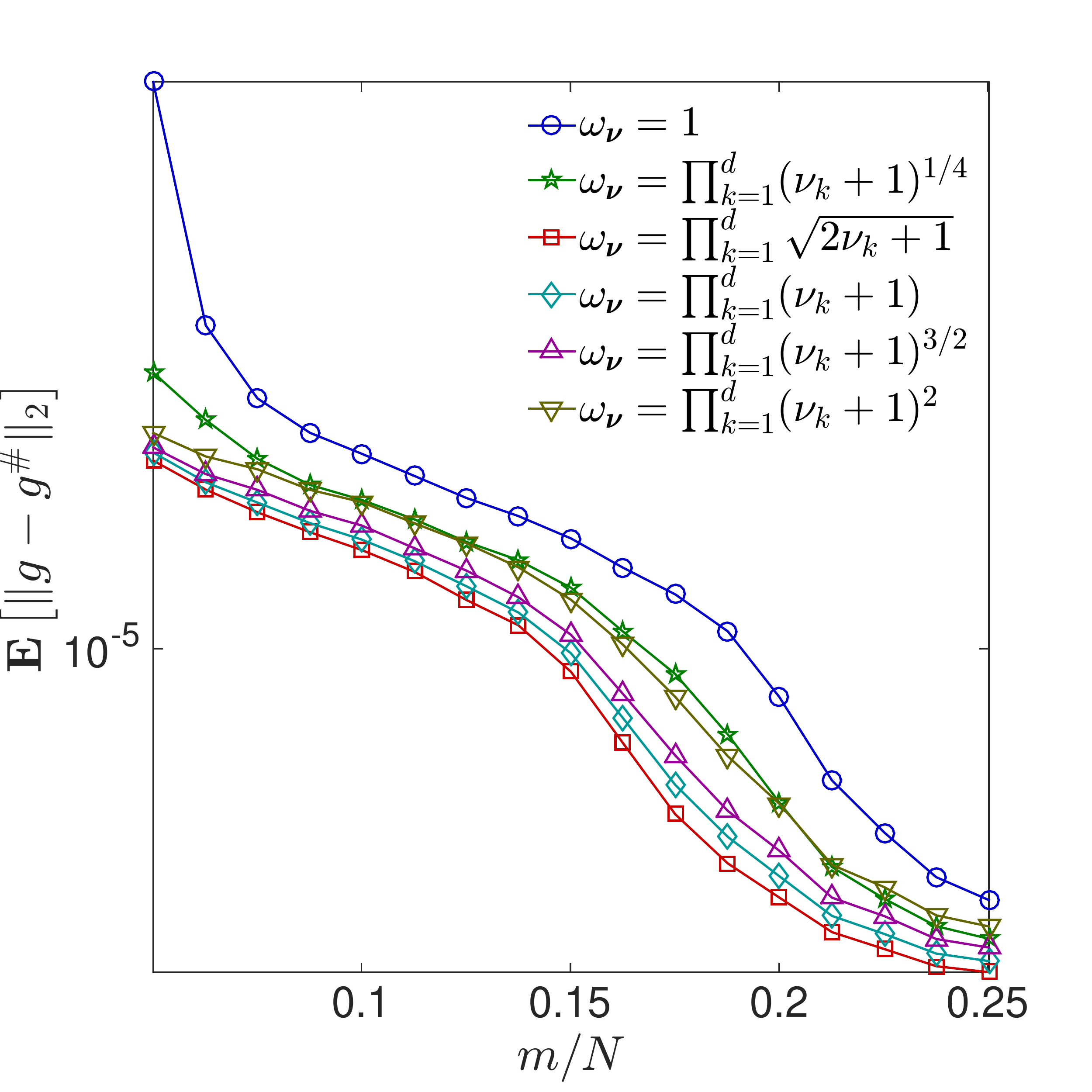}
\vskip5pt
\vspace{-.1in}
\caption{
Comparison of the averaged $L^2_\varrho$ error in approximating $g({\bm y})= \exp\left(\frac{-\sum_{k=1}^{d} \cos(y_k)}{8d}\right) $ using weighted $\ell^1$ minimization with various choices of weights. {\bf(Top)} $d=8$, $N=1843$, and $\| g_{\mathcal{H}_s^c} \|_{2} = 4.0232e-07$. {\bf(Bottom)} $d=16$, $N=4129$, and $\| g_{\mathcal{H}_s^c} \|_{2} = 2.0155e-06$.
}
\label{fig:dropoff_exp_cos_results}
\end{center}
\end{figure}

In this section, we provide several numerical examples to demonstrate the efficiency of our weighted $\ell_1$ minimization with $\omega_{\bm \nu} = \|\Psi_{\bm \nu}\|_{L^\infty}$ for smooth multivariate function recovery. We focus here on the approximation of $g$ in terms of orthonormal Legendre expansions by solving 
\begin{equation}
\min_{\bm{z}\in\C^N} \;\; \| \bm{z} \|_{\omega,1} \quad \text{subject to} \quad  \| \bm{\tilde{g}} - \bm{A}\bm{z} \|_2 \leq \frac{\eta}{\sqrt{m}} ,
\label{weighted_ell1_v2}
\end{equation}
where ${\eta} =  \sqrt{m} \| g_{\mathcal{H}_s^c} \|_{2} $ for various choices of weights.  As we test with simple functions, the expansion of $g$ can be computed numerically with the use of a quadrature approximation yielding an estimate of the tail $\eta$ a priori. The software code {\tt SPGL1}  \cite{BergFriedlander:2008,spgl1:2007} is employed to solve \eqref{weighted_ell1_v2}.
%
In each example, we choose the polynomial subspace $\mathbb{P}_{\cJ} \equiv \mathbb{P}_{\mathcal{H}_s}$, and increase the number of samples $m$ up to some $m_{\max} < N = \#(\cH_s)$. For each ratio $m/N$, the set of random samples is fixed over various choices of weights for a performance comparison. 
We then run 50 trials for the averaged $L^2_{\varrho}$ error, setting the maximum number of iterations in {\tt SPGL1} per trial to 10,000. Our results are shown in Figures \ref{fig:mixed_trig_rational_results}-\ref{fig:exp_results} for several multivariate functions. The left panels represent the magnitudes of polynomial coefficients (computed with MATLAB via a sparse-grid algorithm) indexed in $\cH_s$ and sorted lexicographically. The center panels depict the decay of the coefficients once sorted by magnitude. The right panels show the corresponding convergence results. 

Our experiments indicate that for functions with very fast decaying polynomial expansions, $\omega_{\bm \nu} = \| \Psi_{\bm \nu} \|_{L^\infty}$ is virtually the optimal weight. Indeed, for the function concerned in Figure \ref{fig:mixed_trig_rational_results}, consisting of trigonometric and rational univariate functions, we see in both $d=8$ and $d=16$ that the weight $\omega_{\bm \nu} = \| \Psi_{\bm \nu} \|_{L^\infty}$ significantly outperforms the unweighted $\ell_1$ approach. We also note that in $d=16$, higher weights begin to have decreasing benefit, performing worse than our proposed weight. 
%


The results in Figure \ref{fig:dropoff_exp_cos_results} are related to a function that
involves the exponent of a sum of univariate trigonometric functions. We note that only a small fraction of coefficients exceed $10^{-16}$ in both $d=8$ and $d=16$. In this case, we see that the weighted $\ell^1$ with the weight $\omega_{\bm \nu} = \| \Psi_{\bm \nu} \|_{L^\infty}$ performs the best out of all of the weights supplied. We also observe that increasing the weights beyond $\| \Psi_{\bm \nu} \|_{L^\infty}$ leads to a corresponding increase in the averaged $L^2_\varrho$ error. 


In Figure \ref{fig:slower_decay_rational_results}, we test with a root of a rational function. Here again, the weight $\omega_{\bm \nu} = \| \Psi_{\bm \nu} \|_{L^\infty}$ performs the best, and the approximation errors grow as the weights increase beyond this weight. The two highest weights even perform worse than unweighted $\ell^1$ for higher values of $m/N$. Comparing Figures \ref{fig:mixed_trig_rational_results} and \ref{fig:slower_decay_rational_results}, the similar center panels suggest similar decay rates for both functions inside $\cH_s$. The performance of the weights between the two examples however drastically differs, possibly due to the different support of the large coefficients. For this function, we were unable to test with $d=16$ and $N=4129$ due to the high expansion tail. 


On the other hand, if the polynomial expansion is less sparse, our weight may be not optimal. Figure \ref{fig:exp_results} shows the results for an exponential function of a linear combination of the $1$-$d$ variables. For this function, over half of the coefficients in $\cH_s$ exceed $10^{-8}$, and $\omega_{\bm \nu} = \| \Psi_{\bm \nu} \|_{L^\infty}$ performs worse than the larger weights. Still, we observe here, as in all tests for 
smooth functions in higher dimensions, i.e., $d=8$ and $d=16$, that our weight consistently provides improved accuracy compared with the  unweighted $\ell_1$, thus confirming the theory presented throughout. 
\begin{figure}[!htb]
\begin{center}
\includegraphics[width=0.320\textwidth,clip=true,trim=0mm 1mm 13mm 13mm]{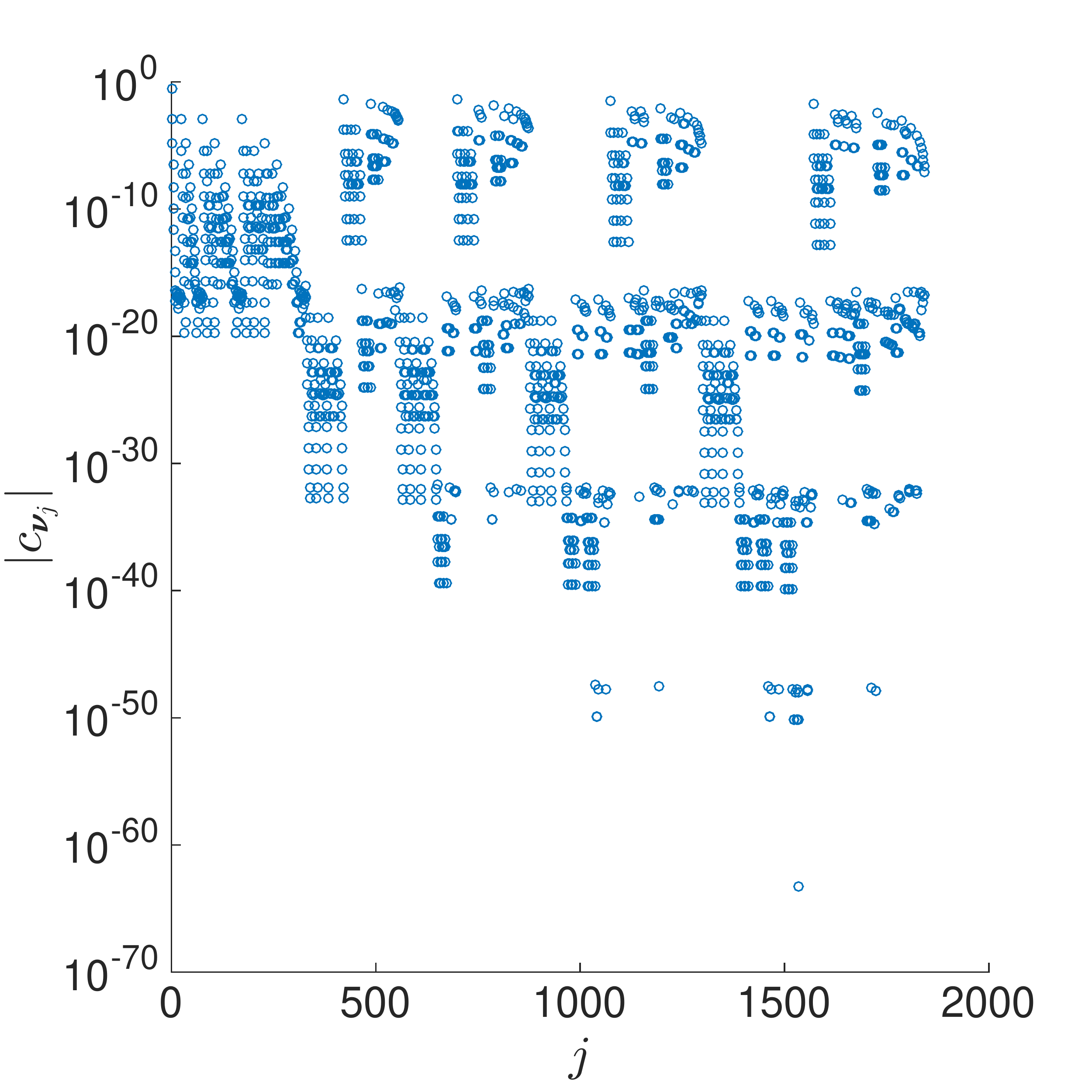}
\includegraphics[width=0.320\textwidth,clip=true,trim=0mm 1mm 13mm 13mm]{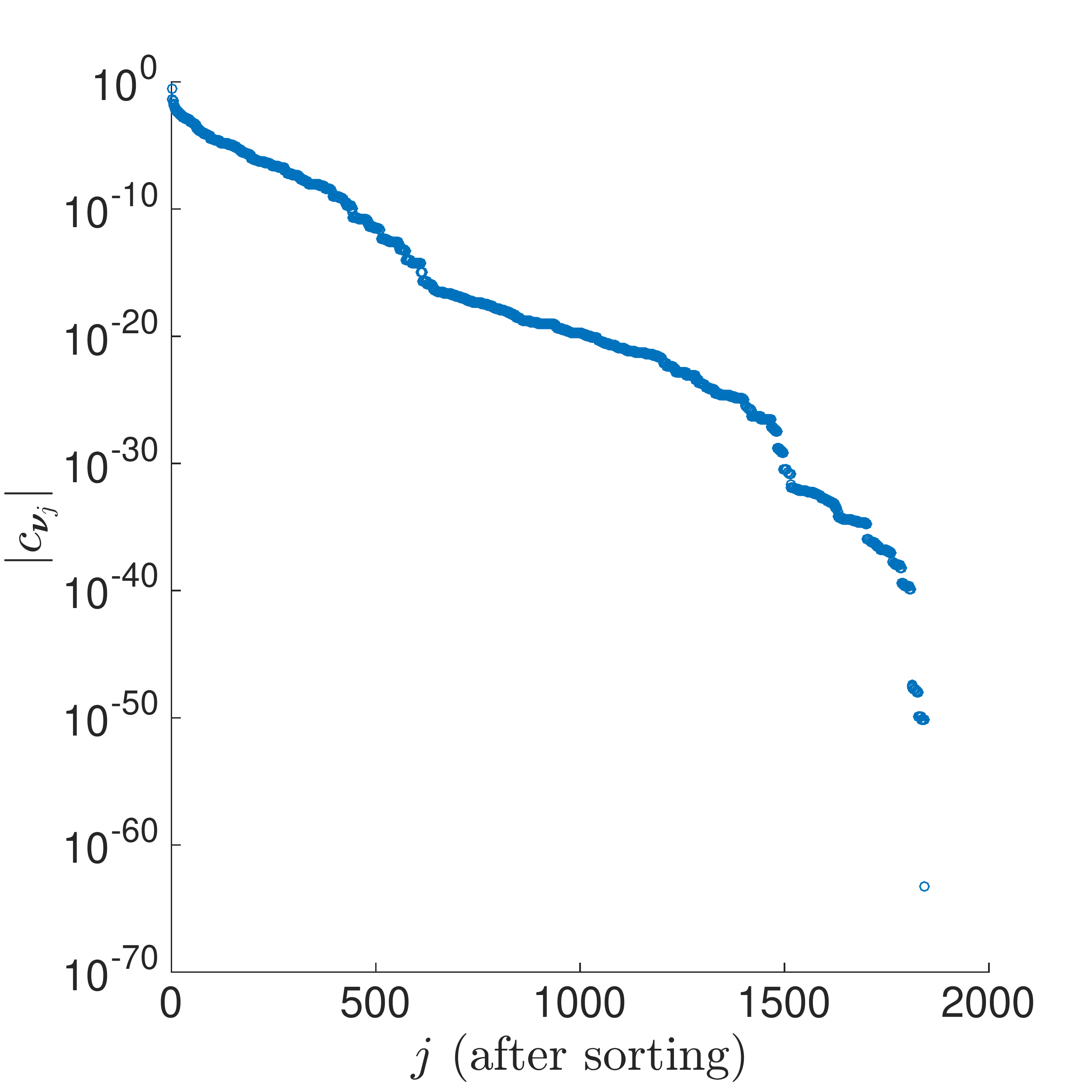}
\includegraphics[width=0.320\textwidth,clip=true,trim=0mm 1mm 13mm 13mm]{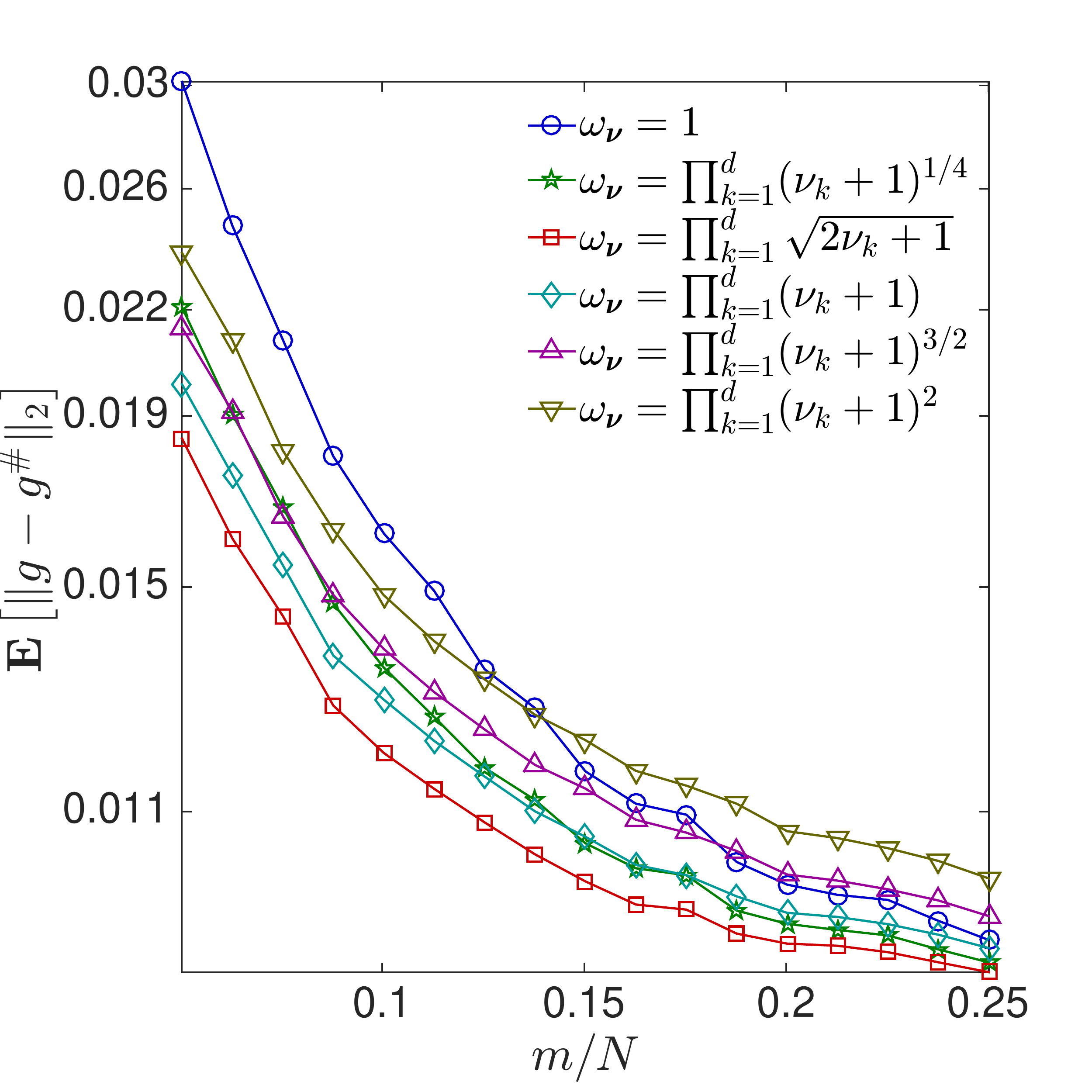}
\vskip5pt
\vspace{-.1in}
\caption{
Comparison of the averaged $L^2_\varrho$ error in approximating $g({\bm y})= \left[ \frac{\prod_{k=1}^{\lceil d/2 \rceil} (1+4^k y_k^2) }{\prod_{k=\lceil d/2 \rceil + 1}^d (100+5y_k)} \right]^{1/d}$ using weighted $\ell^1$ minimization with various choices of weights. $d=8$, $N=1843$ and $\| g_{\mathcal{H}_s^c} \|_{2} = 6.1018e-3$.
}
\label{fig:slower_decay_rational_results}
\end{center}
\end{figure}
%
%
\begin{figure}[!htb]
\begin{center}
\vspace{.1in}
\includegraphics[width=0.320\textwidth,clip=true,trim=0mm 1mm 13mm 13mm]{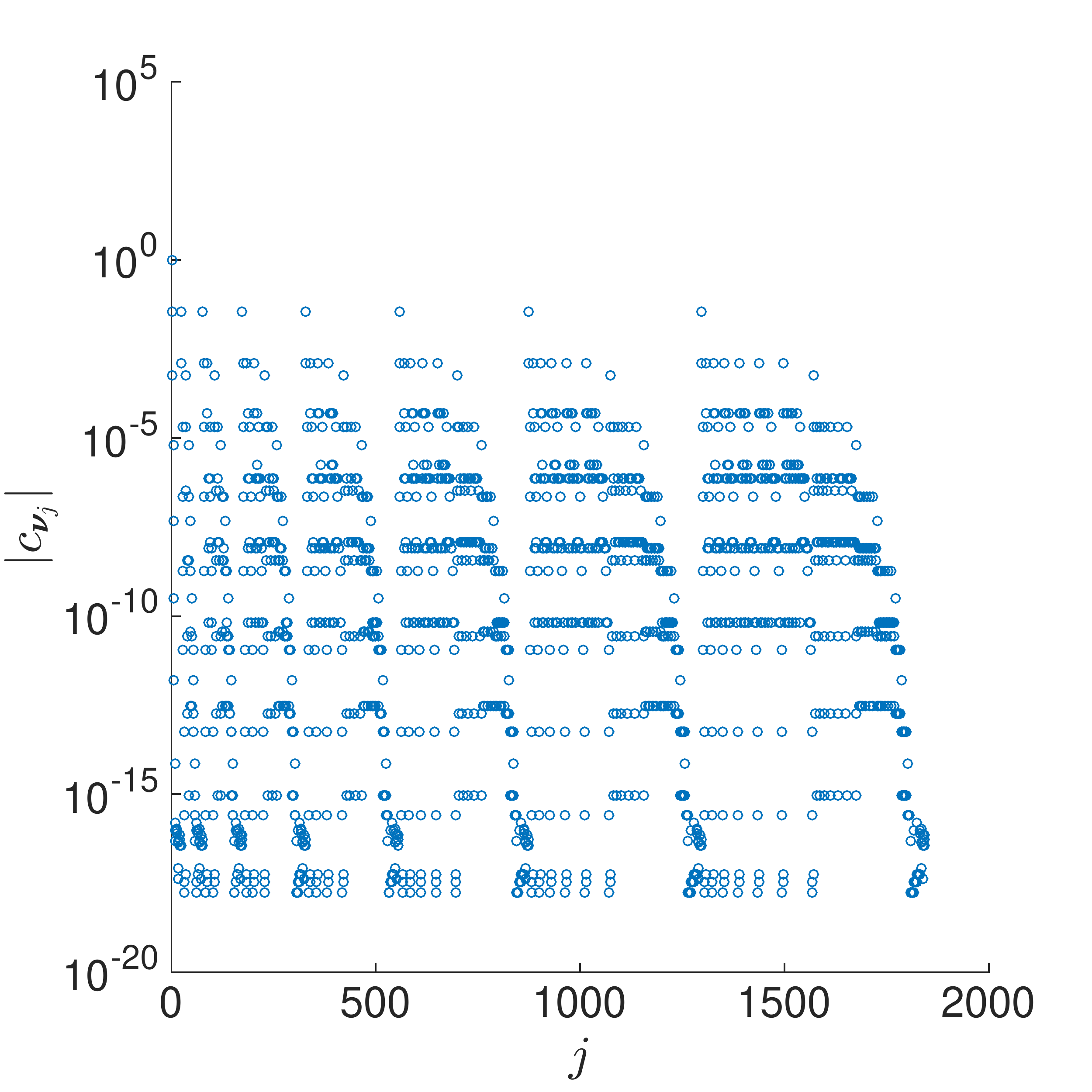}
\includegraphics[width=0.320\textwidth,clip=true,trim=0mm 1mm 13mm 13mm]{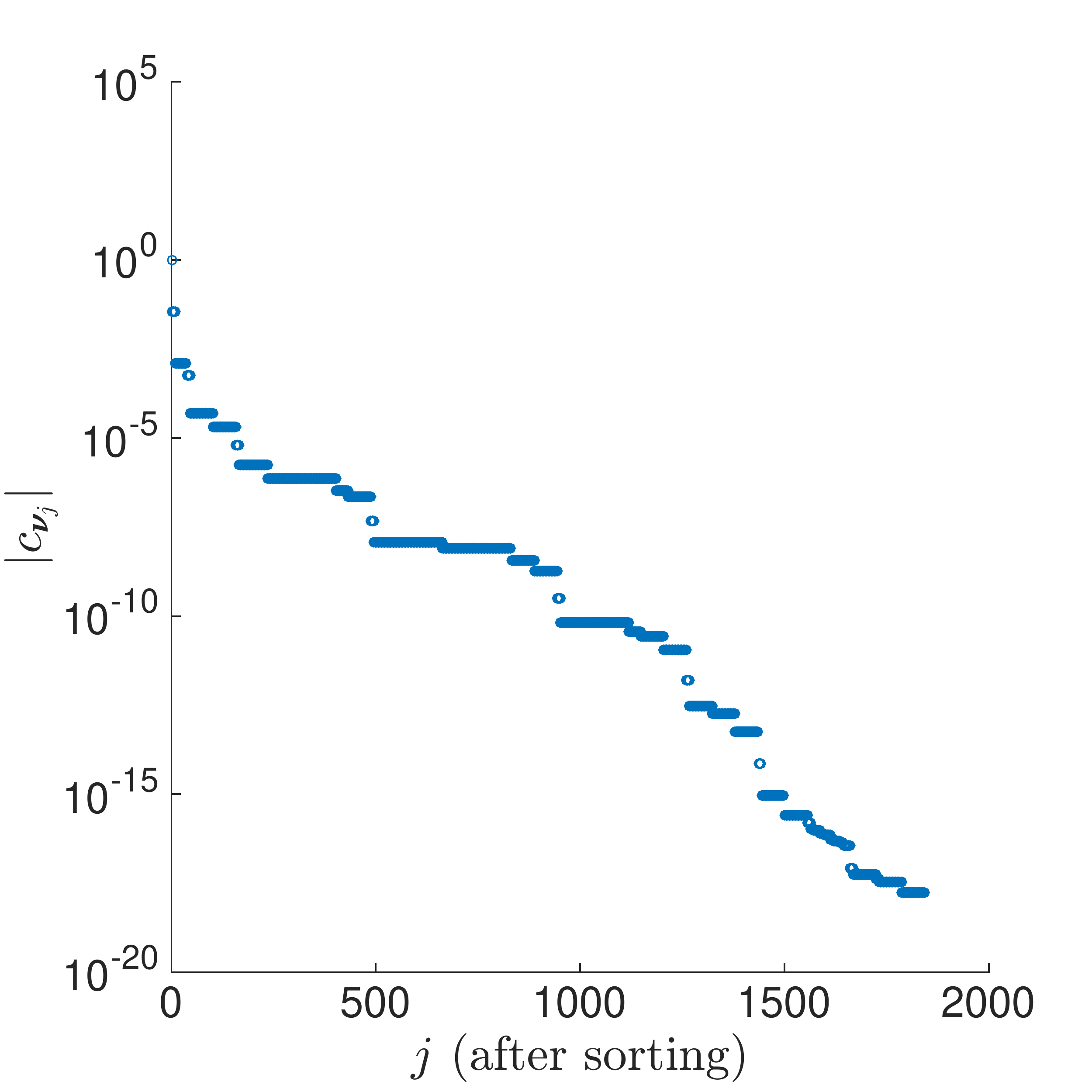}
\includegraphics[width=0.320\textwidth,clip=true,trim=0mm 1mm 13mm 13mm]{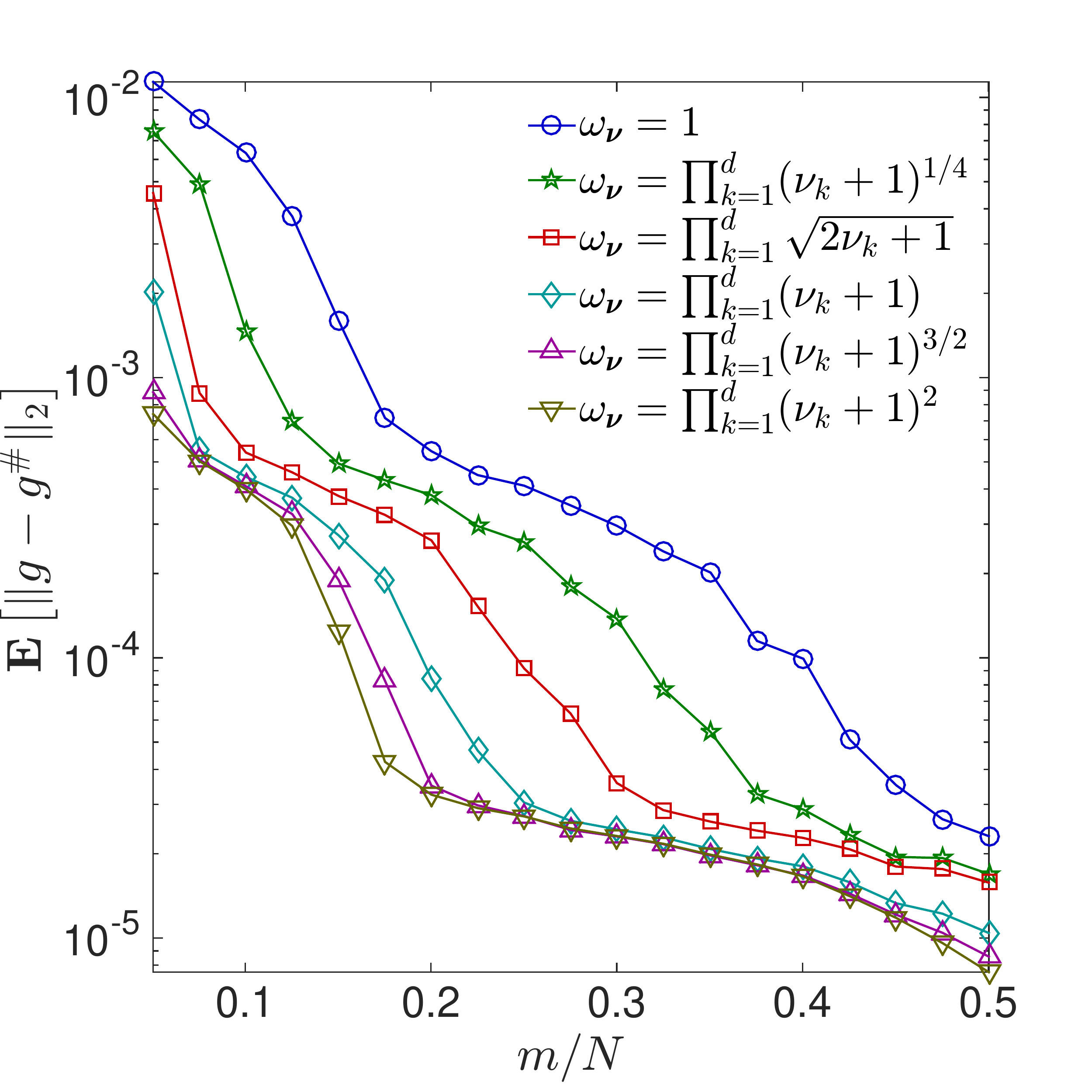}
\vskip5pt
\includegraphics[width=0.320\textwidth,clip=true,trim=0mm 1mm 13mm 13mm]{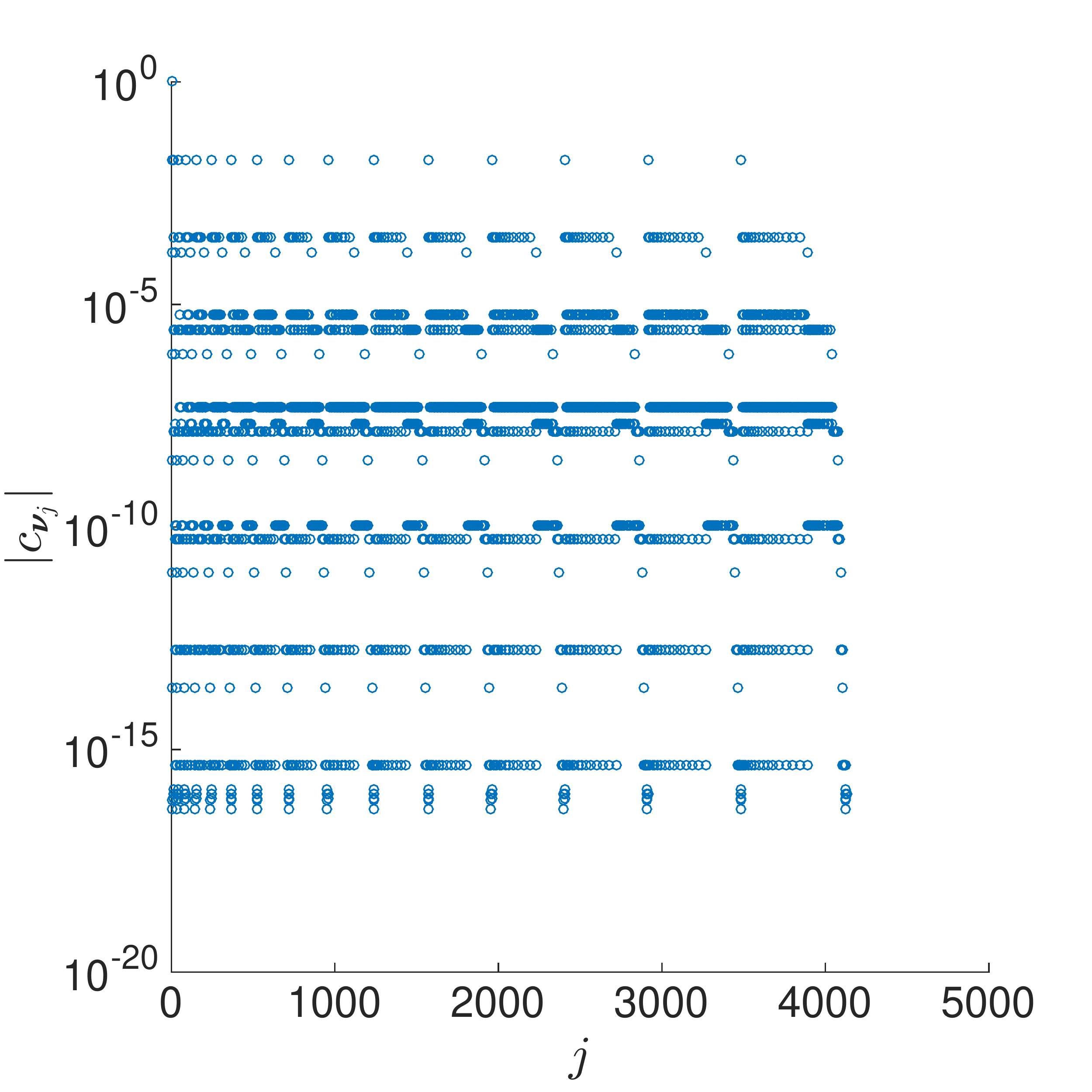}
\includegraphics[width=0.320\textwidth,clip=true,trim=0mm 1mm 13mm 13mm]{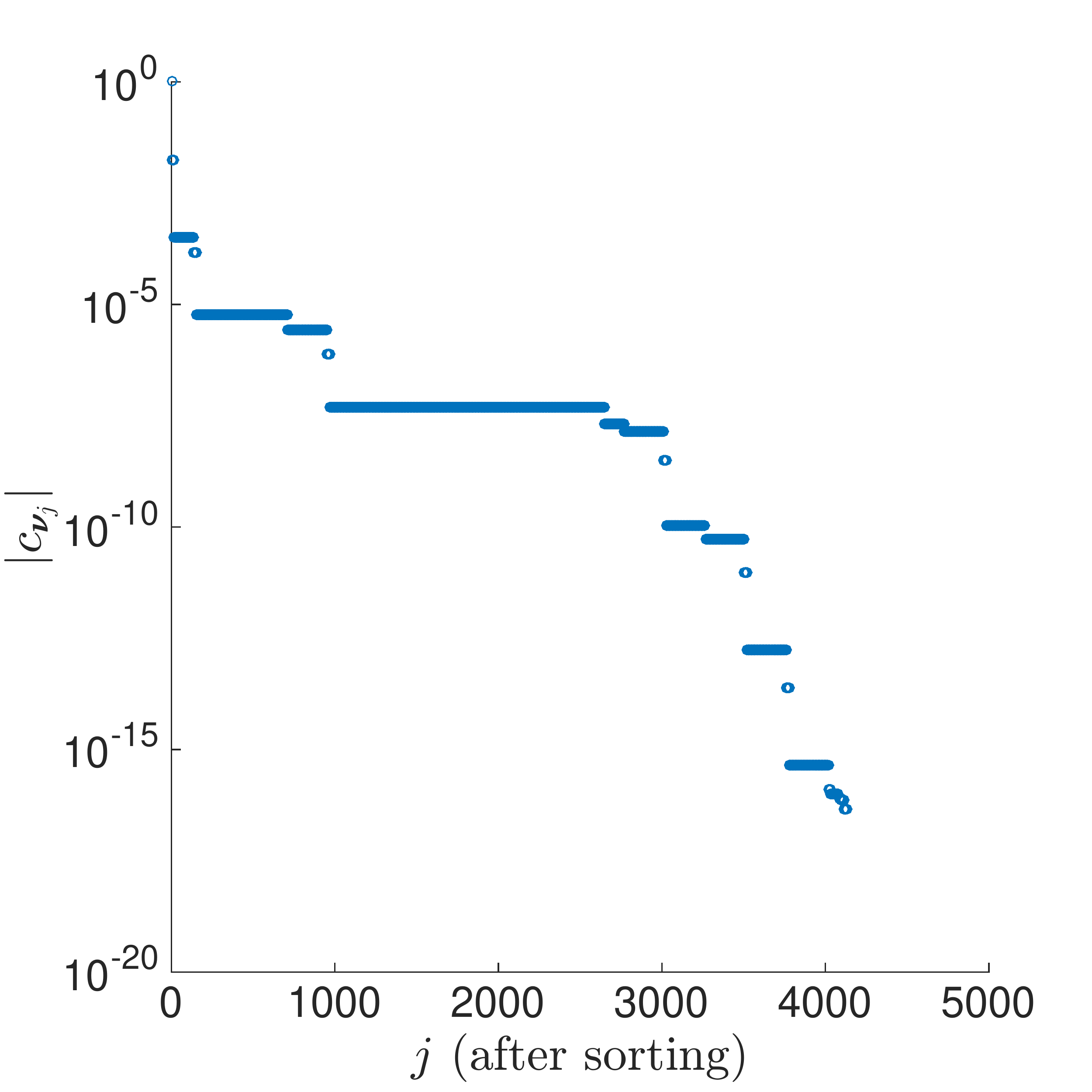}
\includegraphics[width=0.320\textwidth,clip=true,trim=0mm 1mm 13mm 13mm]{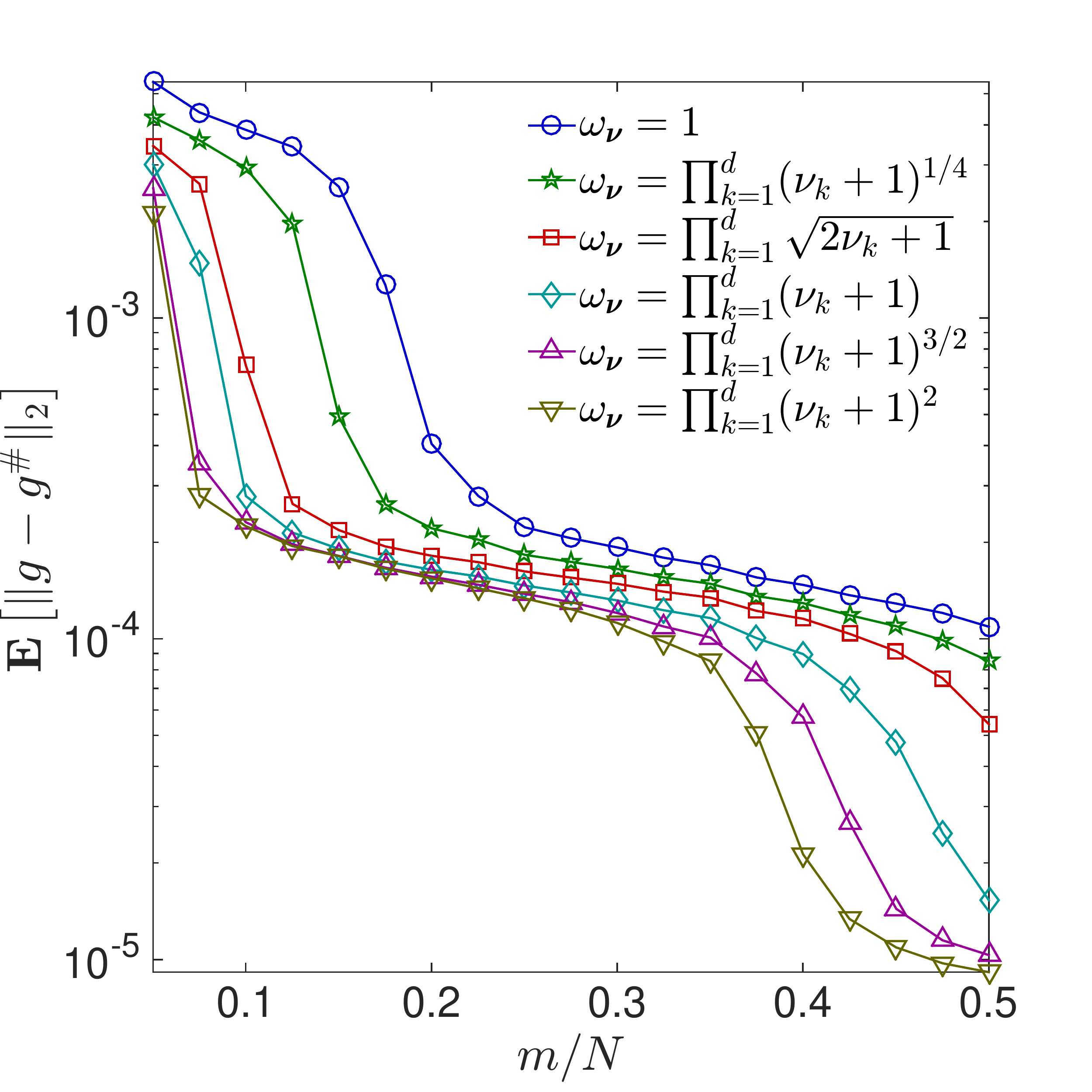}
\vskip5pt
\vspace{-0.1in}
\caption{
Comparison of the averaged $L^2_\varrho$ error in approximating $g({\bm y}) = \exp\left(-\frac{\sum_{k=1}^{d} y_k}{2d}\right)$ using weighted $\ell^1$ minimization with various choices of weights. {\bf(Top)} $d=8$, $N=1843$, and $\| g_{\mathcal{H}_s^c} \|_{2} = 7.2714e-07$. {\bf(Bottom)} $d=16$, $N=4129$, and $\| g_{\mathcal{H}_s^c} \|_{2} = 3.7412e-07$.
}
\label{fig:exp_results}
\end{center}
\end{figure}

\section{Concluding remarks}

In this work we present several novel compressed sensing approaches for sparse Legendre and Chebyshev approximations of real and complex functions in high dimensions. Motivated by the fact that the target function in many applications is smooth and characterized by a rapidly decaying orthonormal expansion, whose most important terms are captured by a lower set, we 
develop new $\ell_1$ minimization and hard thresholding procedures to impose the lower preference. 
Through rigorous analysis and numerical illustrations we demonstrate that the proposed methods possess a significantly reduced sample complexity compared to existing techniques. This advantage is established through the introduction of lower RIP, a weaker version of RIP that is associated with lower sets, and an optimal choice of polynomial subspace. In addition, we prove a generalized version of the result in \cite{Bou14} for bounded orthonormal systems and improve the RIP estimate by one logarithm factor.

Extending the theory and the procedures developed herein for approximating high-dimensional parameterized PDE systems is the next logical step, as it has been known that for a large class of such systems, the polynomial chaos expansions of parameterized solutions decay exponentially.  A significant challenge associated with this problem is that the ``signal'' $ {\bm c} = (c_{\bm \nu})_{\bm{\nu}\in \cJ}$ to be recovered has Hilbert-valued, rather than real or complex, components, i.e., $c_{\bm \nu} \in \mathcal{V}$. As shown in \cite{DO11,MG12,YK13,PHD14,RS14,HD15}, standard compressed sensing techniques only approximate a functional of PDE solutions, for instance, $u(x^\star,{\bm y}) $ (via ${c}_{\bm \nu}(x^{\star})$) at a single location $x^\star$ in physical space. Although $u(\cdot, {\bm y})$ can be constructed from these pointwise evaluations using piecewise polynomial interpolation, least square regression, etc., this practice faces two limitations. First, a priori information concerning the decay of $(\|c_{\bm \nu}\|_{\mathcal{V}})$, available in many applications, cannot be exploited for improving the convergence of recovery algorithms. Second, each $u(x^{\star},\bm{y})$ is only successfully reconstructed, {with a certain  probability}, and there may be a fraction of selected nodes $x^\star$ in which $u(x^\star,{\bm y}) $ is ill-approximated (i.e., with low probability). To address these difficulties, we aim to investigate new convex optimization and thresholding frameworks for Hilbert-valued functions, so that $c_{\bm{\nu}} \in \mathcal{V}$ can be directly computed. The mathematical analysis and computational aspect of this approach will be documented in future work.

\appendix
\section{Proofs of the RIPs} 
\label{sec:RIP}
We recall the following notations, which will be used throughout this section:
\begin{align*}
B_{1,N} &:= \{\bz \in \C^N: \|\bz\|_1 = 1\},\\
\cE_s  &:= \{\bz \in \C^N: \|\bz\|_2 = 1,\;\;\#(\supp(\bm{z})) \le s\},\\
\mathcal{E}^{\ell}_s   & := \{\bm{z} \in \mathbb{C}^N: \|\bm{z}\|_2 = 1,\, K(\supp(\bz)) \le K(s)\}, \text{ and}\\
\psi(\by,\bz) &:= \sum_{{\bm \nu}\in \cJ} {z}_{\bm \nu} \Psi_{\bm \nu}(\by),
\quad\text{with}\quad \by\in \mathcal{U},\;\; \bz \in \C^N.
\end{align*}

\subsection{Supporting lemmas} 
First, we derive a Chernoff-Hoeffding bound for 
complex random variables, as well as a tail bound 
for Bernoulli random variables. 

\begin{lemma}
\label{lemma:hoeffding}
Let $X_1,\ldots,X_M$ be $M$ independent identically distributed complex-valued 
random variables satisfying $|X_k| \le a$ and $\E[X_k] = X$ for all $k$. We denote 
$\overline{X} = \frac1M \sum_{k=1}^M X_k$. For every $\mu> 0$,
\begin{align}
\label{est:hoeffding}
\P\(\left| \overline{X} - X \right| \ge \mu \) \le 4 \exp\({-\frac{M \mu^2}{ 4a^2}}\).
\end{align}   
\end{lemma}

\begin{proof}
We have 
\begin{align*}
\mathbb{P}\left(\left| \overline{X} - X \right| \ge \mu \right) \le \mathbb{P}\left( \left| \Re( \overline{X} - X) \right| \ge \frac{\mu}{\sqrt{2}} \right) + \mathbb{P}\left( \left| \Im( \overline{X} - X) \right| \ge \frac{\mu}{\sqrt{2}} \right).
\end{align*}

Applying Hoeffding's inequality \cite{Hoeff63} for two sequences of bounded real random variables $\{\Re(X_k)\}$ and $\{\Im(X_k)\}$, there holds 
\begin{align*}
\mathbb{P}\left(\! \left| \Re( \overline{X} \!-\! X) \right| \ge \! \frac{\mu}{\sqrt{2}}\! \right) \le 2 \exp\({-\frac{M \mu^2}{ 4a^2}}\),\ \ \mathbb{P}\left( \! \left| \Im( \overline{X} \!-\! X) \right| \ge\! \frac{\mu}{\sqrt{2}} \! \right) \le 2 \exp\({-\frac{M \mu^2}{ 4a^2}}\).
\end{align*}
The proof is then complete.
\end{proof}

{\begin{lemma}
\label{lemma:tailbound}
Let $X_1,\ldots,X_M$ be $M$ independent identically distributed Bernoulli 
random variables with $\E[X_k] = X$ for all $k$. Denote 
$\overline{X} = \frac1M\sum_{k=1}^M X_k$. Then, for every $0< \mu_1 < 1$, 
$\mu_2> 0 $ and $M \ge \frac{16e}{\mu_1\mu_2}$, there holds
\begin{align}
\label{est:tailbound}
\P\( |\overline{X} - X|  \ge \mu_1 X + \mu_2\) \le \exp\(- \frac{M\mu_1\mu_2}{16e}\).
\end{align}   
\end{lemma}


\begin{proof}
Let $\epsilon_1,\ldots,\epsilon_M$ be a Rademacher sequence. Using symmetrization 
and Khintchine inequality {\cite[Section 8]{FouRau13}}, we have for any $q\ge 1$
\begin{gather*}
\begin{aligned}
  \mathbb{E}_X      \left|\overline{X} - X \right| ^q  & \le \frac{2^{q}}{M^q} \, \mathbb{E}_X \mathbb{E}_{\epsilon}    \left|\sum_{k=1}^M \epsilon_k X_k  \right| ^q 
 \\
& \le \frac{2^{q + \frac{3}{4}} e^{-q/2} q^{q/2}}{M^q} \mathbb{E}_X \left(\sum_{k=1}^M X_k^2 \right)^{q/2} = \left(
\frac{2^{1+\frac3{4q}} \sqrt q}{\sqrt{eM}} 
\right)^q  \mathbb{E}_X \overline{X}^{q/2}.
\end{aligned}
\end{gather*}

Denote $C = 2^{7/4} e^{-1/2}$, there follows
\begin{align*}
 \left(   \mathbb{E}_X      \left|\overline{X} - X \right| ^q \right)^{\frac{1}{q}} \le & \frac{2^{1+\frac3{4q}} \sqrt q}{\sqrt{eM}} \left( \mathbb{E}_X \overline{X}^{\frac{q}{2}} \right)^{\frac{1}{q}}
\le C \sqrt{\frac{q}{M}} \sqrt{ X } +   C \sqrt{\frac{q}{M}} \left( \mathbb{E}_X   \left|\overline{X} - X \right|^q  \right)^{\frac{1}{2q}},
\end{align*}
the last inequality implies
\begin{align*}
 \left(   \mathbb{E}_X      \left|\overline{X} - X \right| ^q \right)^{\frac{1}{q}}  \le C^2 {\frac{q}{M}} \, + \,  2 C \sqrt{\frac{q}{M}} \sqrt{ X } . 
\end{align*}


Applying Markov's inequality gives 
\begin{align*}
\mathbb{P}\left(  \left|\overline{X} - X \right| \ge e C^2 \frac{q}{M} \, + \,  2e C \sqrt{\frac{q}{M}} \sqrt{X} \right) \le e^{-q}. 
\end{align*}

Finally, it is easy to see that
\begin{align*}
 e C^2 \frac{q}{M} \, + \,  2e C \sqrt{\frac{q}{M}} \sqrt{X} \le \mu_1 X + \frac{C^2 q e}{M}\left(\frac{ e }{\mu_1} +  1  \right) < \mu_1 X +  \frac{C^2 q e^2\sqrt{2}}{M\mu_1}. 
 \end{align*}
 
Defining $q=  \frac{M\mu_1\mu_2}{16 e}$ so that $ \mu_2 = \frac{C^2 q e^2\sqrt{2}}{M\mu_1}$, we 
conclude the proof.
\end{proof}
}

Next, we state and prove an extended covering number result. 

\begin{lemma}
\label{note:lemma2}
For $0 < \varsigma < 1$, $\mu >0$, there 
exists a set $D \subset \C^N$ such that
\begin{enumerate}
\item[(i)] For all $\bz \in \cE_s$, $\by_1,\ldots,\by_m \in \cU$, there exists $\bz' \in D$ satisfying:
\begin{align}
|\psi(\by, \bm{z} - \bm{z}')| &\le \mu \quad\mbox{with probability exceeding } 
1-\varsigma \mbox{ in } (\cU,\varrho), \text{ and} \\
|\psi(\by_i, \bm{z} - \bm{z}')| &\le \mu \quad\mbox{for at least } (1-\varsigma) m 
\mbox{ indices } i \in \{1,\ldots,m\} .
\end{align}

\item[(ii)] The cardinality of $D$ satisfies 
\begin{align}
\log(\#(D)) \le (8/\mu^2) \Theta^2 s \log(4N)\log({12}/{\varsigma}).
\end{align}
 
\end{enumerate}
\end{lemma} 

\begin{proof}
We will find $D$ using the empirical method of Maurey. First, we observe 
that $\cE_s \subset {\sqrt{s}}B_{1,N}$, hence if we denote {$\mathcal{P} = \{\pm \bm{e}_j\sqrt{2s}, \pm i\bm{e}_j\sqrt{2s}\}_{1\leq j\leq N}$}, 
where $(\bm{e}_j)$ are canonical unit vectors in $\C^N$, we have 
$\cE_s  \subset \text{conv}(\mathcal{P})  $. Every ${\bm z}\in \mathcal{E}_s$ can be represented as ${\bm z} = \sum_{r=1}^{4N} \lambda_r \bm{v}_r$, for some $\lambda_r \ge 0$, $\sum_{r=1}^{4N} \lambda_r = 1$ and ${\bm v}_r$ listing $4N$ elements of $\mathcal{P}$.  
There exists a probability measure $\lambda$ on $\mathcal{P}$ that takes the 
values ${\bm v}_r \in \mathcal{P}$ with probability $\lambda_r$. 

Let $\bz_1,\ldots,\bz_M$ be i.i.d random variables with law $\lambda$. Note 
that $\E\bz_k = \bz$, for all $k =1,\dots, M$. For each $\by \in \mathcal{U}$, $\psi(\by,\bm{z}_k)$ 
is also a complex-valued random variable on probability space 
$(\psi(\by,\mathcal{P}),\lambda)$ with 
$$
|\psi(\by,\bm{z}_k)| \le \Theta\sqrt{2s},
\quad\quad\mbox{and}\quad\quad
\E\psi(\by,\bm{z}_k) = \psi(\by,\bm{z}). 
$$

Denote $\overline\bz =  \frac1M \sum_{k=1}^M \bz_k$, let $D$ be 
the set of all possible outcomes of $\overline\bz $ and $\overline\lambda$ be the probability measure on $D$ according to $\overline\bz$. 
We define a characteristic function $\chi$ on $(\cU\times D,\varrho\otimes {\overline\lambda})$ such that  
\begin{align*}
\chi(\by,\overline\bz) = 
\begin{cases}
1,\qquad \mbox{if }\left|  \psi(\by,\overline{\bm{z}}-\bm{z}) \right| \ge \mu, \\ 
0,\qquad \mbox{if }\left|  \psi(\by,\overline{\bm{z}}-\bm{z}) \right| < \mu. 
\end{cases}
\end{align*}
Applying Lemma 
\ref{lemma:hoeffding} yields for all $\by \in \cU$,
$$
\int_D \chi(\by,\overline\bz)d\overline{\lambda} =  \P_{\overline\bz}\(\left|  \psi(\by,\overline{\bm{z}}-\bm{z}) \right| \ge \mu \) 
\le 4 \exp\({-\frac {M \mu^2}{ 8\Theta^2 s}}\).
$$
There follows 
$$
 \int_D \left( \int_\cU \chi(\by,\overline\bz) d\varrho \right) d\overline{\lambda} =  \int_\cU \left( \int_D \chi(\by,\overline\bz)d\overline{\lambda} \right) d\varrho
\le 4 \exp\({-\frac {M \mu^2}{ 8\Theta^2 s}}\), 
$$
which by Markov's inequality yields that with probability exceeding 
$2/3$, $\overline\bz$ satisfies  
\begin{align}
\label{est8}
\P_{\by}\(\left| \psi(\by,\overline\bz - \bm{z}) \right| \ge \mu \) = 
 \int_\cU \chi(\by,\overline\bz) d\varrho   \le 12 \exp\({-\frac {M \mu^2}{ 8\Theta^2 s}}\). 
\end{align}
Now, repeating the above arguments to the set $\{\by_1,\dots,\by_m\}$ with discrete uniform distribution, one can also derive that with probability of $\overline\bz$ exceeding $2/3$, there holds   
\begin{align}
\label{est9}
\left| \psi(\by_i,\overline\bz - \bm{z}) \right| \ge \mu,\mbox{ for at most }12 \exp\(-\frac {M \mu^2}{ 8\Theta^2 s}\) m\mbox{ indices }i\in \{1,\ldots,m\}. 
\end{align}
Hence, there exists a realization $\bz'$ of $\overline\bz$ in $ D$ fulfilling 
both \eqref{est8} and \eqref{est9}. Note that $\#(D) \le (4N)^M $. By defining a 
new variable $\varsigma = 12 \exp\({- \frac {M \mu^2}{ 8\Theta^2 s}}\)$ 
and eliminating $M$, we conclude the proof. 
\end{proof}

{
We observe that a sharper estimate of $\log(\#(D))$ is possible. Indeed, one 
can bound $\#(D)$ by ${4N + M \choose M}$ instead of $(4N)^M$, under 
which the assertion (ii) is replaced by 
\begin{align}
\label{est7}
\log(\#(D)) \le \dfrac{8}{\mu^2} \Theta^2 s \log\( e+ \dfrac{eN\mu^2}{\Theta^2 s \log(12/\varsigma)} \)\log\({12}/{\varsigma}\). 
\end{align}
Consequently, the RIP estimate can be improved with a slightly weaker 
logarithm factor. We will detail this point later in Remark \ref{remark:weak_log}.  
}

\subsection{Proof of Theorem \ref{note:RIP_theorem}}
\label{appen:mainproof}
\begin{proof}
We define the set of integers 
\begin{align} 
\cL =  \Z\cap \(\frac{\log(\delta)}{\log(1+\delta)}+1, 
\dfrac{\log(\Theta\sqrt{s})}{\log(1+\delta)}+1\),
\label{defineIndices}
\end{align} 
and denote by $ \underline l, \overline l$ the minimum and maximum of 
$\cL$ respectively, where $\cL$ has been chosen so that the integers $\underline l$ 
and $\overline l$ satisfy 
\begin{align} 
(1+\delta)^{\underline l -2} \leq \delta 
\quad\mbox{and}\quad
(1+\delta)^{\overline l } \geq \Theta \sqrt s. 
\label{maxmin}
\end{align} 
Let $Q := \{\by_1,\dots,\by_m\}$ be the sample set containted in $\cU$ and denote by $\varrho_{Q}$ the discrete uniform measure
associated with $Q$.
  
\textbf{Step 1:} For $0<\varsigma <1$ (exact value will be set later), we seek to construct $\tilde{\psi}$ approximating $\psi$ such that:
\begin{enumerate}
\item[(i)] For all $\bm{z} \in \mathcal{E}_s $, the following holds with probability exceeding $1 - \varsigma$ in $(\cU,\varrho)$, as well as probability exceeding $1 - \varsigma$ in $(Q,\varrho_Q)$
\begin{gather}
\begin{aligned} 
\left(1-{3\delta}/{2}\right)\tilde{\psi}(\by,\bm{z})  < & |{\psi}(\by,\bm{z})| <  
\left(1+ {3\delta}/{2}\right)\tilde{\psi}(\by,\bm{z}) , &\mbox{ if }\tilde{\psi}(\by,\bm{z}) > 0,
\\ 
&  |{\psi}(\by,\bm{z})| < {6\delta}/{5},\!  &\mbox{ if }\tilde{\psi}(\by,\bm{z}) = 0.
\end{aligned}
\label{note:eq1b}
\end{gather}
\item[(ii)] For all ${\bm z}\in \mathcal{E}_s$, there exists a pairwise disjoint family of subsets $(I^{({\bm z})}_l)_{l\in \mathcal{L}}$ of $\cU$ depending on $\bm{z}$ such that 
\begin{align}
\label{note:eq2}
\tilde{\psi}(\cdot,{\bm z}) = \sum_{l\in\mathcal{L}} (1+\delta)^{l} \chi_{I^{({\bm z})}_l}. 
\end{align}
\item[(iii)] For every $l \in \mathcal{L}$, $(I_l^{({\bm z})})_{{\bm z}\in \mathcal{E}_s}$ belongs to a finite class $F_l$ of subsets of $\cU$ satisfying 
\begin{align}
\label{note:eq2b}
\log(\#F_l) \le \dfrac{32}{\delta^3 (1 + \delta)^{2l-2} } \Theta^2 s \log\left( 4N \right)\log\left(\frac{12  \log( \delta^{-1} \Theta \sqrt{s})}{\varsigma\log(1+\delta)}\right). 
\end{align}
\end{enumerate}

First, for $l\in \mathcal{L}$, let $D_l$ be a finite subset of $\mathbb{C}^N$ determined as in Lemma \ref{note:lemma2} with $\mu = \frac {{\delta}(1+\delta)^{l-1}}{2}$ and $0 < \varsigma'< 1$ (to be set accordingly to meet our needs). We have 
\begin{align}
\log(\# D_l ) & \le  \dfrac{32}{\delta^2 (1 + \delta)^{2l-2} } \Theta^2 s \log\left( 4N \right)\log\left({12}/{\varsigma}'\right).
\label{note:cardD}
\end{align}
%
For a fixed ${\bm z}\in \mathcal{E}_s$, there exists $\bm{z}_l \in D_l$ and a measurable set $\cU_l \subset \cU$ with $\varrho(\cU_l) \ge 1-\varsigma'$ such that 
\begin{align*}
|\psi(\by, \bm{z} - \bm{z}_l)| \le  \frac{{\delta}(1+\delta)^{l-1}}{2},\, \forall \by \in \cU_l,
\end{align*} 
and $\by_i$'s are contained in $\cU_l$ for at least $(1 - \varsigma') m${ indices }$i \in \{1,\ldots,m\} $. 

We construct a pairwise disjoint family of subsets $(I_l^{({\bm z})})_l$ and mapping $\tilde{\psi}(\cdot,\bm{z}):\mathcal{U} \to \mathbb{R}$ which depend on $\bz$ and $Q$, inductively for the integers 
$\overline l > \dots > \underline l$ according to: 
\begin{gather}
\begin{aligned}
 I'^{({\bm z})}_l  = \{\by \in \cU: (1+\delta)^{l-1} &< |\psi(\by,{\bm z}_l)| < (1+\delta)^{l+1}\}, 
\\
 I^{({\bm z})}_l  &= I'^{({\bm z})}_l \,  {\setminus} \, \bigcup_{r>l} I'^{({\bm z})}_{r}, \text{ and }
\\
 \tilde{\psi}(\cdot,{\bm z}) &= \sum_{l\in \mathcal{L}} (1+\delta)^{l} \chi_{I^{({\bm z})}_l}. 
\end{aligned}
\label{constructionPsiZ}
\end{gather}

In the following, we prove that $\tilde{\psi}$ satisfies \eqref{note:eq1b}--\eqref{note:eq2b}. First, consider $\by\in \bigcap_{l\in \cL} \cU_l$. If $\by\in I^{({\bm z})}_l$ for some $l\in \cL$, then 
\begin{align*}
\tilde{\psi}(\by,{\bm z}) = (1+\delta)^{l} > 0 \mbox{ and }(1+\delta)^{l-1} < |\psi(\by,\bm{z}_l)| < (1+\delta)^{l+1}.
\end{align*}
Since $| |\psi(\by,{\bm z})| - |\psi(\by,{\bm z}_l)| | \le |\psi(\by,\bm{z}) - \psi(\by,\bm{z}_l)| \le {\delta}(1+\delta)^{l-1}/2$, we have 
\begin{align*}
|\psi(\by,{\bm z})| & < (1+\delta)^{l+1} + \frac{\delta}{2}(1+\delta)^{l-1}   < \left(1+\frac{3}{2}\delta\right)\tilde{\psi}(\by,{\bm z}), 
\\
\text{ and }|\psi(\by,{\bm z})| & > (1+\delta)^{l-1} - \frac{\delta}{2}(1+\delta)^{l-1} > \left(1-\frac{3\delta}{2}\right)\tilde{\psi}(\by,{\bm z}). 
\end{align*}

If $ \by \notin \bigcup_{l\in \cL} I_l^{({\bm z})}$, then $\tilde{\psi}(\by,\bm{z}) = 0$ and for every $l\in \mathcal{L}$, 
\begin{align*}
  |\psi(\by,\bm{z}_l)| \notin ((1+\delta)^{l-1},(1+\delta)^{l+1}). 
\end{align*}
We notice that $| |\psi(\by,{\bm z})| - |\psi(\by,{\bm z}_l)| | < {\delta}(1+\delta)^{l-1} /2 $, there follows
\begin{align*}
|\psi(\by,{\bm z})| \notin \bigcup_{l\in \mathcal{L}}\left((1+\frac{\delta}{2})(1+\delta)^{l-1} , (1+\frac{3\delta}{2} + \delta^2)(1+\delta)^{l-1}  \right).
\end{align*}
Observe that $(1+\frac{3\delta}{2} + \delta^2)(1+\delta)^{l-1}  > (1+\frac{\delta}{2})(1+\delta)^{l}$, the previous intervals intersect for any two consecutive values of $l$. We infer
\begin{align*}
|\psi(\by,{\bm z})| \le (1+\frac{\delta}{2})(1+\delta)^{\underline{l}-1},\ \mbox{ or }\ |\psi(\by,{\bm z})| \ge (1+\frac{3\delta}{2} + \delta^2)(1+\delta)^{\overline{l}-1}. 
\end{align*}
In view of the identities in \eqref{maxmin} and 
$\|\psi\|_{L^\infty}\leq \Theta \sqrt s$, 
the second inequality cannot occur. As for the first, 
it implies by \eqref{maxmin} and assumption $\delta<1/13$ that $|\psi(\by,\bz)| \leq \delta(1+\delta/2)(1+\delta) < 6\delta/5$.


Next, consider $\by\notin \bigcap_{l\in \cL} \cU_l$. Condition \eqref{note:eq1b} is not guaranteed in this case. However, these only hold with probability not exceeding
\begin{equation}
\begin{aligned}
 \varrho\biggl(\cU \setminus  \bigcap_{l\in \mathcal{L}} \cU_l \biggl) \le \sum_{l\in \cL}\varrho(\cU\setminus  \cU_l)  \le \varsigma' (\#\mathcal{L}) \le  \frac{ \log( \delta^{-1} \Theta \sqrt{s})}{\log(1+\delta)} \varsigma' = \varsigma,
 \\ \text{and }
  \# \biggl\{ i :  \by_i\notin  \bigcap_{l\in \cL} \cU_l \biggl\}  \le \varsigma' m (\#\mathcal{L}) \le  \frac{ \log( \delta^{-1} \Theta \sqrt{s})}{\log(1+\delta)} \varsigma' m = \varsigma m,
\end{aligned}
 \label{note:badset}
\end{equation}
when setting $\varsigma' =  \dfrac{\log(1+\delta)}{ \log( \delta^{-1} \Theta \sqrt{s})} \varsigma$. 

In summary, we have proved that for all $\bz\in \cE_s$, the following is satisfied 
\begin{eqnarray} 
&\left(1-{3\delta}/{2}\right)\tilde{\psi}(\by,\bm{z})  < |{\psi}(\by,\bm{z})| <  \left(1+ {3\delta}/{2}\right)\tilde{\psi}(\by,\bm{z}) , & \mbox{ for }\by \in I, \label{note:partition1}
\\
& 0\le |{\psi}(\by,\bm{z})| < 6\delta/5, \ \ \mbox{ and }\ \ \tilde{\psi}(\by,\bm{z}) = 0, & \mbox{ for }\by \in \widehat{I},  \label{note:partition2}
\\
&  \varrho(\cU' )\le \varsigma, \label{note:partition3}
\end{eqnarray}
where the three sets 
 $I := \left(\bigcap_{l\in \cL} \cU_l \right) \bigcap \left(\bigcup_{l\in \cL} I_l^{({\bm z})}\right),\ \ \widehat{I} := \left(\bigcap_{l\in \cL} \cU_l \right) \setminus \left(\bigcup_{l\in \cL} I_l^{({\bm z})}\right),$ 
and 
$\cU' := \cU\setminus \left(\bigcap_{l\in \cL} \cU_l\right)$
define a partition of $\cU$, and depend on $\bz$ and $Q$.

It remains to verify \eqref{note:eq2b}. For any $l\in \mathcal{L}$, $\#\{I'^{({\bm z})}_l\, |\, {\bm z}\in \mathcal{E}_s\} \le \#D_l$ and 
$\# F_l \le \prod_{r \ge l} \# D_{r}$. From \eqref{note:cardD}, we see that 
\begin{align*}
\log(\# F_l)  &\le \sum_{r\ge l} \log(\# D_{r}) \le \dfrac{32}{\delta^3 (1 + \delta)^{2l-2} } \Theta^2 s \log\left( 4N \right)\log\left({12}/{\varsigma}'\right).
\end{align*}

\textbf{Step 2:} Derive essential estimates of $\|\bz\|_2$ and $\|\bm{Az}\|_2$ in terms of $\tilde{\psi}(\cdot,\bz)$. First, given $\bz \in \cE_s$, we observe that
\begin{align}
1=\|\bz\|^2_2 
= \int_{\cU}  |\psi(\by,\bz)|^2 d\varrho,\quad\text{and }
\|\bA\bz\|^2_2  
=\sum_{i=1}^m \frac{|\psi(\by_i,\bz)|^2}m
= \int_{Q}  |\psi(\by,\bz)|^2 d\varrho_Q.
\end{align}

It is easy to check that if $\delta\leq1/13$, one has for real numbers $a,b > 0$, 
$(1-\frac {3\delta}2) a < b < (1+\frac {3\delta}2) a$ implies 
$(1-3\delta)a^2 < b^2 < (1+4\delta) a^2$, which also implies 
$(1-4\delta)b^2 < a^2 < (1+4\delta) b^2$, so that 
$ |b^2-a^2| < 4\delta \min(a^2,b^2)$. Therefore, from \eqref{note:partition1}  we get that
\begin{align*}
\Big|
|{\psi}(\by,\bz)|^2 - |\tilde\psi(\by,\bz)|^2
\Big|
&<  
4\delta|\psi(\by,\bz)|^2,\quad\quad\text{for } \by \in I,\\
\text{ and }
\Big|
|{\psi}(\by,\bz)|^2 - |\tilde\psi(\by,\bz)|^2
\Big|
&<  
4\delta |\tilde\psi(\by,\bz)|^2,\quad\quad\text{for } \by \in I.
\end{align*}
This, combined with \eqref{note:partition2} and \eqref{note:partition3} and the fact that $|\psi(\cdot,\bz)|^2$ and $|\tilde\psi(\cdot,\bz)|^2$ are uniformly bounded in $\cU$ by $\Theta^2 s$ and by $(1+\delta)^2 \Theta^2 s \le 2 \Theta^2 s$, implies
\begin{align}
 &\left| \| \bz\|_2^2 - \int_{\cU} |\tilde\psi(\by,\bz)|^2 d\varrho \right| \le    \int_{I}\left | |\psi(\by,\bz)|^2 -  |\tilde\psi(\by,\bz)|^2 \right| d\varrho \notag
\\
  & \qquad\qquad +  \int_{\widehat{I}} \left | |\psi(\by,\bz)|^2 -  |\tilde\psi(\by,\bz)|^2 \right| d\varrho +  \int_{\cU'}\left | |\psi(\by,\bz)|^2 -  |\tilde\psi(\by,\bz)|^2 \right| d\varrho \notag
  \\
   &\qquad \le\, 4\delta  \int_{I}|\psi(\by,\bz)|^2  d\varrho \, + \frac{36\delta^2}{25} \varrho({\widehat{I}}) + 2\Theta^2 s \varsigma.
     \label{note:comp0}
\end{align}
By noticing that $\varrho(\widehat{I})\leq 1$ and setting $ \varsigma = \frac{\delta}{6\Theta^2 s}$,
we infer
\begin{align}
\label{note:comp1}
\left| \| \bz\|_2^2 - \int_{\cU} |\tilde\psi(\by,\bz)|^2 d\varrho \right| < 4\delta + \frac{\delta}{6} + \frac{\delta}{3} = \frac{9\delta}{2}. 
\end{align}
Repeating the above argument for the probability space $(Q,\varrho_Q)$ with notice that $\varrho_Q(\cU'\cap Q) \le \varsigma$ yields
\begin{align}
\label{note:comp2}
 &\left| \| \bA \bz\|_2^2 - \int_{Q} |\tilde\psi(\by,\bz)|^2 d\varrho_Q \right| 
  \le\,  4\delta  \int_{Q}|\tilde{\psi}(\by,\bz)|^2  d\varrho_Q \, + \frac{\delta}{2}. 
\end{align}
From \eqref{note:comp1} and \eqref{note:comp2}, we obtain
\begin{align}
\notag
& \Big|\|\bA\bz\|_2^2 - \|\bz\|_2^2\Big|
 \leq 4\delta  \int_{\cU} |\tilde\psi(\by,\bz)|^2 d\varrho 
+5\delta
\\
&\qquad\qquad    +\, (1+4\delta)
\left|
\int_{Q} |\tilde \psi(\by,\bz)|^2 d\varrho_Q
-
\int_{\cU} |\tilde \psi(\by,\bz)|^2 d\varrho
\right| \label{note:comp3}
\\ 
&\qquad \le  4\delta\(1 + \frac{9\delta}{2}\) + 5\delta + (1+4\delta)
\left|
\int_{Q} |\tilde \psi(\by,\bz)|^2 d\varrho_Q
-
\int_{\cU} |\tilde \psi(\by,\bz)|^2 d\varrho
\right|. \notag
\end{align}

\textbf{Step 3:} We derive an upper bound of $\Big|\|\bA\bz\|_2^2 - \|\bz\|_2^2\Big|$ via \eqref{note:comp3}, by employing a basic tail estimate (Lemma \ref{lemma:tailbound}) and the union bound. From the definition of $\tilde \psi$, we have that 
\begin{align}
\label{note:comp4}
\left| \int_{Q} |\tilde \psi(\by,\bz)|^2 d\varrho_Q
-
\int_{\cU} |\tilde \psi(\by,\bz)|^2 d\varrho \right|
\le
\sum_{l\in\cL}  
(1+\delta)^{2l}\!
\left|
\frac{\#(Q\cap I_l^{(\bz)})}m
-
\varrho(I^{(\bz)}_l)
\right|.
\end{align}
Let $(\kappa_l)_{l\in\mathcal{L}}$ be a sequence of positive numbers. Applying Lemma \ref{lemma:tailbound}, for any set $\Delta$ in the class $F_l$, with probability of $Q$ exceeding $1-\exp\left(- \frac{m \kappa_l \delta}{16e}\right)$, there holds
\begin{align}
\left | \frac{ \#( Q\cap \Delta)}{m} - \varrho (\Delta) \right| \le \delta \varrho (\Delta)  + \kappa_l. \label{note:eq6}
\end{align}
By the union bound, with probability exceeding $1 - \sum_{l\in \mathcal{L}} \exp\left(- \frac{m \kappa_l \delta}{16e}\right) (\# F_l) $, the previous inequality holds uniformly for all sets $\Delta \in \cup_{l\in \cL} F_l$. Therefore, with probability exceeding $1 - \sum_{l\in \mathcal{L}} \exp\left(- \frac{m \kappa_l \delta}{16e}\right) (\# F_l) $, we can apply \eqref{note:eq6} with $\Delta = I^{(\bz)}_l$ ($l\in \cL$) to the sum in \eqref{note:comp4} and combine with \eqref{note:comp3} to infer that for all $\bz\in \cE_s$, 
\begin{align*}
 \Big|\|\bA\bz\|_2^2 \! -\! \|\bz\|_2^2\Big|
 & \leq 
  4\delta\(1\! +\!  \frac{9\delta}{2}\) + 5\delta + \delta (1\!+\!4\delta) 
  \int_{\cU} |\tilde\psi(\by,\bz)|^2 d\varrho  + (1\!+\!4\delta)
\sum_{l\in\cL}  
(1+\delta)^{2l}\!
\kappa_l
\\
& \leq 
  4\delta\(1\! +\!  \frac{9\delta}{2}\) + 5\delta + \delta (1\!+\!4\delta) 
 \(1\! +\!  \frac{9\delta}{2}\)  + (1\!+\!4\delta)
\sum_{l\in\cL}  
(1+\delta)^{2l}\!
\kappa_l 
\\
& \leq 
  12\delta + \delta/3 + (1\!+\!4\delta)
\sum_{l\in\cL}  
(1+\delta)^{2l}\!
\kappa_l , 
\end{align*}
for the last inequality we have used $\delta < 1/13$. 

Finally, in order to obtain Theorem \ref{note:RIP_theorem}, we 
need to assign appropriate values for $\kappa_l$ and derive 
conditions on $m$ such that 
\begin{align*}
\sum_{l\in\cL}  
(1+\delta)^{2l}\!
\kappa_l  \le \delta/2,
\quad\quad\mbox{and}\quad\quad
\sum_{l\in \cL}\exp\(- \frac{m\kappa_l \delta }{16 e} + \log(\# F_l)\)\le \gamma.
\end{align*}
The two inequalities can be fulfilled if for example the numbers 
$\kappa_l$ and the integer $m$ are chosen as follows  
\begin{align*}
\kappa_l :=\frac{\delta/2}{(\#\cL) (1+\delta)^{2l}},\quad\quad
- \frac{m\kappa_l \delta }{16 e} + \log(\# F_l) 
\le \log\left(\frac{\gamma}{\#\cL}\right),\quad \quad  l \in \cL.
\end{align*}
This implies that 
\begin{align*}
m & \ge  32 e~(\#\cL) \frac{(1+\delta)^{2l}}{\delta^2} \left[ \log(\# F_l)  
+ \log\left(\frac{\#\cL}{\gamma}\right)\right],\ \quad\quad l\in \cL. 
\end{align*}
Observe that since $\varsigma =\dfrac{\delta}{6 \Theta^{2}s}$, we have in view 
of \eqref{note:eq2b} that

\begin{align*}
&32 e\,  (\#\cL) \frac{(1+\delta)^{2l}}{\delta^2} \log(\# F_l)  
\\
&\qquad \le 2^{10} e\,  \frac{(1+\delta)^{2}}{\delta^5\log(1+\delta)} \Theta^2 s \log\left( 4N \right)\log\left( \frac{\Theta \sqrt{s}}{\delta}\right)\log\left(\frac{72\Theta^2 s}{\delta} \cdot \frac{  \log( \Theta \sqrt{s}/\delta)}{\log(1+\delta)}\right)\notag
 \\
 &\qquad <     2^{10} e ~ \frac{\Theta^2 s}{\delta^6} \log(4N) 
 \log\( \frac{\Theta^2 {s}}{\delta^2}\) \log\( 40 \frac{ \Theta^{2} s}{\delta^2}  
 \log\( \frac{\Theta^2 {s}}{\delta^2}\) \), \notag
\\
&32 e\,  (\#\cL) \frac{(1+\delta)^{2l}}{\delta^2}  \log\left(\frac{\#\cL}{\gamma}\right)
\\
 &\qquad  \le 32 e {\frac{ (1+\delta)^2}{ {\delta^{2}} 
 \log(1+\delta)} } \Theta^2 s   \log\( \frac{\Theta \sqrt{s}}{\delta}\) 
 \log\(\frac{1}{\gamma } \cdot \frac{\log(  
 \Theta \sqrt{s}/\delta )}{\log(1+\delta)}\) \notag
 \\
 & \qquad  < 2^5 e \frac{\Theta^2 s}{\delta^3} 
 \log\( \frac{\Theta^2 {s}}{\delta^2}\)  \log\(\frac{1}{\gamma \delta } 
 \log\( \frac{\Theta^2 {s}}{\delta^2}\)  \). \notag
\end{align*}
Here, we employed estimates $(1+\delta)^2 < (14/13)^2$ and $\log^{-1}(1+\delta) < 1.1/\delta$, obtained from the small condition of $\delta$. 
Combining the two estimates and 
$(a+b)\leq 2\max(a,b)$ shows that $m$ as in 
Theorem \ref{note:RIP_theorem} is suitable. 
\end{proof}

We conclude this subsection with two remarks detailing some slight technical improvements of Theorem \ref{note:RIP_theorem}.

\begin{remark}
We can give {a sharper approximation of $\psi$} by refining 
the map $\tilde \psi(\cdot,\bz)$. For example,
given an integer $k\geq1$, we define $\cL$ as 
\begin{align} 
\cL =  \frac\Z k\cap \(\frac{\log(\delta)}{\log(1+\delta)}+\frac1k, 
\dfrac{\log(\Theta\sqrt{s})}{\log(1+\delta)}+\frac1k\),
\label{defineIndices2}
\end{align} 
and construct the domains $D_l$ using $\mu= \frac {\delta (1+\delta)^{l-1/k}}{2k} $. We replace the domains $I_l'^{(\bz)}$ in \eqref{constructionPsiZ} by
\begin{align*}
I_l'^{(\bz)}:= 
\{\by \in \cU: 
(1+\delta)^{l- 1/k} 
< |\psi(\by,\bz_l)| < 
(1+\delta)^{l+1/k}\}.
\end{align*}
Using the elementary 
inequalities $(1+\delta)^{1/k}\leq 1+\delta/k$ and $(1+\delta)^{1/k} - \frac{\delta}{{2k}(1+\delta)^{1/k}}\geq 1+\frac \delta{2k}$ and assuming
$\delta/k \leq 1/13$, we verify that \eqref{note:partition1} can be improved as 
\begin{gather}
\begin{aligned} 
\left(1-\frac{3\delta}{2k}\right)\tilde{\psi}(\by,\bm{z})  < |{\psi}(\by,\bm{z})| <  \left(1+ \frac{3\delta}{2k}\right)\tilde{\psi}(\by,\bm{z}) , & \mbox{ for }\by \in I, 
\end{aligned}
\label{note:partition4}
\end{gather}
while \eqref{note:partition2} and \eqref{note:partition3} are unchanged. 
An inspection of the proof shows that this yields \eqref{note:comp0} with $4\delta/k$ instead of $4\delta$ and with 
$36\delta^2/25$ and $2\Theta^2 s \varsigma$ {unchanged}. We 
mention however that the cardinality $\#(\cL)$ and the bound 
in \eqref{note:eq2b} on $\log(\#(F_l))$ gets roughly multiplied 
by $k$ and $k^3$ respectively.
\end{remark}



\begin{remark}
\label{remark:weak_log}
If \eqref{est7} is applied to bound $\log(\#D_l)$, we can obtain 
\begin{align*}
\log(\# D_l ) & \le  \dfrac{32}{\delta^2 (1 + \delta)^{2l-2} } \Theta^2 s  \log\( e+ \dfrac{eN\delta^2 (1 + \delta)^{2l-2}}{4 \Theta^2 s \log(12/\varsigma')} \) \log(12 / \varsigma' )
\\
& \le \dfrac{32}{\delta^2 (1 + \delta)^{2l-2} } \Theta^2 s  \log\( e+ \dfrac{eN\delta^2 }{4 \log(12/\varsigma')} \) \log(12 / \varsigma' ),\quad\quad  l\in \cL,
\end{align*}
where $ \varsigma' = \frac{\delta\log(1+\delta)}{ 6 \Theta^2 s \log( \delta^{-1} \Theta \sqrt{s})}$, instead of \eqref{note:cardD}. Subsequently, the term $\log(4N)$ in sample complexity \eqref{num_sample0} and \eqref{num_sample} can be replaced by $\log\( e+ \dfrac{eN\delta^2 }{4 \log\(36 \frac{ \Theta^2 s}{\delta^2} { \log\( \frac{ \Theta^2 {s}}{\delta^2}\)} \)} \) $.
\end{remark}

\subsection{Proof of Theorem \ref{theorem:lower-RIP}} 
\label{app:lowerRIP}
\begin{proof}
The proof of Theorem \ref{theorem:lower-RIP} follows closely that of 
Theorem \ref{note:RIP_theorem} with one critical change: Instead of $\cE_s$, we only approximate $\|\bm {Az}\|_2$ on {the set $\cE^{\ell}_s$ 
defined as 
\begin{align}
\cE^{\ell}_s  = \{\bz \in \mathbb{C}^N: \|\bz\|_2 = 1\text{ and }K(\supp(\bm{z})) \le K(s)\}. 
\end{align}
}
We have that  $\cE^{\ell}_s \subset  \text{conv}({\mathcal{P}}^\ell)  $, where {${\mathcal{P}}^{\ell} = \{\pm \frac{\bm{e}_j}{\omega_j}\sqrt{2K(s)}, \pm i\frac{\bm{e}_j}{\omega_j}\sqrt{2K(s)}\}_{1\leq j\leq N}$}, with $\omega_j = \|\Psi_j\|_{L^{\infty}}$
and $(\bm{e}_j)$ being canonical unit vectors in $\C^N$. On the other hand, 
\begin{align}
\label{Cauchy}
|\psi(\by,\bz)| \le \sqrt{K(s)}, \quad\quad
\forall \bz \in \cE^{\ell}_s, \, \by\in \cU. 
\end{align}
We thus can derive an extended covering number result for $\cE^{\ell}_s$ (similar to Lemma \ref{note:lemma2} for $\cE_s$), then replace the bound $\Theta \sqrt{s}$ by $\sqrt{K(s)}$ throughout the previous proofs,
resulting in Theorem \ref{theorem:lower-RIP}. 
\end{proof}

\bibliographystyle{amsplain}
\bibliography{CS_final}  

\end{document}